\documentclass[11pt]{amsart}

\usepackage{amsthm}

\usepackage{amsmath,latexsym}
\usepackage{amssymb}
\usepackage{geometry}
\usepackage{amsfonts}
\usepackage{hyperref}
\usepackage[skip=4pt]{subcaption}

\usepackage{pgfplots}

\usetikzlibrary{arrows}

\usepackage{multicol}
\usepackage{multirow}
\usepackage{array}
\usepackage{float}
\usepackage{caption}
\newcommand{\dsum}{\displaystyle\sum}

\usetikzlibrary{arrows,decorations.markings}
\usetikzlibrary{shapes,snakes}

\usepackage{pgfplots}
\pgfplotsset{compat=newest}
\usetikzlibrary{decorations.markings}
\usepackage{color}
\newtheorem{lemma}{Lemma}
\newtheorem{example}{Example}

\let\origmaketitle\maketitle
\def\maketitle{
  \begingroup
  \def\uppercasenonmath##1{} % this disables uppercasing title
  \let\MakeUppercase\relax % this disables uppercasing authors
  \origmaketitle
  \endgroup
}

\begin{document}
\title[Upgrading nodes in tree-shaped hub location]{\large Upgrading nodes in tree-shaped hub location}

\author[V\'ictor Blanco \MakeLowercase{and} Alfredo Mar\'in]{{\large V\'ictor Blanco$^\dagger$ and Alfredo Mar\'in$^\ddagger$}\medskip\\
$^\dagger$Dpt. Quant. Methods for Economics \& Business, Universidad de Granada\\
$^\ddagger$Dpt.\ Statistics \& Operational Research, Universidad de Murcia
}

%\author{V\'ictor Blanco}
\address{Dpt. Quant. Methods for Economics \& Business, Universidad de Granada}
\email{vblanco@ugr.es}

%\author{Alfredo Mar\'in}
\address{Dpt.\ Statistics \& Operational Research, Universidad de Murcia}
\email{amarin@um.es}

\date{\today}

\begin{abstract}
In this paper, we introduce the Tree of Hubs Location Problem with Upgrading, a mixture of the Tree of Hubs Location Problem, presented in \cite{contreras}, and the Minimum Cost Spanning Tree Problem with Upgraded nodes, 
studied for the first time in \cite{krumke}. In addition to locate the hubs, to determine the tree that connects the hubs and to allocate non-hub nodes to hubs, a decision has to be made about which of the hubs will be upgraded, taking into account that the total number of upgraded nodes is given. We present two different Mixed Integer Linear Programming formulations for the problem, tighten the formulations and generate several families of valid
inequalities for them. A computational study is presented showing the improvements attained with the strengthening of the formulations and comparing them.
\end{abstract}

\maketitle

\section{Introduction}

The main goal of Discrete Facility Location is to decide which set of facilities, among a finite set of potential 
sites, must be opened in order to optimize a given objective function which represents traveling costs of a set of given customers and/or set-up costs. Reaching such a goal is closely related with  Mixed Integer Linear Programming (MILP) models. According to \cite{laporte},  the first MILP formulation for a discrete location problem, the \textit{Uncapacitated
Facility Location Problem}, was proposed by Balinski \cite{balinski}. In the following few years, also MILP formulations appeared for the discrete $p$-Median Problem~\cite{swain} and for a covering-location problem~\cite{toregas}. Since then, Discrete Facility Location has gained importance over the years as a fundamental part of both, Integer Programming and Location Science.

Among the wide family of Discrete Facility Location problems, hub location problems have taken on significance in the last decades. Hub Location Problems (HLP) emerge as an important design 
tool in transportation systems, where several sites must send and receive some goods through some transshipment points, {\em the hubs}. In the hubs, the product is collected and distributed, so reducing transportation costs. In this sense, a common characteristic of HLP is that the transportation cost  is reduced if the product is routed between two hubs. The interest of hub location is still increasing, and one can read a number of surveys on the matter published over the last years (see \cite{alumur,libro,ck,chlp,farahani}, among many others).

Hub location problems are usually classified according to the way of allocating non-hub customers (also called {\em spokes}) to hub nodes. On the one hand, when the product is allowed to be routed using all the available hubs, which can be different for one origin 
depending on the destination to be reached, the problem is said of {\em multiple--allocation}. The first formulation for this family of problems was presented in \cite{cejor94}. On the other hand, 
in {\em single--allocation} models, the customers are only allowed to send and receive the service from a unique 
hub, i.e., each origin/destination is allocated to exactly one hub. 
Our study fits within this category, to which the pioneering works \cite{aykin95}, \cite{ekejor98}, \cite{okbsksk}, \cite{okelly} and \cite{sksko} belong to.

A few papers also incorporate to the HLP the design of an inter-hub network in which not all hubs are necessarily neighbors (see \cite{campbelletal05a,campbelletal05b}). In this paper, we analyze the situation that occurs when the infrastructure costs (those costs associated 
with the direct links between hubs) are high compared to the rest of the costs in the model (e.g., transportation costs). In this case, a minimally nested and connected structure
is desirable: (i) a tree of hubs ({\em small tree}),
composed by the hubs and their links; and (ii) a tree of spokes and hubs ({\em large tree}) where the non-hub nodes 
and their single connections to the corresponding hub are linked to the small tree. This model, introduced by Contreras, Fern\'andez and Mar\'in  \cite{cfm2,contreras}, was named the \textit{Tree of Hubs Location Problem} (THLP). Since then, the THLP has received growing attention (see, e.g., \cite{alukar}, \cite{camikeca}, \cite{conisa}, \cite{ljubic}, \cite{miluca} and \cite{sedehzadeh}). Due to the high applicability and difficulty of solving the problem, both heuristic and exact efficient approaches have been developed for the problem. In particular, in \cite{pessoa}, the authors developed a genetic algorithm 
for the THLP, and in \cite{sa}, a Benders-based branch-and-cut strategy is designed to solve the THLP.

On the other hand, one can find in the literature extensions of classical combinatorial optimization problems in which some of the nodes are allowed to be upgraded. Upgrading implies a reduction of the cost of traversing edges connecting those special nodes (see for instance, \cite{alvarez,dilkina,paik}). Thus, the cost structure in these problems depends on the upgrading type of the extreme nodes of the edges (no upgraded extremes, only one upgraded extreme or two upgraded extremes), and then, it is part of the decision. In particular, in the recent publication by \'Alvarez-Miranda and Sinnl~\cite{alvarez}, the authors introduce the Minimum Spanning Tree (MST)  
problem with Upgrading (MSTU). In that problem, a minimum cost spanning tree has to be built on a graph, but the cost associated to the edges of the graph is reduced by upgrading nodes in two ways: (1) when one (but not both) of the extremes is upgraded; or (2) when both of the extremes are upgraded. In this way,  if a node is
upgraded, all edges of the small tree containing it benefit from the decision.  This extension 
transforms the polynomial-time solvable MST problem into an NP-hard problem~\cite{alvarez}. This problem is of interest in the telecommunications field, where the 
installation of better infrastructures in a given node improves the quality of the
transmissions to and from all the neighbours of the node.

We are interested in studying the effect of upgrading nodes in the THLP. Thus 
we introduce in this work the THLP with  upgrading (THLPU), a mixture of the 
MSTPU and the THLP where, in addition to locate the hubs, determine a tree
structure for them, and allocate every non-hub node to a hub, one has to decide which hubs are upgraded, assuming that the number of nodes that can be upgraded is given. We propose and analyze two different formulations for the problem. These formulations 
are compared in terms of the lower bound they produce and the computational 
effort required by a standard optimization solver to provide an optimal solution. The 
formulations share a kernel based on the THLP structure and differ in the way they 
account for the reduced costs that the upgrading of nodes produce on the edges 
of the small tree.

The paper is organized as follows. In Section \ref{sec:2} we introduce the problem and fix the notation for the rest of the paper. In Section \ref{sec:3} we present a first Mixed Integer Linear Programming formulation for the THLPU based on the ideas presented in \cite{contreras} but introducing a new family of variables to represent the new cost structure.  We also provide sets of valid inequalities that allow us to tighten the model, and report the results of a preliminary battery of computational experiments to draw some conclusions on the weakness of this model. Section \ref{sec:4} is devoted to present an alternative MILP formulation for the problem by disaggregating the variables representing the flow circulating through the hubs and also we extend the family of valid inequalities to the disaggregated model as well as the separation strategy. In Section \ref{sec:6} we present the results of our computational experiments for the disaggregated formulation, and we compare it with the first (non-disaggregated) formulation. Finally, in Section \ref{sec:7} we draw some conclusions and further research on the topic.

\section{Tree of hubs with upgraded nodes location}\label{sec:2}

In this section we introduce the Tree of Hubs Location Problem with Upgrading (THLPU) and fix the notation for the rest of the sections. 

Given a set of customers and a matrix of flows between each pair of customers, the goal of THLPU is three-fold:
\begin{enumerate}
\item Locate a given number of hubs among the customers and allocate each of the remainder customers to a single hub (\textit{Hub Location}),
\item adequately connect the hubs with a tree structure (\textit{Tree of Hubs}), and
\item upgrade a pre-specified number of hubs, provided that a reduction on the cost is performed with the upgrading (\textit{Upgrading Hubs}),
\end{enumerate}
by minimizing the overall sum of the transportation costs.

In order to formulate the problem, we will use the following list of parameters, which are widely used in single-allocation hub location models:
\begin{itemize}
\item $N=\{ 1,\ldots ,n\}$: the set of origins and destinations of the flow, which acts, without loss of generality, as the set of potential hubs.
\item $G=(N,E)$: the undirected, connected graph with set of edges $E$ through which the  flow must be sent. In this paper, we use the notation $\{k,m\}$ to identify  an edge and $(k,m)$ to identify an arc (directed edge).
\item $w_{ij}\geq0$: the amount of flow sent from origin $i$ to destination $j$, 
for all $i, j \in N$.
\item $p\in \{ 2,\ldots ,n-1\}$: the number of hubs to be located.
\item $q\in \{ 1,\ldots ,p-1\}$: the number of hubs to be upgraded.
\end{itemize}

In what follows, we describe a cost structure inspired in the literature on network upgrading problems, that has as a 
particular case the cost structure of the classical discrete hub location problems. For each $\{i,j\}\in E$, we introduce the following unit cost associated to the flow between $i$ and $j$:
\begin{itemize}
\item $d_{ij}\ge 0$ if one of the extremes is a hub and the other is a spoke or $i=j$ (in which case, we assume that $d_{ii}=0$).
\item $c_{ij}\in [0,d_{ij}]$ if both $i$ and $j$ are hubs and none of them has been upgraded.
\item $c'_{ij}\in [0,c_{ij}]$ if both $i$ and $j$ are hubs and exactly one ($i$ or $j$) has been upgraded.
\item $c''_{ij}\in [0,c'_{ij}]$ if both $i$ and $j$ are upgraded hubs.
\end{itemize}

With the above notation, the cost of $x$ units of flow traversing an edge $\{i,j\}\in E$ will be:
$$
\left\{\begin{array}{cl}
x\cdot d_{ij} & \mbox{if only one of the extremes of $\{i,j\}$ is a hub,}\\
x\cdot c_{ij} & \mbox{if both $i$ and $j$ are hubs, but none of them are upgraded,}\\
x\cdot c'_{ij} & \mbox{if both $i$ and $j$ are hubs, and only one of them is upgraded,}\\
x\cdot c''_{ij} & \mbox{if both $i$ and $j$ are upgraded hubs.}\\
\end{array}\right.
$$

Observe also that the above settings for the problem under study can be easily adapted to deal with the case of many
variants considered in the literature, as: (i) separated sets of origins and potential hubs; (ii)  fixed costs associated to the opening of hubs and/or the link of these hubs; or (iii) different discount factors associated with the costs of the first and last edges in each route and the costs between all kinds of hubs.

Costs $d_{ij}$, $c_{ij}$, $c'_{ij}$ and $c''_{ij}$ are considered non negative, but 
no other condition is required, e.g., satisfaction of the triangle inequality or symmetry. In most cases, the simplest unit costs, $d_{ij}$, are assumed to be the Euclidean distances between the pairs of nodes. In \cite{paik}, the authors consider a particular structure of the above costs, 
based on the use of reduction factors over the $d$-costs, and that can be adapted to the general settings. Given $\alpha, \rho, \gamma \in [0,1]$, with $\alpha \geq \rho \geq \gamma$, one may define $c_{ij}= \alpha d_{ij}$, $c'_{ij}= \rho d_{ij}$ and $c''_{ij}= \gamma d_{ij}$, for all $\{i,j\}\in E$. With this cost structure, the available instances for $p$-hub location can be easily adapted for the THLPU. Also, observe that if $c_{ij}=c'_{ij}=c''_{ij}$ the  problem becomes the THLP, being the THLPU a generalization of such a problem.

Additional parameters used to simplify the forthcoming formulations are:

\begin{itemize}
\item $O_i=\dsum_{j: \{i,j\}\in E} w_{ij}$: the total amount of flow sent from origin $i\in N$.
\item $D_i=\dsum_{j: \{i,j\}\in E}  w_{ji}$: the total amount of flow sent to destination $i\in N$.
\item $O_{ikm}=\left\{ \begin{array}{cc} 
O_i-w_{ii}-\min \{ w_{ik},w_{im}\} & \hbox{ if } i\neq k \\
O_i-w_{ii} & \hbox{ if } i=k 
\end{array} \right. $: an upper bound on 
the amount of flow with origin $i$ that can traverse edge $\{k,m\}$ (in any direction), for all $i\in N$, $k<m\in N$.
\end{itemize}

In Figure \ref{xxy} we illustrate how $O_{ikm}$ is calculated. An origin $i$ with $O_i=100$ is assumed, 
and two hubs, $k$ and $m$ are selected. In the figure, $w_{ii}=5$, $w_{ik}=4$ and $w_{im}=7$. The flow 
with origin $i$ that traverses the link between $k$ and $m$ will not include the $w_{ii}$ units
going from $i$ to itself, since they will remain at $i$ (in case $i$ is a hub) or it will go to 
only one hub and back (in case $i$ is a spoke). On the other hand, either the path in the large tree from $i$ to 
$k$ will traverse $m$ or the path from $i$ to $m$ will traverse $k$. This means that 
the flow associated with edge $\{k,m\}$ will not include either $w_{ik}$ or $w_{im}$ units. 
In the figure, the worst case is $\min \{w_{ik}, w_{im}\} = w_{ik}=4$, thus the upper bound can be fixed in 
$100-5-4=91$ units. When $k=i$ the last argument does not apply, but 
still $w_{ii}$ can be subtracted from $O_i$. 
\begin{figure}[h]
\begin{center}
% This file was created by matplotlib2tikz v0.6.13.
\begin{tikzpicture}[xscale=0.65,yscale=0.65,shorten >=1pt]
\tikzstyle{arrow1} = [thick,scale=3,->,>=stealth]
\tikzstyle{arrow2} = [thick,scale=3,<->,>=stealth]

\coordinate(X1) at (8.426952,0.357648);
\coordinate(X2) at (7.566704,9.700764);
\coordinate(X3) at (9.119774,3.694816);
\coordinate(X4) at (5.128467,9.970362);
\coordinate(X5) at (3.674127,5.450939);
\coordinate(X6) at (8.709690,8.260512);
\coordinate(X7) at (5.651874,2.469880);
\coordinate(X8) at (9.677091,5.847142);
\coordinate(X9) at (3.998395,8.242882);
\coordinate(X10) at (4.199842,2.748433);
%\draw (X1) circle (10pt);
\node[circle,draw](X-1) at (X1) {};
\node[circle,draw](X-2) at (X2) {};
\node[circle,draw](X-3) at (X3) {$m$};
\node[circle,draw](X-4) at (X4) {};
\node[circle,draw](X-5) at (X5) {};
\node[circle,draw](X-6) at (X6) {$i$};
\node[circle,draw](X-7) at (X7) {};
\node[circle,draw](X-8) at (X8) {};
\node[circle,draw](X-9) at (X9) {$k$};
\node[circle,draw](X-10) at (X10) {};
\node[xshift=0.5cm, yshift=0.5cm] at (X6) {$O_i=100$};
\node[xshift=1cm] at (X6) {$w_{ii}=5$};
\node[yshift=-1.5cm, xshift=-0.1cm,rotate=95] at (X6) {$w_{im}=7$};
\node[xshift=2cm, yshift=0.2cm] at (X9) {$w_{ik}=4$};

\node[yshift=-1.5cm, xshift=2cm,rotate=320] at (X9) {Flow $\leq 91$};

\draw (X-1)--(X-3);
\draw (X-2)--(X-6);
\draw (X-3)--(X-1);
\draw [dashed, arrow1] (X-6)--(X-3);
\draw [dotted, arrow2, very thick](X-3)--(X-9);
\draw (X-3)--(X-7);
\draw [dotted, arrow1,very thick](X-6)--(X-9);
\draw (X-4)--(X-9);
\draw (X-5)--(X-9);
\draw (X-7)--(X-3);
\draw (X-8)--(X-3);
\draw (X-10)--(X-7);

\end{tikzpicture}
\caption{Upper bound on the flow with origin in $i$ traversing 
edge $\{ k,m\}$.\label{xxy}}
\end{center}
\end{figure}
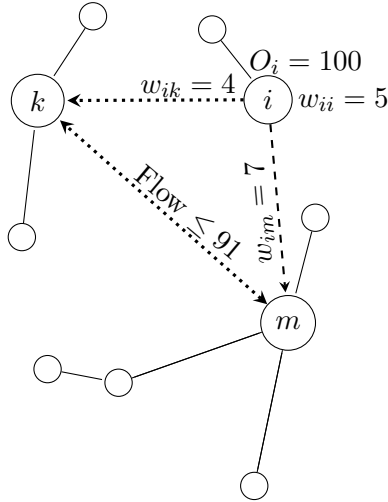
With the above notation, the goals of THLPU are drawn in Figure \ref{fig2}. The first decision to be made is to choose a subset of $p$ nodes, among those in $N$, to be used as hubs (filled nodes) and link the hubs in tree shape (we will call this 
the {\em small tree} --see Figure \subref{sfig:1}). Second, one has to allocate
every non-hub node to a hub, giving rise to the so-called {\em large tree} (see Figure 
\subref{sfig:2}). And finally, one has to decide which $q$ nodes to upgrade out of the $p$ hubs
(see Figure \subref{sfig:3}) in such a way that the total cost be minimized (upgraded nodes are represented by ring nodes in the picture). The three decisions are to be made simultaneously.

\setlength{\abovecaptionskip}{-2pt}
\begin{figure}[h]
\begin{center}
\begin{subfigure}[b]{0.3\textwidth}
\fbox{\begin{tikzpicture}[scale=0.6]
\coordinate(X1) at (8.426952,0.357648);
\coordinate(X2) at (7.566704,9.700764);
\coordinate(X3) at (9.119774,3.694816);
\coordinate(X4) at (5.128467,9.970362);
\coordinate(X5) at (3.674127,5.450939);
\coordinate(X6) at (8.709690,8.260512);
\coordinate(X7) at (5.651874,2.469880);
\coordinate(X8) at (9.677091,5.847142);
\coordinate(X9) at (3.998395,8.242882);
\coordinate(X10) at (4.199842,2.748433);
%\draw (X1) circle (10pt);
\node[circle,draw,fill](X-1) at (X1) {};
\node[circle,draw](X-2) at (X2) {};
\node[circle,draw,fill](X-3) at (X3) {};
\node[circle,draw](X-4) at (X4) {};
\node[circle,draw](X-5) at (X5) {};
\node[circle,draw,fill](X-6) at (X6) {};
\node[circle,draw,fill](X-7) at (X7) {};
\node[circle,draw](X-8) at (X8) {};
\node[circle,draw,fill](X-9) at (X9) {};
\node[circle,draw](X-10) at (X10) {};

\draw[very thick] (X-9)--(X-3);
\draw[very thick] (X-6)--(X-3);
\draw[very thick] (X-7)--(X-3);
\draw[very thick] (X-1)--(X-3);
\end{tikzpicture}}\caption{\label{sfig:1}}
\end{subfigure}~\begin{subfigure}[b]{0.3\textwidth}
\fbox{\begin{tikzpicture}[scale=0.6]
\coordinate(X1) at (8.426952,0.357648);
\coordinate(X2) at (7.566704,9.700764);
\coordinate(X3) at (9.119774,3.694816);
\coordinate(X4) at (5.128467,9.970362);
\coordinate(X5) at (3.674127,5.450939);
\coordinate(X6) at (8.709690,8.260512);
\coordinate(X7) at (5.651874,2.469880);
\coordinate(X8) at (9.677091,5.847142);
\coordinate(X9) at (3.998395,8.242882);
\coordinate(X10) at (4.199842,2.748433);
%\draw (X1) circle (10pt);
\node[circle,draw,fill](X-1) at (X1) {};
\node[circle,draw](X-2) at (X2) {};
\node[circle,draw,fill](X-3) at (X3) {};
\node[circle,draw](X-4) at (X4) {};
\node[circle,draw](X-5) at (X5) {};
\node[circle,draw,fill](X-6) at (X6) {};
\node[circle,draw,fill](X-7) at (X7) {};
\node[circle,draw](X-8) at (X8) {};
\node[circle,draw,fill](X-9) at (X9) {};
\node[circle,draw](X-10) at (X10) {};

\draw[very thick] (X-9)--(X-3);
\draw[very thick] (X-6)--(X-3);
\draw[very thick] (X-7)--(X-3);
\draw[very thick] (X-1)--(X-3);

\draw (X-4)--(X-9);
\draw (X-5)--(X-9);
\draw (X-2)--(X-6);
\draw (X-8)--(X-3);
\draw (X-10)--(X-7);
\end{tikzpicture}}
\caption{\label{sfig:2}}
\end{subfigure}~\begin{subfigure}[b]{0.3\textwidth}
\fbox{\begin{tikzpicture}[scale=0.6]
\coordinate(X1) at (8.426952,0.357648);
\coordinate(X2) at (7.566704,9.700764);
\coordinate(X3) at (9.119774,3.694816);
\coordinate(X4) at (5.128467,9.970362);
\coordinate(X5) at (3.674127,5.450939);
\coordinate(X6) at (8.709690,8.260512);
\coordinate(X7) at (5.651874,2.469880);
\coordinate(X8) at (9.677091,5.847142);
\coordinate(X9) at (3.998395,8.242882);
\coordinate(X10) at (4.199842,2.748433);
%\draw (X1) circle (10pt);
\node[circle,draw,fill](X-1) at (X1) {};
\node[circle,draw](X-2) at (X2) {};
\node[circle,draw,line width=1mm](X-3) at (X3) {};
\node[circle,draw](X-4) at (X4) {};
\node[circle,draw](X-5) at (X5) {};
\node[circle,draw,fill](X-6) at (X6) {};
\node[circle,draw,fill](X-7) at (X7) {};
\node[circle,draw](X-8) at (X8) {};
\node[circle,draw,line width=1mm](X-9) at (X9) {};
\node[circle,draw](X-10) at (X10) {};

\draw[very thick] (X-9)--(X-3);
\draw[very thick] (X-6)--(X-3);
\draw[very thick] (X-7)--(X-3);
\draw[very thick] (X-1)--(X-3);

\draw (X-4)--(X-9);
\draw (X-5)--(X-9);
\draw (X-2)--(X-6);
\draw (X-8)--(X-3);
\draw (X-10)--(X-7);
\end{tikzpicture}}\caption{\label{sfig:3}}
\end{subfigure}
\end{center}
\caption{Decisions in THLPU: Construction of the small tree (left); allocate spokes to hubs (center); and upgrade nodes (right).\label{fig2}}
\end{figure}
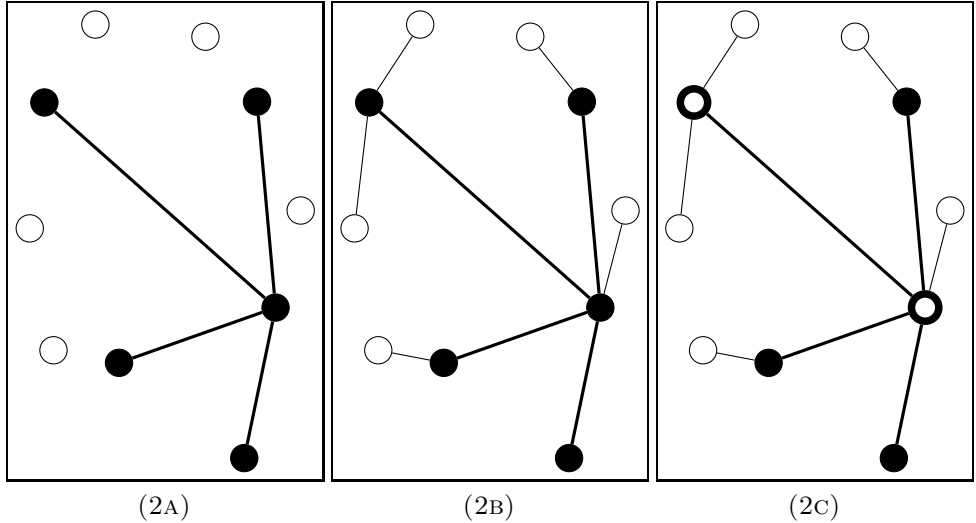

The above mentioned overall cost is obtained adding up the costs associated with all the elements $(i,j)$ in 
$N\times N$. Given $i, j \in N$, its associated cost is the minimum of the costs 
of the walks from $i$ to $j$ in the large tree which traverse at least one hub, times 
the amount of flow $w_{ij}$. We consider several possible cases in Figure \ref{fig3}. 
In Subfigure \subref{sfig:4}, origin $i$ and destination $j$ are different non-hub nodes. The travel cost of sending $w_{ij}$ from $i$ to $j$ is given by the unique path between $i$ and $j$ in the large 
tree. The same occurs when $i$ or $j$ (or both) are hub nodes (see Subfigure \subref{sfig:5}). In case $i=j$, two different situations may happen: $i$ is a hub or $i$ is a non-hub node. In the first case, the $w_{ii}$ units of product 
do not carry any cost (recall that $d_{ii}=0$). In the second,  the $w_{ii}$ units of product are assigned to the hub to which $i$ has 
been allocated and back, repeating the edge; for this reason the optimal route is not always 
a path but sometimes a walk in the graph $G$ (see Subfigure \subref{sfig:6}).

The cost of a given path is the overall sum of the costs of its edges,
taking into account that these depend on the category of the two extremes of the edge: 
non-hub, non-upgraded hub or upgraded hub.

\begin{figure}[h]
\begin{center}
\begin{subfigure}[b]{0.3\textwidth}
\fbox{\begin{tikzpicture}[scale=0.6]
\tikzstyle{arrow} = [thick,scale=3,->,>=stealth]
\coordinate(X1) at (8.426952,0.357648);
\coordinate(X2) at (7.566704,9.700764);
\coordinate(X3) at (9.119774,3.694816);
\coordinate(X4) at (5.128467,9.970362);
\coordinate(X5) at (3.674127,5.450939);
\coordinate(X6) at (8.709690,8.260512);
\coordinate(X7) at (5.651874,2.469880);
\coordinate(X8) at (9.677091,5.847142);
\coordinate(X9) at (3.998395,8.242882);
\coordinate(X10) at (4.199842,2.748433);
%\draw (X1) circle (10pt);
\node[circle,draw,fill](X-1) at (X1) {};
\node[circle,draw,inner sep=1pt](X-2) at (X2) {$i$};
\node[circle,draw,line width=1mm](X-3) at (X3) {};
\node[circle,draw](X-4) at (X4) {};
\node[circle,draw,inner sep=1pt](X-5) at (X5) {$j$};
\node[circle,draw,fill](X-6) at (X6) {};
\node[circle,draw,fill](X-7) at (X7) {};
\node[circle,draw](X-8) at (X8) {};
\node[circle,draw,line width=1mm](X-9) at (X9) {};
\node[circle,draw](X-10) at (X10) {};

\draw[very thick, dashed,arrow] (X-3)--(X-9);
\draw[very thick,dashed,arrow] (X-6)--(X-3);
\draw[very thick] (X-7)--(X-3);
\draw[very thick] (X-1)--(X-3);

\draw (X-4)--(X-9);
\draw[dashed,arrow] (X-9)--(X-5);
\draw[dashed,arrow] (X-2)--(X-6);
\draw (X-8)--(X-3);
\draw (X-10)--(X-7);
\end{tikzpicture}}\caption{\label{sfig:4}}
\end{subfigure}~\begin{subfigure}[b]{0.3\textwidth}
\fbox{\begin{tikzpicture}[scale=0.6]
\tikzstyle{arrow} = [thick,scale=3,->,>=stealth]
\coordinate(X1) at (8.426952,0.357648);
\coordinate(X2) at (7.566704,9.700764);
\coordinate(X3) at (9.119774,3.694816);
\coordinate(X4) at (5.128467,9.970362);
\coordinate(X5) at (3.674127,5.450939);
\coordinate(X6) at (8.709690,8.260512);
\coordinate(X7) at (5.651874,2.469880);
\coordinate(X8) at (9.677091,5.847142);
\coordinate(X9) at (3.998395,8.242882);
\coordinate(X10) at (4.199842,2.748433);
%\draw (X1) circle (10pt);
\node[circle,draw,fill,inner sep=1pt](X-1) at (X1) {{\color{white}$j$}};
\node[circle,draw,inner sep=1pt](X-2) at (X2) {$i$};
\node[circle,draw,line width=1mm](X-3) at (X3) {};
\node[circle,draw](X-4) at (X4) {};
\node[circle,draw](X-5) at (X5) {};
\node[circle,draw,fill](X-6) at (X6) {};
\node[circle,draw,fill](X-7) at (X7) {};
\node[circle,draw](X-8) at (X8) {};
\node[circle,draw,line width=1mm](X-9) at (X9) {};
\node[circle,draw](X-10) at (X10) {};

\draw[very thick] (X-3)--(X-9);
\draw[very thick,dashed,arrow] (X-6)--(X-3);
\draw[very thick] (X-7)--(X-3);
\draw[very thick,dashed, arrow] (X-3)--(X-1);

\draw (X-4)--(X-9);
\draw (X-9)--(X-5);
\draw[dashed,arrow] (X-2)--(X-6);
\draw (X-8)--(X-3);
\draw (X-10)--(X-7);
\end{tikzpicture}}\caption{\label{sfig:5}}
\end{subfigure}~\begin{subfigure}[b]{0.3\textwidth}
\fbox{\begin{tikzpicture}[scale=0.6]
\tikzstyle{arrow} = [thick,scale=3,->,>=stealth]
\coordinate(X1) at (8.426952,0.357648);
\coordinate(X2) at (7.566704,9.700764);
\coordinate(X3) at (9.119774,3.694816);
\coordinate(X4) at (5.128467,9.970362);
\coordinate(X5) at (3.674127,5.450939);
\coordinate(X6) at (8.709690,8.260512);
\coordinate(X7) at (5.651874,2.469880);
\coordinate(X8) at (9.677091,5.847142);
\coordinate(X9) at (3.998395,8.242882);
\coordinate(X10) at (4.199842,2.748433);
%\draw (X1) circle (10pt);
\node[circle,draw,fill](X-1) at (X1) {};
\node[circle,draw,inner sep=1pt](X-2) at (X2) {$i$};
\node[circle,draw,line width=1mm](X-3) at (X3) {};
\node[circle,draw](X-4) at (X4) {};
\node[circle,draw](X-5) at (X5) {};
\node[circle,draw,fill](XA6) at (X6) {};
\node[circle,draw,fill](X-7) at (X7) {};
\node[circle,draw](X-8) at (X8) {};
\node[circle,draw,line width=1mm](X-9) at (X9) {};
\node[circle,draw](X-10) at (X10) {};

\draw[very thick] (X-3)--(X-9);
\draw[very thick] (X-6)--(X-3);
\draw[very thick] (X-7)--(X-3);
\draw[very thick] (X-3)--(X-1);

\draw (X-4)--(X-9);
\draw (X-9)--(X-5);
%\draw  (X-2)--(X-6);
\draw[dashed,arrow] (X-2.225) -- (X-6.180);
\draw[dashed,arrow] (X-6.90) -- (X-2.0);
\draw (X-8)--(X-3);
\draw (X-10)--(X-7);
\end{tikzpicture}}\caption{\label{sfig:6}}
\end{subfigure}
\end{center}
\caption{Computation of costs in THLPU.\label{fig3}}
\end{figure}
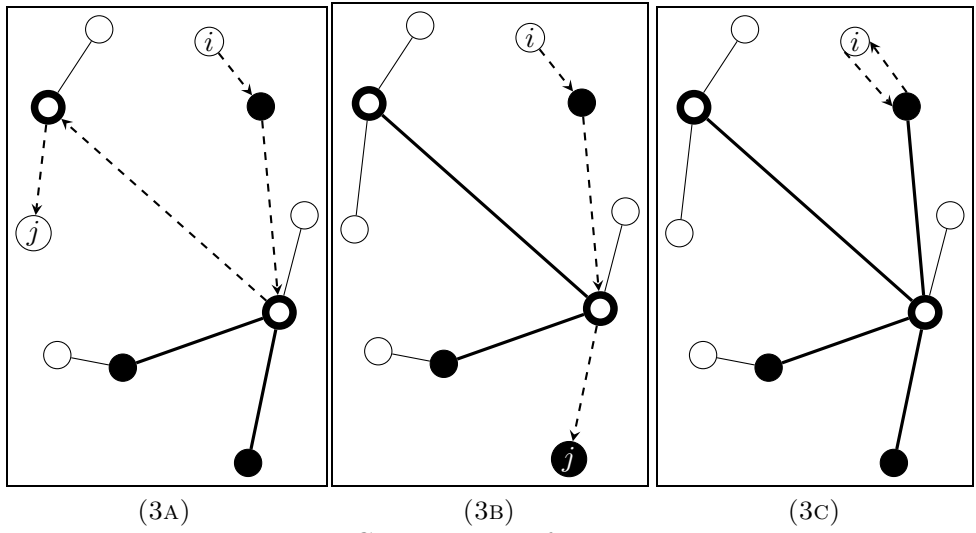

Finally, we would like to illustrate that THLPU is not equivalent to solve first the THLP and then decide the upgraded hubs, since in many simple situations, the new cost structure may change the combinatorial shape of the solution.

\begin{example}\label{example}
Let us consider the following $10$ random points on the plane $\left\{(8.43,0.36)\right.$, $(7.57,9.70), (9.12,3.69), (5.13,9.97), (3.67,5.45), (8.71,8.26)$, $(5.65,2.47), (9.68,5.85)$, $\left.(4.00,8.24), (4.20,2.75)\right\}$, with O-D flow matrix given by
$$w=\begin{pmatrix}
634&731&794&783&482&84&914&575&17&123\\
510&593&31&163&659&535&902&533&817&382\\
821&287&109&775&958&262&478&326&996&572\\
702&763&802&396&760&171&912&28&198&840\\
184&218&2&34&676&299&102&555&763&3\\
673&897&748&260&519&121&577&174&0&459\\
861&645&11&236&5&236&503&750&681&246\\
982&54&468&912&705&919&175&548&698&497\\
832&249&947&282&183&485&552&956&147&713\\
292&826&616&95&720&485&382&19&393&940
\end{pmatrix}.$$

We consider as basic cost structure the Euclidean distance between the demand points, with  discount factors $\alpha \geq \rho \geq \gamma$, such that: $\alpha$ is the discount factor over the basic cost for flow traversing non-upgraded hubs, $\rho$ for sending flow through an upgraded and a non upgraded hub, and $\gamma$ when links connecting two upgraded nodes are used. Recall that the THLPU when $\alpha=\rho=\gamma$ coincides with the THLP. In Figure \ref{fig4} one can see that the tree structure of the THLP for $p=5$ and $\alpha=0.8$ (Subfigure \subref{sfig:7}) does not necessarily  coincide with that of the THLP. In Subfigures \subref{sfig:8} and \subref{sfig:9} we draw the solutions of THLPU for $p=5$, $q=2$ (two nodes are upgraded) and two different discount factors $(\alpha,\rho,\gamma)=(0.8,0.6,0.6)$ (where no extra discount is assumed when both extremes are upgraded with respect to the case in which only one of them is) and  $(\alpha,\rho,\gamma)=(0.8,0.4,0.2)$.

\begin{figure}[h]
\begin{center}
\begin{subfigure}[b]{0.3\textwidth}
\fbox{\begin{tikzpicture}[scale=0.6]
\coordinate(X1) at (8.426952,0.357648);
\coordinate(X2) at (7.566704,9.700764);
\coordinate(X3) at (9.119774,3.694816);
\coordinate(X4) at (5.128467,9.970362);
\coordinate(X5) at (3.674127,5.450939);
\coordinate(X6) at (8.709690,8.260512);
\coordinate(X7) at (5.651874,2.469880);
\coordinate(X8) at (9.677091,5.847142);
\coordinate(X9) at (3.998395,8.242882);
\coordinate(X10) at (4.199842,2.748433);
%\draw (X1) circle (10pt);
\node[circle,draw](X-1) at (X1) {};
\node[circle,draw](X-2) at (X2) {};
\node[circle,draw,fill](X-3) at (X3) {};
\node[circle,draw](X-4) at (X4) {};
\node[circle,draw,fill](X-5) at (X5) {};
\node[circle,draw,fill](X-6) at (X6) {};
\node[circle,draw,fill](X-7) at (X7) {};
\node[circle,draw,fill](X-8) at (X8) {};
\node[circle,draw](X-9) at (X9) {};
\node[circle,draw](X-10) at (X10) {};

\draw[very thick] (X-5)--(X-6);
\draw[very thick] (X-6)--(X-8);
\draw[very thick] (X-8)--(X-3);
\draw[very thick] (X-3)--(X-7);

\draw (X-4)--(X-6);
\draw (X-9)--(X-6);
\draw (X-2)--(X-6);
\draw (X-1)--(X-3);
\draw (X-10)--(X-7);
\end{tikzpicture}}\caption{$(\alpha,\rho,\gamma)=(0.8,0.8,0.8)$\label{sfig:7}}
\end{subfigure}~
\begin{subfigure}[b]{0.3\textwidth}
\fbox{\begin{tikzpicture}[scale=0.6]
\coordinate(X1) at (8.426952,0.357648);
\coordinate(X2) at (7.566704,9.700764);
\coordinate(X3) at (9.119774,3.694816);
\coordinate(X4) at (5.128467,9.970362);
\coordinate(X5) at (3.674127,5.450939);
\coordinate(X6) at (8.709690,8.260512);
\coordinate(X7) at (5.651874,2.469880);
\coordinate(X8) at (9.677091,5.847142);
\coordinate(X9) at (3.998395,8.242882);
\coordinate(X10) at (4.199842,2.748433);
%\draw (X1) circle (10pt);
\node[circle,draw,fill](X-1) at (X1) {};
\node[circle,draw](X-2) at (X2) {};
\node[circle,draw,line width=1mm](X-3) at (X3) {};
\node[circle,draw](X-4) at (X4) {};
\node[circle,draw](X-5) at (X5) {};
\node[circle,draw,fill](X-6) at (X6) {};
\node[circle,draw,fill](X-7) at (X7) {};
\node[circle,draw](X-8) at (X8) {};
\node[circle,draw,line width=1mm](X-9) at (X9) {};
\node[circle,draw](X-10) at (X10) {};

\draw[very thick] (X-9)--(X-3);
\draw[very thick] (X-6)--(X-3);
\draw[very thick] (X-7)--(X-3);
\draw[very thick] (X-1)--(X-3);

\draw (X-4)--(X-9);
\draw (X-5)--(X-9);
\draw (X-2)--(X-6);
\draw (X-8)--(X-3);
\draw (X-10)--(X-7);
\end{tikzpicture}}\caption{$(\alpha,\rho,\gamma)=(0.8,0.6,0.6)$\label{sfig:8}}
\end{subfigure}~
\begin{subfigure}[b]{0.3\textwidth}
\fbox{\begin{tikzpicture}[scale=0.6]
\coordinate(X1) at (8.426952,0.357648);
\coordinate(X2) at (7.566704,9.700764);
\coordinate(X3) at (9.119774,3.694816);
\coordinate(X4) at (5.128467,9.970362);
\coordinate(X5) at (3.674127,5.450939);
\coordinate(X6) at (8.709690,8.260512);
\coordinate(X7) at (5.651874,2.469880);
\coordinate(X8) at (9.677091,5.847142);
\coordinate(X9) at (3.998395,8.242882);
\coordinate(X10) at (4.199842,2.748433);
%\draw (X1) circle (10pt);
\node[circle,draw,fill](X-1) at (X1) {};
\node[circle,draw,fill](X-2) at (X2) {};
\node[circle,draw,fill](X-3) at (X3) {};
\node[circle,draw](X-4) at (X4) {};
\node[circle,draw](X-5) at (X5) {};
\node[circle,draw](X-6) at (X6) {};
\node[circle,draw,line width=1mm](X-7) at (X7) {};
\node[circle,draw](X-8) at (X8) {};
\node[circle,draw,line width=1mm](X-9) at (X9) {};
\node[circle,draw](X-10) at (X10) {};

\draw[very thick] (X-9)--(X-7);
\draw[very thick] (X-9)--(X-2);
\draw[very thick] (X-7)--(X-3);
\draw[very thick] (X-1)--(X-7);

\draw (X-4)--(X-9);
\draw (X-5)--(X-9);
\draw (X-2)--(X-6);
\draw (X-8)--(X-3);
\draw (X-10)--(X-7);
\end{tikzpicture}}\caption{$(\alpha,\rho,\gamma)=(0.8,0.4,0.2)$\label{sfig:9}}
\end{subfigure}
\end{center}
\caption{Solutions for THLP and THLPU in Example \ref{example}\label{fig4}}
\end{figure}
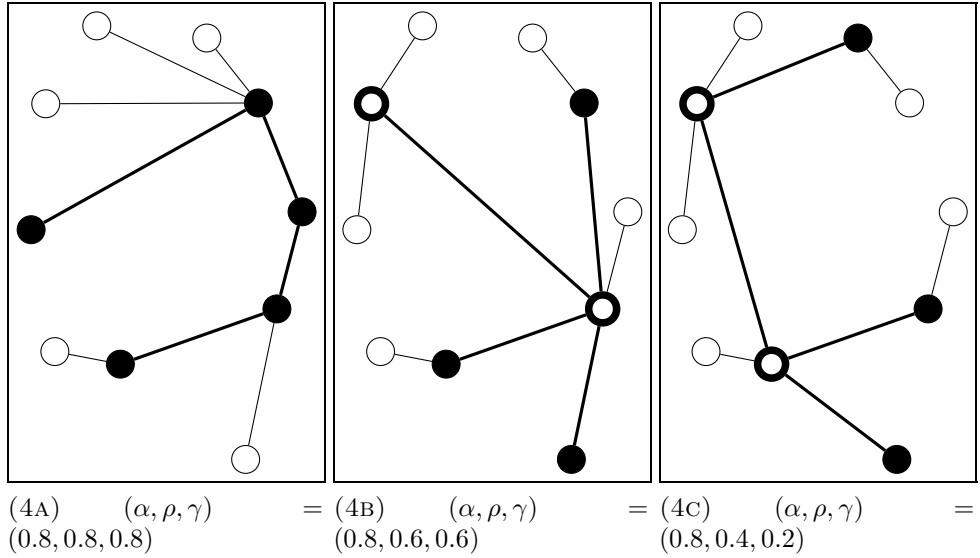

\end{example}

\section{A Mixed Integer Linear Programming formulation for THLPU}\label{sec:3}

In this section we develop a first Mixed Integer Linear Programming (MILP) formulation for the THLPU. It uses the variables formerly introduced in \cite{contreras} for the THLP 
plus two sets of specific variables $t_k$ and $\theta_{ij}$, to account for upgraded nodes and the reduced costs when the flow traverses upgraded nodes.

We first define several families of binary variables:

\begin{equation}\label{zvars}
z_{ik} = \left\{\begin{array}{cl}
1 & \mbox{if non-hub $i$ is allocated to hub $k$, for $i\neq k$}\\
1 & \mbox{if $k$ is a hub node and $i=k$}\\
0 &\mbox{otherwise}
\end{array}\right. \forall i,k \in N,
\end{equation}

\begin{equation}\label{tvars}
t_k = \left\{\begin{array}{cl}
1 & \mbox{if node  $k$ is an upgraded hub,}\\
0 &\mbox{otherwise}
\end{array}\right. \forall k \in N,
\end{equation}

\begin{equation}\label{svars}
s_{km}= \left\{\begin{array}{cl}
1 & \mbox{if $k$ and $m$ are linked hubs in the small tree,}\\
0 &\mbox{otherwise}
\end{array}\right. \forall k<m:\ \{k,m\} \in E.
\end{equation}
Observe that, as usual in discrete location problems, $z_{kk}$ taking value 1 
can be interpreted as a self-allocation of hub $k$. Hence, all nodes will be allocated to some hub. In Subfigure \subref{sfig:10}, the $z$-values equal to 1 are shown on the optimal solution for the THPLU of Example \ref{example} (for $(\alpha,\rho,\gamma)=(0.8,0.4,0.2)$). On the other hand, $t_k$ will take value 1 if $k$ is one 
of the nodes chosen as hubs and also upgraded, thus implying $z_{kk}=1$. In Subfigure \subref{sfig:11}, we illustrate for the same solution the nodes for which the $t$-values take value $1$. Finally, in Subfigure \subref{sfig:12} we show the values for the $s$-variables taking value 1 according with the depicted small tree 
and chosen hubs.
\begin{figure}[H]
\begin{center}
\begin{subfigure}[b]{0.3\textwidth}
\fbox{\begin{tikzpicture}[scale=0.6]
\coordinate(X1) at (8.426952,0.357648);
\coordinate(X2) at (7.566704,9.700764);
\coordinate(X3) at (9.119774,3.694816);
\coordinate(X4) at (5.128467,9.970362);
\coordinate(X5) at (3.674127,5.450939);
\coordinate(X6) at (8.709690,8.260512);
\coordinate(X7) at (5.651874,2.469880);
\coordinate(X8) at (9.677091,5.847142);
\coordinate(X9) at (3.998395,8.242882);
\coordinate(X10) at (4.199842,2.748433);
%\draw (X1) circle (10pt);
\node[circle,draw,fill,inner sep=2pt](X-1) at (X1) {\small {\color{white} $1$}};
\node[circle,draw,fill,inner sep=2pt](X-2) at (X2) {\small {\color{white} $1$}};
\node[circle,draw,fill,inner sep=2pt](X-3) at (X3) {\small {\color{white} $1$}};
\node[circle,draw](X-4) at (X4) {};
\node[circle,draw](X-5) at (X5) {};
\node[circle,draw](X-6) at (X6) {};
\node[circle,draw,line width=1mm,inner sep=2pt](X-7) at (X7) {\small $1$};
\node[circle,draw](X-8) at (X8) {};
\node[circle,draw,line width=1mm,inner sep=2pt](X-9) at (X9) {\small $1$};
\node[circle,draw](X-10) at (X10) {};

\draw[very thick] (X-9)--(X-7);
\draw[very thick] (X-9)--(X-2);
\draw[very thick] (X-7)--(X-3);
\draw[very thick] (X-1)--(X-7);

\draw (X-4) -- node [midway, above,left]{\tiny$1$} (X-9);
\draw (X-5) -- node [midway, above,left]{\tiny$1$}  (X-9);
\draw (X-2) -- node [midway, above,left]{\tiny$1$}  (X-6);
\draw (X-8) -- node [midway, above,left]{\tiny$1$}  (X-3);
\draw (X-10) -- node[midway,below,pos=0.42]{\tiny$1$}  (X-7);

\end{tikzpicture}}\caption{\label{sfig:10}}
\end{subfigure}~\begin{subfigure}[b]{0.3\textwidth}
\fbox{\begin{tikzpicture}[scale=0.6]
\coordinate(X1) at (8.426952,0.357648);
\coordinate(X2) at (7.566704,9.700764);
\coordinate(X3) at (9.119774,3.694816);
\coordinate(X4) at (5.128467,9.970362);
\coordinate(X5) at (3.674127,5.450939);
\coordinate(X6) at (8.709690,8.260512);
\coordinate(X7) at (5.651874,2.469880);
\coordinate(X8) at (9.677091,5.847142);
\coordinate(X9) at (3.998395,8.242882);
\coordinate(X10) at (4.199842,2.748433);
%\draw (X1) circle (10pt);
\node[circle,draw,fill,inner sep=2pt](X-1) at (X1) {\small {\color{white} $0$}};
\node[circle,draw,fill,inner sep=2pt](X-2) at (X2) {\small {\color{white} $0$}};
\node[circle,draw,fill,inner sep=2pt](X-3) at (X3) {\small {\color{white} $0$}};
\node[circle,draw](X-4) at (X4) {};
\node[circle,draw](X-5) at (X5) {};
\node[circle,draw](X-6) at (X6) {};
\node[circle,draw,line width=1mm,inner sep=2pt](X-7) at (X7) {\small $1$};
\node[circle,draw](X-8) at (X8) {};
\node[circle,draw,line width=1mm,inner sep=2pt](X-9) at (X9) {\small $1$};
\node[circle,draw](X-10) at (X10) {};

\draw[very thick] (X-9)--(X-7);
\draw[very thick] (X-9)--(X-2);
\draw[very thick] (X-7)--(X-3);
\draw[very thick] (X-1)--(X-7);

\draw (X-4)--(X-9);
\draw (X-5)--(X-9);
\draw (X-2)--(X-6);
\draw (X-8)--(X-3);
\draw (X-10)--(X-7);
\end{tikzpicture}}\caption{\label{sfig:11}}
\end{subfigure}~\begin{subfigure}[b]{0.3\textwidth}
\fbox{\begin{tikzpicture}[scale=0.6]
\coordinate(X1) at (8.426952,0.357648);
\coordinate(X2) at (7.566704,9.700764);
\coordinate(X3) at (9.119774,3.694816);
\coordinate(X4) at (5.128467,9.970362);
\coordinate(X5) at (3.674127,5.450939);
\coordinate(X6) at (8.709690,8.260512);
\coordinate(X7) at (5.651874,2.469880);
\coordinate(X8) at (9.677091,5.847142);
\coordinate(X9) at (3.998395,8.242882);
\coordinate(X10) at (4.199842,2.748433);
%\draw (X1) circle (10pt);
\node[circle,draw,fill](X-1) at (X1) {\tiny {\color{black} $1$}};
\node[circle,draw,fill](X-2) at (X2) {\tiny {\color{black} $1$}};
\node[circle,draw,fill](X-3) at (X3) {\tiny {\color{black} $1$}};
\node[circle,draw](X-4) at (X4) {};
\node[circle,draw](X-5) at (X5) {};
\node[circle,draw](X-6) at (X6) {};
\node[circle,draw,line width=1mm](X-7) at (X7) {\tiny {\color{white} $1$}};
\node[circle,draw](X-8) at (X8) {};
\node[circle,draw,line width=1mm](X-9) at (X9) {\tiny {\color{white} $1$}};
\node[circle,draw](X-10) at (X10) {};

\draw[very thick] (X-9)--  node [midway, above,left]{\tiny$1$} (X-7);
\draw[very thick] (X-9)--  node [pos=0.5, above]{\tiny$1$}(X-2);
\draw[very thick] (X-7)--  node [pos=0.5, above]{\tiny$1$}(X-3);
\draw[very thick] (X-1)--  node [pos=0.5, below]{\tiny$1$}(X-7);

\draw (X-4)--(X-9);
\draw (X-5)--(X-9);
\draw (X-2)--(X-6);
\draw (X-8)--(X-3);
\draw (X-10)--(X-7);
\end{tikzpicture}}\caption{\label{sfig:12}}
\end{subfigure}
\end{center}
\caption{Illustration on feasible unitary values for the $z$ (left), $t$ (center) and $s$ (right) variables.\label{fig5}}
\end{figure}
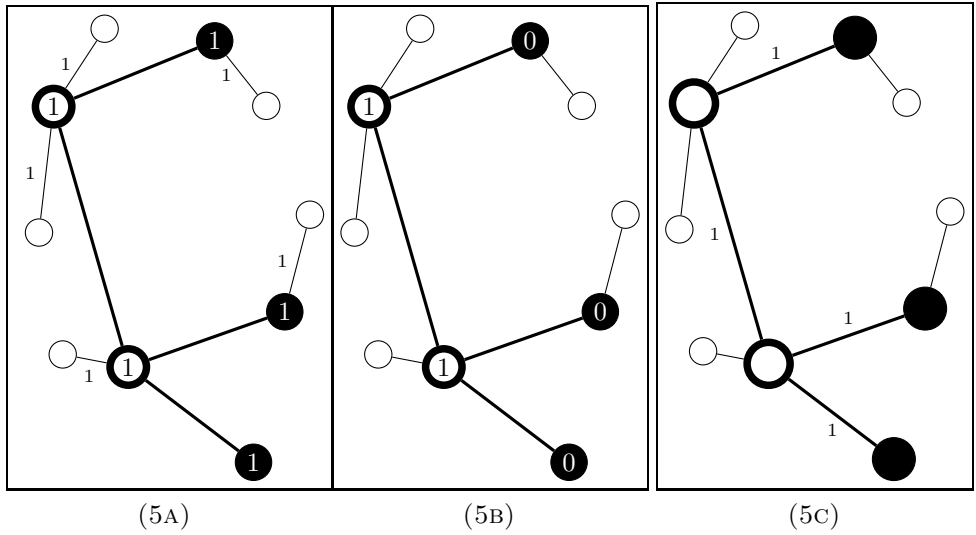
Since the objective value associated to a solution of THLPU includes the cost of the 
flow with origin in each node $i\in N$, we also consider a set of continuous variables to represent the amount of flow sent from a origin that traverses the directed link between two hubs. Hence, we split the edge $\{k,m\} \in E$ into two arcs $(k,m)$ and $(m,k)$, and we define:
\begin{itemize}
\item $r_{ikm}$: amount of flow with origin in node $i$ that traverses arc $(k,m)$ if $k$ and $m$ are both hubs and $k$ and $m$ are linked in the small tree. 
\end{itemize}

We illustrate the values for this set of variables for the same example as above in Figure \ref{fig6}, for the flow with a fixed origin node $i$ (the 
amount inside the nodes represent each of the flows $w_{ij}$ for $j \in N$).
\setlength{\abovecaptionskip}{6pt}
\begin{figure}[h]
\begin{center}
\fbox{\begin{tikzpicture}[scale=0.5]
\coordinate(X1) at (8.426952,0.357648);
\coordinate(X2) at (7.566704,9.700764);
\coordinate(X3) at (9.119774,3.694816);
\coordinate(X4) at (5.128467,9.970362);
\coordinate(X5) at (3.674127,5.450939);
\coordinate(X6) at (8.709690,8.260512);
\coordinate(X7) at (5.651874,2.469880);
\coordinate(X8) at (9.677091,5.847142);
\coordinate(X9) at (3.998395,8.242882);
\coordinate(X10) at (4.199842,2.748433);
%\draw (X1) circle (10pt);
\node[circle,draw,fill, inner sep=2pt](X-1) at (X1) {\scriptsize \color{white}$673$};
\node[circle,draw,fill, inner sep=2pt](X-2) at (X2) {\scriptsize \color{white}$897$};
\node[circle,draw,fill, inner sep=2pt](X-3) at (X3) {\scriptsize \color{white}$748$};
\node[circle,draw, inner sep=2pt](X-4) at (X4) {\scriptsize $260$};
\node[circle,draw, inner sep=2pt](X-5) at (X5) {\scriptsize $519$};
\node[circle,draw, inner sep=2pt](X-6) at (X6) {\scriptsize $121$};
\node[circle,draw,line width=1mm, inner sep=2pt](X-7) at (X7) {\scriptsize $577$};
\node[circle,draw, inner sep=2pt](X-8) at (X8) {\scriptsize $174$};
\node[circle,draw,line width=1mm, inner sep=4pt](X-9) at (X9) {\scriptsize $0$};
\node[circle,draw, inner sep=2pt](X-10) at (X10) {\scriptsize $459$};

\node[above,right] at (9.2,8.5) {$i$};

\draw[very thick] (X-2) -- node [below]{\tiny $3410$} (X-9);
\draw[very thick] (X-7) -- node [above]{\tiny $673$} (X-1);
\draw[very thick] (X-7) -- node [above]{\tiny $922$} (X-3);
\draw[very thick] (X-9) -- node [above,right]{\tiny $2631$} (X-7);

\draw (X-4)--(X-9);
\draw (X-5)--(X-9);
\draw (X-2)--(X-6);
\draw (X-8)--(X-3);
\draw (X-10)--(X-7);
\end{tikzpicture}}
\caption{$r$-values associated with a given origin $i$\label{fig6}}
\end{center}
\end{figure}
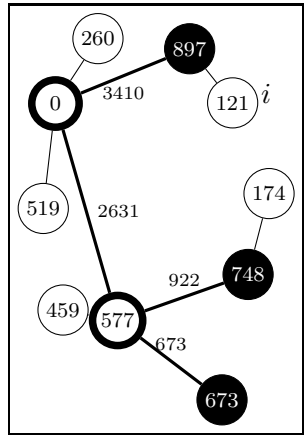
Finally,  we also consider a family of continuous variables, $\theta$, to account for the costs of the flows between hubs. This will depend on the type of the extremes (hubs) of the arcs, i.e., if they are upgraded or not:
\begin{itemize}
\item $\theta_{ikm}$:  total cost of the flow with origin in node $i$ which traverses  edge $\{k,m\}$, $k<m\in N$, if $k$ and $m$ are both hubs.
\end{itemize}

Observe that the $\theta$-variables are directly related with the $r$-variables, in the sense that $\theta_{ikm}$ is proportional to $r_{ikm}+r_{imk}$.
 In particular, with the notation above, we get that
$$
\theta_{ikm} = \left\{\begin{array}{cl}
c_{km} (r_{ikm}+r_{imk}) & \mbox{if $k$ and $m$ are hubs and none of them is upgraded,}\\
c'_{km} (r_{ikm}+r_{imk}) & \mbox{if $k$ and $m$ are hubs and only one of them is upgraded,}\\
c''_{km} (r_{ikm}+r_{imk}) & \mbox{if $k$ and $m$ are upgraded hubs,}\\
0 & \mbox{otherwise,}
\end{array}\right.
$$
for $i\in N$, $\{k,m\} \in E$ with $k<m$.

We are now in a position to state the first MILP formulation for THLPU.
\begin{align}
 \min  &\sum_{i\in N} \sum_{k=1:\atop k\neq i}^n (O_i d_{ik} + D_i d_{ki}) z_{ik} + \sum_{i\in N} 
 \sum_{k=1}^{n-1} \sum_{m=k+1:\atop \{k,m\} \in E}^n \theta_{ikm} \label{thlpu1}\tag{${\rm THLPU}$} \\
\mbox{s.t. } 
  & \sum_{k\in N} z_{kk} = p, & \label{a1}\\
  & \sum_{k=1}^{n-1} \sum_{m=k+1:\atop \{k,m\}\in E}^n s_{km} = p-1,  \label{st1} \\
  & \sum_{k\in N} z_{ik} = 1, \hspace{0.3cm}  \forall i\in N,  \label{b1}\\
  & s_{km} + z_{mk} \le z_{kk}, \hspace{0.3cm} \forall k<m\in N:\ \{k,m\}\in E,\label{ca1} \\
  & s_{km} + z_{km} \le z_{mm}, \hspace{0.3cm} \forall k<m\in N:\ \{k,m\}\in E, \label{cb1} \\
  & \sum_{k\in N} t_k = q,   \label{d1}\\
  & t_k \le z_{kk}, \hspace{0.3cm}  \forall k\in N, \label{e1}\\
  & r_{ikm} + r_{imk} \le O_{ikm} s_{km}, \forall i\neq k<m\in N:\ \{k,m\}\in E, \label{g11a} \\
  & r_{iim} \le O_{iim} s_{\min\{ i,m\},\max\{ i,m\}}, \hspace{0.3cm} \forall \{i,m\}\in E, \label{g11b}\\
  & O_i z_{ik} + \sum_{m=1:\atop m\neq k}^n r_{imk} = \sum_{m=1:\atop m\neq k}^n r_{ikm} + \sum_{j\in N} w_{ij} z_{jk}, \forall i,k\in N,\label{h1}\\
&	\theta_{ikm} \ge c''_{km} (r_{ikm}+r_{imk}),  \forall i\in N,\ k< m\in N, \label{teta1}
        \end{align}
  \begin{align}
&	\theta_{ikm} + \Delta'_{km} O_{ikm} t_k \ge c'_{km} (r_{ikm}+r_{imk}),  \forall i\in N,\ k<m\in N:\ \{k,m\}\in E, \label{teta2a} \\
&	\theta_{ikm} + \Delta'_{km} O_{ikm} t_m \ge c'_{km} (r_{ikm}+r_{imk}), \forall i\in N,\ k<m\in N:\ \{k,m\}\in E,\label{teta2b}\\
&	\theta_{ikm}+ \Delta_{km} O_{ikm}(t_k+t_m) \ge  c_{km} (r_{ikm}+r_{imk}),  \forall i,\ k<m\in N:\{k,m\}\in E,  \label{teta3}\\
&  z_{ik},t_k\in \{0,1\}, \hspace{0.3cm} \forall i,k\in N, \nonumber \\
&  s_{km}\in \{0,1\}, \hspace{0.3cm} \forall k<m\in N:\ \{k,m\}\in E, \nonumber \\
& r_{ikm} \ge 0, \hspace{0.3cm}  \forall i\in N,\ k\neq m\in N. \nonumber
\end{align}
where $\Delta_{km} = c_{km}-c''_{km}$ and $\Delta'_{km}=c'_{km}-c''_{km}$, $\forall  k<m\in N$ with  $\{k,m\}\in E$.

In the objective function of \eqref{thlpu1}, the total cost of the flow with origin and destination in each non-hub node $i$ is added with cost $d$, and the overall sum of the  $\theta$-variables represent the costs due to flow-links between hubs.

Constraints \eqref{a1} and \eqref{st1} state the number of hubs and edges in the small tree. These two sets of constraints, plus the connection, forced by the flows between nodes, ensure that the resulting structure will be a tree. The sets of constraints \eqref{b1} guarantee that each non-hub node is allocated exactly to a single hub. Constraints \eqref{ca1} and \eqref{cb1} also fix to zero $s$-variables when one or two of their extremes are not hubs. 

Regarding the $t$-variables, constraints \eqref{d1} establish in $q$ the number of upgraded nodes, whereas \eqref{e1} ensure that the upgraded nodes will be hubs. Constraints \eqref{g11a} and \eqref{g11b} fix to zero the $r$-variables when appropriate, in this case using an upper bound $O_{ikm}$ since they are continuous variables.  Note also that the values $O_i$ could have been used as natural upper bounds on the values of these 
sums of variables, but $O_{ikm}$ is a better choice that tighten the constraints.
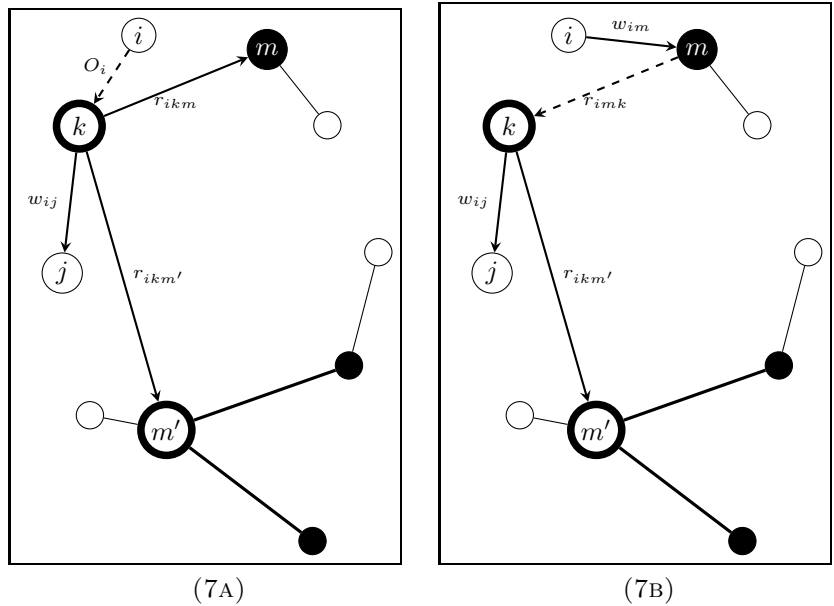
\begin{figure}[h]
\begin{center}
\begin{subfigure}[b]{0.4\textwidth}\fbox{\begin{tikzpicture}[scale=0.7]
\tikzstyle{arrow} = [thick,scale=3,->,>=stealth]
\coordinate(X1) at (8.426952,0.357648);
\coordinate(X2) at (7.566704,9.700764);
\coordinate(X3) at (9.119774,3.694816);
\coordinate(X4) at (5.128467,9.970362);
\coordinate(X5) at (3.674127,5.450939);
\coordinate(X6) at (8.709690,8.260512);
\coordinate(X7) at (5.651874,2.469880);
\coordinate(X8) at (9.677091,5.847142);
\coordinate(X9) at (3.998395,8.242882);
\coordinate(X10) at (4.199842,2.748433);
%\draw (X1) circle (10pt);
\node[circle,draw,fill](X-1) at (X1) {};
\node[circle,draw,fill,inner sep=2pt](X-2) at (X2) {\small {\color{white}$m$}};
\node[circle,draw,fill](X-3) at (X3) {};
\node[circle,draw,inner sep=2pt](X-4) at (X4) {\small $i$};
\node[circle,draw,inner sep=2pt](X-5) at (X5) {\small $j$};
\node[circle,draw](X-6) at (X6) {};
\node[circle,draw,line width=1mm,  inner sep=2pt](X-7) at (X7)  {\small {\color{black}$m^\prime$}};
\node[circle,draw](X-8) at (X8) {};
\node[circle,draw,line width=1mm, inner sep=3pt](X-9) at (X9) {\small {$k$}};
\node[circle,draw](X-10) at (X10) {};

\draw[very thick,arrow] (X-9) -- node [below]{\tiny $r_{ikm}$} (X-2);
\draw[very thick] (X-7) -- node [above]{\tiny $ $} (X-1);
\draw[very thick] (X-7) -- node [above]{\tiny $ $} (X-3);
\draw[very thick,arrow] (X-9) -- node [above,right]{\tiny $r_{ikm'}$} (X-7);

\draw (X-10)--(X-7);
\draw[arrow] (X-9)-- node[left,pos=.5] {\tiny $w_{ij}$}(X-5);
\draw[arrow, dashed] (X-4)-- node[left,pos=.3] {\tiny $O_i$} (X-9);
\draw (X-8)--(X-3);
\draw (X-6)--(X-2);
\end{tikzpicture}}
\caption{\label{sfig:13}}
\end{subfigure}~\begin{subfigure}[b]{0.4\textwidth}
\fbox{\begin{tikzpicture}[scale=0.7]
\tikzstyle{arrow} = [thick,scale=3,->,>=stealth]
\coordinate(X1) at (8.426952,0.357648);
\coordinate(X2) at (7.566704,9.700764);
\coordinate(X3) at (9.119774,3.694816);
\coordinate(X4) at (5.128467,9.970362);
\coordinate(X5) at (3.674127,5.450939);
\coordinate(X6) at (8.709690,8.260512);
\coordinate(X7) at (5.651874,2.469880);
\coordinate(X8) at (9.677091,5.847142);
\coordinate(X9) at (3.998395,8.242882);
\coordinate(X10) at (4.199842,2.748433);
%\draw (X1) circle (10pt);
\node[circle,draw,fill](X-1) at (X1) {};
\node[circle,draw,fill,inner sep=2pt](X-2) at (X2) {\small {\color{white}$m$}};
\node[circle,draw,fill](X-3) at (X3) {};
\node[circle,draw,inner sep=2pt](X-4) at (X4) {\small $i$};
\node[circle,draw,inner sep=2pt](X-5) at (X5) {\small $j$};
\node[circle,draw](X-6) at (X6) {};
\node[circle,draw,line width=1mm,  inner sep=2pt](X-7) at (X7)  {\small {\color{black}$m^\prime$}};
\node[circle,draw](X-8) at (X8) {};
\node[circle,draw,line width=1mm, inner sep=3pt](X-9) at (X9) {\small {$k$}};
\node[circle,draw](X-10) at (X10) {};

\draw[very thick,arrow,dashed] (X-2) -- node [below]{\tiny $r_{imk}$} (X-9);
\draw[very thick] (X-7) -- node [above]{\tiny $ $} (X-1);
\draw[very thick] (X-7) -- node [above]{\tiny $ $} (X-3);
\draw[very thick,arrow] (X-9) -- node [above,right]{\tiny $r_{ikm'}$} (X-7);

\draw (X-10)--(X-7);
\draw[arrow] (X-9)-- node[left,pos=.5] {\tiny $w_{ij}$}(X-5);
%\draw[arrow] (X-9)-- node[left,pos=.5] {\tiny $w_{ij'}$}(X-4);
\draw[arrow] (X-4)-- node[above] {\tiny $w_{im}$}(X-2);
\draw (X-8)--(X-3);
\draw (X-6)--(X-2);
\end{tikzpicture}}
\caption{\label{sfig:14}}
\end{subfigure}
\caption{Graphical representation of flow conservation 
constraints \eqref{h1}} \label{fig7}
\end{center}
\end{figure}
The flow conservation constraints \eqref{h1} are graphically represented in Figure \ref{fig7}. Observe that for fixed nodes $i,k \in N$, the 
terms of the equation are non-null only when $i$ is any node but $k$ is a hub node.  In such a case, when modeling the \textit{inflow} in $k$, coming from $i$, two situations may occur: 1) $i$ is directly allocated to $k$, and 2) $i$ is not allocated to $k$. In the first case, the amount of flow with origin $i$ and which is routed via the hub node $k$ comes directly from $i$, and the amount is the total flow with origin $i$, $O_i$ (see dashed arrow in Subfigure \subref{sfig:13}). On the other hand, if $i$ is not allocated to $k$, the flow from $i$ which traverses $k$ may come from other hub $m$ (see dashed arrow in Subfigure \subref{sfig:14}), amount represented with the value of the variables $r_{imk}$. Concerning the \textit{outflow} from $k$, observe that it can be directly served to final non-hub nodes (as $j$ in Figure \subref{sfig:14}) or routed via another hub nodes (as $m'$ in Figure \subref{sfig:14}), being the amount modeled via the sum of the variables $r_{ikm}$.

Constraints \eqref{teta1}--\eqref{teta3} allow to model the values of the $\theta$-variables. For nodes $i, k, m\in N$ (for $k$ and $m$ hub nodes), whenever an $r_{ikm}$ takes a positive value (recall that by the tree structure of the large tree $r_{imk}$ will take value $0$), constraints \eqref{teta1} bound from below the cost of sending the $r_{ikm}+r_{imk}$ units of flow through 
$\{k,m\}$ to the \textit{default} minimum possible cost, i.e., $c''_{km}$. If $k$ or $m$ (but not both) are upgraded,
\eqref{teta2a}-\eqref{teta2b} change this bound to $c'_{km}$. Finally, if none of the nodes are upgraded, 
constraints \eqref{teta3} increase the bound to $c_{km}$. These constraints together with the minimization criteria, ensure that the costs of traversing hub nodes are well defined.

Observe that \eqref{thlpu1} inherits some of the valid inequalities for the THLP described in \cite{contreras} and which only concerns running flows ($r$-variables) and allocation decisions ($z$-variables):
\begin{align*}
r_{ikm}+r_{imk} \leq& \max \{O_{ikm},O_{imk}\} z_{kk},\quad \forall \{k,m\}\in E, k<m, \mbox{ and}\\
r_{ikm}+r_{imk} \leq& \max \{O_{ikm},O_{imk}\} z_{mm},\quad \forall \{k,m\}\in E, k<m.
\end{align*}

In what follows we derive a larger family of tightening inequalities for  \eqref{thlpu1} based on those obtained in \cite{ortega} and later extended and applied to solving THLP in \cite{contreras}.
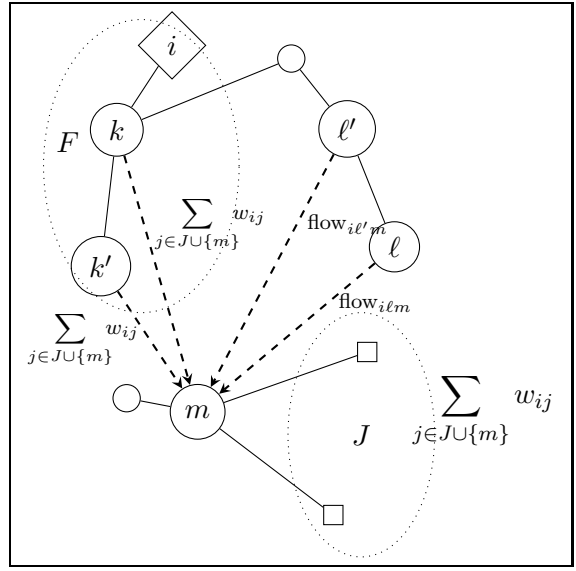
\begin{figure}[h]
\begin{center}
\fbox{\begin{tikzpicture}[scale=0.65]
\tikzstyle{arrow} = [thick,scale=3,->,>=stealth]
\coordinate(X1) at (8.426952,0.357648);
\coordinate(X2) at (7.566704,9.700764);
\coordinate(X3) at (9.119774,3.694816);
\coordinate(X4) at (5.128467,9.970362);
\coordinate(X5) at (3.674127,5.450939);
\coordinate(X6) at (8.709690,8.260512);
\coordinate(X7) at (5.651874,2.469880);
\coordinate(X8) at (9.677091,5.847142);
\coordinate(X9) at (3.998395,8.242882);
\coordinate(X10) at (4.199842,2.748433);
%\draw (X1) circle (10pt);
\node[rectangle,draw](X-1) at (X1) {};
\node[circle,draw](X-2) at (X2) {};
\node[rectangle,draw](X-3) at (X3) {};
\node[diamond,draw](X-4) at (X4) {$i$};
\node[circle,draw](X-5) at (X5) {$k'$};
\node[circle,draw](X-6) at (X6) {$\ell^\prime$};
\node[circle,draw](X-7) at (X7) {$m$};
\node[circle,draw](X-8) at (X8) {$\ell$};
\node[circle,draw](X-9) at (X9) {$k$};
\node[circle,draw](X-10) at (X10) {};

\draw (X-2) --  (X-9);
\draw (X-7) -- (X-1);
\draw (X-7) -- (X-3);
%\draw (X-9) -- (X-7);

\draw (X-4)--(X-9);
\draw (X-5)--(X-9);
\draw (X-2)--(X-6);
%\draw (X-8)--(X-3);
\draw (X-10)--(X-7);

%\draw (X-5)--(X-10);
\draw[dashed, arrow] (X-5)-- node[left,pos=0.5] {\scriptsize $\dsum_{j\in J\cup\{m\}}\!\!\!w_{ij}$} (X-7);
\draw[dashed, arrow] (X-9)-- node[pos=0.3,right] {\scriptsize $\dsum_{j\in J\cup\{m\}}\!\!\!w_{ij}$} (X-7);
\draw[dashed, arrow] (X-8)-- node[pos=0.3,right] {\scriptsize ${\rm flow}_{i\ell m}$} (X-7);
\draw[dashed, arrow] (X-6)-- node[pos=0.3,right] {\scriptsize ${\rm flow}_{i\ell^\prime m}$} (X-7);

\draw (X-6)--(X-8);
%\draw (X-2)--(X-7);
%\draw (X-6)--(X-7);
\draw[dotted] (9,2) ellipse (1.5cm and 2.5cm);
\node at (9,2) {$J$};
\node at (11.5,2.5) {$\dsum_{j \in J \cup \{m\}} w_{ij}$};

\draw[dotted] (4.5,7.5) ellipse (2cm and 3cm);
\node at (3,8) {$F$};

%\draw[dotted] (6,1.5) ellipse (4cm and 2cm);
%\node at (6,1.5) {$N\backslash F$};
\end{tikzpicture}}
\caption{Explanation of valid inequalities \eqref{separadas1}.\label{ineq}}
\end{center}
\end{figure}

To explain them, consider the graph depicted in Figure \ref{ineq}. Let $i\in N$ be a fixed origin (diamond-shaped in the figure) and 
$m$ be a hub to which $i$ has not been allocated (i.e., $i$ can be a hub itself or
$i$ can be a non-hub node allocated to a hub different from $m$). Let then $F$ be a subset 
of nodes not containing $m$ (nodes inside the ellipse in the north-west corner in the figure). Moreover, let $J$ be a subset of non-hub 
nodes allocated to $m$ (nodes inside the ellipse in the south-east corner in the figure). Note 
that $J$ and $F$ do not relate each other. The amount of flow which departing from $i$ is distributed through $m$ to all nodes 
in $J$, represented in the figure by $\sum_{j\in J\cup\{m\}} w_{ij}$, is a lower bound on the amount of incoming flow from $m$.
The set $F$ is used to split this incoming flow among two kind of variables, those defining the adjacency between hubs and those defining the flow circulating through hub nodes.

In \eqref{thlpu1}, the connections of hubs nodes, as well as the flow through hub nodes, are modeled by different sets of variables. The set $F \cup \{m\}$ is associated to connection of hub nodes, by means of the $s$-variables. These variables are multiplied by the amount of outgoing flow from $m$. On the other hand, nodes in 
$N\setminus (J\cup \{m\})$ are associated with the flow circulating through hub nodes, i.e., the $r$-variables, which are represented in the figure with the label ``flow''. 
Thus, provided that $i$ is not allocated to $m$, the flow with origin in $i$ will arrive to  $m$ from 
exactly one hub $k$ (because of the tree structure of the hubs). If $k\in F$, $s_{km}$ will take value 1 in the inequality; otherwise, 
$r_{ikm}$ will take the value of the outgoing flow. Therefore, the family of valid inequalities is given by

\begin{equation} \label{separadas1}
\big(\!\!\!\sum_{j\in J\cup\{m\}} w_{ij}\big)
\Big(\!\!\!\sum_{k=1:\atop \{k,m\}\in E, k \in F}^{m-1}\!\!\!\!\!\!\!\!\!s_{km} + \!\!\!\sum_{k=m+1:\atop \{k,m\}\in E, k \in F}^n \!\!\!\!\!\!\!\!\!s_{mk}\Big) + 
\sum_{k\notin F\atop \{k,m\}\in E}\!\!\!r_{ikm} \ge \sum_{j\in J\cup\{m\}} w_{ij}(z_{jm}-z_{im})
\end{equation}
for all $i,m\in N,\ F \subseteq N\setminus\{m\},\ J\subseteq N\setminus \{ i,m \}$.
Observe that the above family of valid inequalities, \eqref{separadas1}, is of exponential size. The interested is referred to \cite{contreras} in which the authors describe a separation procedure to generate violated constraints of the family \eqref{separadas1}. 

We have observed that in some cases the separation procedure is computationally costly with respect to the gain in terms of the overall consumed time and gap. Also, in many cases the \textit{optimal} sets of nodes $J$ and $F$ are singletons. Hence, we implemented a simple strategy, based on the above, that allows us to find sets in the form $J=\{j\}$ and $F=\{k\}$ with maximal violation of \eqref{separadas1} given $i, m \in N$.

Let $i, m \in N$ and $j \in N\backslash\{i,m\}$. Hence, the goal is to find $k \neq m$ such that the following inequalities are maximally violated:
\begin{equation}\label{sep1}
\big(w_{ij}+w_{im}\big)  s_{\min\{k,m\}\max\{k,m\}} +\!\!\!\! \dsum_{\ell \in N\backslash\{m,k\}} r_{i \ell m} \geq w_{ij}(z_{jm}-z_{im}) + w_{im}(z_{mm}-z_{im}).
\end{equation}
Let us denote $G_{ijm}=w_{ij}(z_{jm}-z_{im}) + w_{im}(z_{mm}-z_{im})$, $\alpha_{ijm}=w_{im}+w_{ij}$, $R_{im}^+=\sum_{\ell=m+1}^n r_{i\ell m}$ (with $R^+_{in}=0$) and $R_{im}^-=\sum_{\ell=1}^{m-1} r_{i\ell m}$ (with $R^-_{i1}=0$). The amount $\Gamma_{ijm}=G_{ijm} - R_{im}^+ - R^-_{im}$ does not depend of $k$. Hence, the minimum of the $n-1$ amounts in the sets  $\{\alpha_{ikm}s_{\ell m} - r_{i\ell m}: \ell <m\}$ and $\{\alpha_{ikm}s_{m\ell} - r_{i\ell m}: \ell >m\}$, if negative, allows us to construct the maximum violated inequality. That is, if $\ell=k$ is the index reaching the minimum amount of the above, if $\Gamma_{ijm} - \alpha_{ikm}s_{\min\{k,m\} \max\{k,m\}} - r_{ikm} <0$, the above constraints are violated and the new constraint can be added. Otherwise, all the constraints of the form \eqref{sep1} for those given $i, j$ and $m$ are verified.

\subsection{Preliminary Experiments}
\label{ss:exp0}

We have performed a series of preliminary experiments to test the MILP formulation \eqref{thlpu1} as well as the effect of the valid inequalities \eqref{separadas1}.  We have tested the model in a set of instances commonly used in the hub location literature, namely AP (Australian Post) and CAB (Civil Aeronautics Board), which are available at  \url{people.brunel.ac.uk/~mastjjb/jeb/orlib/phubinfo.html}. These instances consist of a distance matrix between cities in Australia (AP) and the United States (CAB), as well as a O-D flow matrix.
The models were coded in \texttt{Python 3.6}, and solved using \texttt{Gurobi 7.51} in a Mac OSX El Capitan with an Intel Core i7
processor at 3.3 GHz and 16GB of RAM. 

We construct the test instances following a similar structure that in \cite{contreras}. The number of nodes, $n$, ranges in $\{10,20,25\}$ for the AP dataset and in $\{10,15,20,25\}$ for the CAP dataset. The number of nodes $p$ ranges in $\{3,5,8\}$ (with $p<n$) and $q$, the number of upgraded hubs, in $\{1,3,5,8\}$ (with $q<p$). In order to use the information provided in the instances, the basic costs, i.e., the $d$-parameters (Euclidean distances between pairs of nodes) are reduced by an adequate factor. We denote by $\alpha$ the discount factor for connection between non-upgraded nodes, $\rho$ the discount factor between an upgraded node and a non-upgraded node and $\gamma$ the discount factor between two upgraded nodes. $\alpha$,  $\rho$ and $\gamma$ range in $\{0.2,0.5,0.8\}$ and such that $\alpha \geq \rho \geq \gamma$ with any of the inequalities strict to avoid running the standard THLP. With these settings, we have solved 168 instances of the CAB dataset and 126 for the AP dataset.

In tables \ref{cab0} and \ref{ap0}, we report the results of running in Gurobi both \eqref{thlpu1} and \eqref{thlpu1} with the separation procedure to add valid inequalities of the family \eqref{separadas1} (\eqref{thlpu1}+VI) for the CAB and the AP dataset, respectively. For the two procedures, we report the average duality gaps ({\sc GAP} and {\sc GAP}$_{VI}$),
the number of nodes of the branching tree ({\sc Nodes} and {\sc Nodes}$_{VI}$),  the CPU times, in seconds, needed to solve the
instances ({\sc Time} and {\sc Time}$_{VI}$), the number of valid inequalities added in the separation procedure ({\sc CUTS}) and the percentage of unsolved instances   ({\sc UnS} and {\sc UnS}$_{VI}$). In the implementation, the default Gurobi cuts were disabled and a time limit of 2 hours was considered to solve the problems. In case of reaching the time limit without optimally solving the problem, the GAP is the one with respect the best solution found.

Concerning the generation of valid inequalities of the family \eqref{sep1} and its separation procedure, we run the LP relaxed model a maximum of $10$ times. Then, for each $i, j$ and $k$ we perform the detection of violated inequalities and add the maximum violated one. We limited to $100$ the number of new cuts added to the model (we have observed that a larger number of cuts highly increases the consumed CPU time while the decreasing of the LP gap is small). A gap between consecutive lower bounds smaller that $1\%$ stops the separation procedure.

\renewcommand{\tabcolsep}{0.1cm}
{\small
\begin{table}[h]\label{t:gaps}
\centering\begin{tabular}{|c|c|c||cc|rr|rr|c|cc|}\hline
$n$ & $p$ & $q$ &{\sc GAP} & {\sc GAP}$_{VI}$ & {\sc Nodes} & {\sc Nodes}$_{VI}$ & {\sc Time} & {\sc Time}$_{VI}$ & {\sc Cuts} & {\sc UnS} & {\sc UnS}$_{VI}$\\\hline\hline
\multirow{6}{*}{10} & 3  &  1 & 14.59\% & 0.91\%  & 55  &  98   &    0.67  &  0.86  &  44.14   &   0\%   &   0\%\\\cline{2-12}
& \multirow{2}{*}{5}  &  1 & 30.17\% & 2.35\%  & 901  &  2456   &    2.66  &  5.85  &  63.29   &   0\%   &   0\%\\
   &    &  3 & 18.08\% & 1.8\%  & 434  &  824   &    0.91  &  2.18  &  78.86   &   0\%   &   0\%\\\cline{2-12}
&\multirow{3}{*}{8}  &  1 & 41.06\% & 7.66\%  & 38897  &  14418   &    62.19  &  31.28  &  92.29   &   0\%   &   0\%\\
   &    &  3 & 31.94\% & 4.45\%  & 9807  &  1324   &    10.26  &  3.47  &  93.14   &   0\%   &   0\%\\
   &    &  5 & 22.26\% & 7.51\%  & 50421  &  3978   &    39.89  &  10.25  &  100   &   0\%   &   0\%\\\hline
\multirow{6}{*}{15} & 3  &  1 & 12.86\% & 0.57\%  & 66  &  42   &    7.23  &  8.38  &  65.57   &   0\%   &   0\%\\\cline{2-12}
& \multirow{2}{*}{5}  &  1 & 26.1\% & 2.3\%  & 2210  &  1082   &    54.32  &  34.98  &  77.86   &   0\%   &   0\%\\
   &    &  3 & 14.19\% & 1.75\%  & 730  &  1765   &    11.7  &  66.74  &  87.14   &   0\%   &   0\%\\\cline{2-12}
&\multirow{3}{*}{8}  &  1 & 37.95\% & 6.1\%  & 106815  &  60480   &    2475.53  &  1581.08  &  100   &   28.57\%   &   0\%\\
   &    &  3 & 29.22\% & 4.06\%  & 18728  &  2184   &    186.98  &  62.72  &  100   &   0\%   &   0\%\\
   &    &  5 & 21.75\% & 6.71\%  & 164716  &  16187   &    1352.03  &  438.23  &  100   &   0\%   &   0\%\\\hline
\multirow{6}{*}{20} & 3  &  1 & 14.21\% & 1.07\%  & 69  &  219   &    55.34  &  93.33  &  45   &   0\%   &   0\%\\\cline{2-12}
& \multirow{2}{*}{5}  &  1 & 27.22\% & 4\%  & 8121  &  6517   &    1637.18  &  1770.63  &  95.43   &   0\%   &   0\%\\
   &    &  3 & 16.6\% & 3.38\%  & 2018  &  3559   &    361.92  &  1105.03  &  100   &   0\%   &   0\%\\\cline{2-12}
&\multirow{3}{*}{8}  &  1 & 36.95\% & 6.19\%  & 13080  &  12297   &    3213.18  &  3361.7  &  100   &   42.86\%   &   42.86\%\\
   &    &  3 & 27.96\% & 6.85\%  & 24414  &  18452   &    3077.46  &  3701.57  &  100   &   28.57\%   &   28.57\%\\
   &    &  5 & 19.44\% & 7.67\%  & 70308  &  24390   &    4387.13  &  5171.51  &  100   &   57.14\%   &   57.14\%\\\hline
\multirow{6}{*}{25} & 3  &  1 & 12.69\% & 0.79\%  & 78  &  318   &    269.22  &  546.01  &  45.14   &   0\%   &   0\%\\\cline{2-12}
& \multirow{2}{*}{5}  &  1 & 25.47\% & 4.47\%  & 2727  &  2852   &    3561.09  &  5337.84  &  98.57   &   42.86\%   &   71.43\%\\
   &    &  3 & 14.38\% & 3.51\%  & 4048  &  2914   &    2846.83  &  4385.51  &  100   &   14.29\%   &   42.86\%\\\cline{2-12}
&\multirow{3}{*}{8}  &  1 & 37.13\% & 8.4\%  & 4383  &  4679   &    5193.18  &  6958.07  &  100   &   71.43\%   &   85.71\%\\
   &    &  3 & 26.14\% & 6.12\%  & 8086  &  6714   &    5528.64  &  6763.24  &  100   &   57.14\%   &   71.43\%\\
   &    &  5 & 19.11\% & 7.07\%  & 16258  &  6165   &     $>$7200  &  5874.18  &  100   &   100\%   &   71.43\%\\\hline
      \end{tabular}
\caption{Average Results for the CAB dataset using \eqref{thlpu1}.\label{cab0}}
\end{table}}

{\small
\begin{table}[h]\label{t:gaps}
\centering\begin{tabular}{|c|c|c||cc|rr|rr|c|cc|}\hline
$n$ & $p$ & $q$ &{\sc GAP} & {\sc GAP}$_{VI}$ & {\sc Nodes} & {\sc Nodes}$_{VI}$ & {\sc Time} & {\sc Time}$_{VI}$ & {\sc Cuts} & {\sc UnS} & {\sc UnS}$_{VI}$\\\hline\hline
\multirow{6}{*}{10} & 3  &  1 & 13.58\% & 1.94\%  & 107  &  105   &    2.47  &  2.35  &  72.43   &   0\%   &   0\%\\\cline{2-12}
& \multirow{2}{*}{5}  &  1 & 25.15\% & 4.68\%  & 2558  &  3061   &    23.5  &  25.67  &  82.57   &   0\%   &   0\%\\
   &    &  3 & 15.94\% & 3.78\%  & 1038  &  1122   &    6.69  &  13.39  &  94.43   &   0\%   &   0\%\\\cline{2-12}
&\multirow{3}{*}{8}  &  1 & 39.91\% & 8.86\%  & 94849  &  22573   &    551.74  &  170.9  &  90.71   &   0\%   &   0\%\\
   &    &  3 & 31.72\% & 8.72\%  & 43448  &  6952   &    211.66  &  50.85  &  100   &   0\%   &   0\%\\
   &    &  5 & 24.73\% & 9.76\%  & 121932  &  6378   &    415.57  &  46.95  &  100   &   0\%   &   0\%\\\hline
\multirow{6}{*}{20} & 3  &  1 & 11.3\% & 1.97\%  & 247  &  434   &    229.23  &  328.59  &  83.71   &   0\%   &   0\%\\\cline{2-12}
& \multirow{2}{*}{5}  &  1 & 19.46\% & 4.51\%  & 3437  &  5599   &    3347.99  &  3924.91  &  100   &   42.86\%   &   28.57\%\\
   &    &  3 & 9.8\% & 2.54\%  & 3438  &  2434   &    1765.75  &  2136.17  &  100   &   0\%   &   0\%\\\cline{2-12}
&\multirow{3}{*}{8}  &  1 & 27.41\% & 6.64\%  & 6950  &  6553   &    4889.96  &  5267.49  &  100   &   57.14\%   &   42.86\%\\
   &    &  3 & 18.86\% & 5.15\%  & 15215  &  8451   &    5422.1  &  6950.07  &  100   &   42.86\%   &   85.71\%\\
   &    &  5 & 14.04\% & 5.57\%  & 26174  &  9799   &    6952.1  &  6321.67  &  100   &   85.71\%   &   71.43\%\\\hline
\multirow{6}{*}{25} & 3  &  1 & 10.22\% & 1.89\%  & 355  &  738   &    1435.73  &  2513.64  &  85   &   0\%   &   14.29\%\\\cline{2-12}
& \multirow{2}{*}{5}  &  1 & 18.83\% & 4.39\%  & 1018  &  1359   &    4711.68  &  6070.7  &  98   &   57.14\%   &   71.43\%\\
   &    &  3 & 9.97\% & 2.74\%  & 2267  &  1614   &    3960.22  &  4638.89  &  100   &   14.29\%   &   57.14\%\\\cline{2-12}
&\multirow{3}{*}{8}  &  1 & 27.32\% & 8.24\%  & 886  &  1543   &    5259.53  &   $>$7200  &  100   &   71.43\%   &   100\%\\
   &    &  3 & 18.26\% & 6.22\%  & 3074  &  1971   &     $>$7200  &  $>$7200  &  100   &   100\%   &   100\%\\
   &    &  5 & 12.7\% & 6.05\%  & 4108  &  2450   &    $>$7200  &   $>$7200  &  100   &   100\%   &   100\%\\\hline
      \end{tabular}
\caption{Average Results for the AP dataset using \eqref{thlpu1}.\label{ap0}}
\end{table}}

The first observation that comes after running the experiments is that solving the THLPU is not an easy task. Even for a small number of nodes, the problems were very time consuming. Actually, for the CAB dataset 18.5\% (for \eqref{thlpu1}) and 19.6\% (for \eqref{thlpu1}+VI) of the instances were not optimally solved within the time limit, while for the AP dataset the percentages of unsolved instances are 31.7\% and 37.3\% respectively. Note that, although according to \cite{contreras}, the set of valid inequalities \eqref{separadas1} has a good performance when applied to the THLP, the gain obtained applying them to the THLPU is only partial: while the LP gaps are significatively smaller when a subset of \eqref{separadas1} is incorporated to the model (an average difference of 19.66\% in the CAB dataset and 14.20\% for the AP dataset when comparing the two strategies), it does not always reduce the CPU times needed by Gurobi to solve the MILP problem. Actually, in some cases the LP gap is smaller than $1\%$, but still the solver takes a long time to check optimality of the solution. In fact, only $6$ of instances of CAB and $4$ of AP dataset were optimally solved adding the valid inequalites but not without them. Indeed, $3$ instances of CAB and $7$ of AP were not solved adding the valid inequalities, but they were solved without including them. Concerning the CPU times, in only $55\%$ of the CAB instances and $60\%$ of the AP instances, the CPU times needed to solve the problems with the valid inequalities are smaller than not using them.

Thus, based on the results, the set of inequalities \eqref{separadas1} (in its simplest form \eqref{sep1}) highly strengthen the MILP formulation \eqref{thlpu1} for the Tree of Hub Location problem with Upgrading, but such a strengthen is not reflected in the CPU times needed for solving the problems. In the next section we provide a different MILP formulation for  the problem, in order to check whether better results can be obtained.

\section{A disaggregated model for THLPU}\label{sec:4}

In this section a different idea is exploited to model the THLPU, by using some of the variables previously used in \eqref{thlpu1}, but disaggregating the flow variables into  
some others that allow us to represent the different types of costs of traversing hub nodes. In particular, we will use the $z$ and $t$ variables in this model, in the same manner they were defined for \eqref{thlpu1} (see \eqref{zvars} and \eqref{tvars}).

Additional binary variables associated to edges between two hubs are used, instead of $r$ and $s$ as for \eqref{thlpu1}, to define the small tree and also to keep track of the type of discount to be applied to the flow. We define three different families of binary variables which are closely related to the $s$-variables in \eqref{thlpu1}, but in which we distinguish 
 between the type of hub nodes that the extremes of the edges in the small tree are. For $\{k,m\} \in E$ with $k\neq m$, we denote:

 $$
y_{km}= \left\{\begin{array}{cl}
1 & \mbox{if $k$ and $m$ are adjacent but non-upgraded hubs,}\\
0 &\mbox{otherwise,}
\end{array}\right.
$$
$$
y'_{km}= \left\{\begin{array}{cl}
1 & \mbox{if $k$ or $m$ are adjacent hubs, but only one is upgraded,}\\
0 &\mbox{otherwise,}
\end{array}\right.
$$
$$
y''_{km}= \left\{\begin{array}{cl}
1 & \mbox{if $k$ or $m$ are upgraded hubs, and they are adjacent,}\\
0 &\mbox{otherwise.}
\end{array}\right.
$$

Figure \ref{y} shows those variables in these three families taking value 1 according with 
the depicted small tree and chosen upgraded or not hubs.

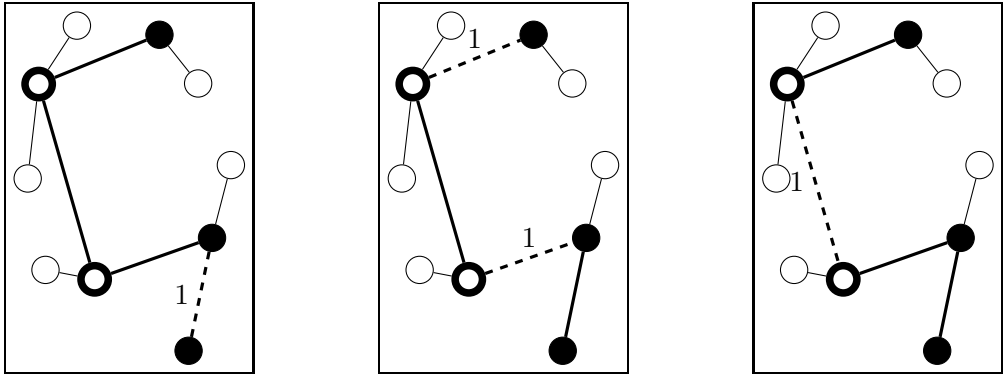
\begin{figure}[H]
\begin{center}
\begin{center}
\begin{subfigure}[b]{0.35\textwidth}
\fbox{\begin{tikzpicture}[scale=0.45]
\coordinate(X1) at (8.426952,0.357648);
\coordinate(X2) at (7.566704,9.700764);
\coordinate(X3) at (9.119774,3.694816);
\coordinate(X4) at (5.128467,9.970362);
\coordinate(X5) at (3.674127,5.450939);
\coordinate(X6) at (8.709690,8.260512);
\coordinate(X7) at (5.651874,2.469880);
\coordinate(X8) at (9.677091,5.847142);
\coordinate(X9) at (3.998395,8.242882);
\coordinate(X10) at (4.199842,2.748433);
%\draw (X1) circle (10pt);
\node[circle,draw,fill](X-1) at (X1) {};
\node[circle,draw,fill](X-2) at (X2) {};
\node[circle,draw,fill](X-3) at (X3) {};
\node[circle,draw](X-4) at (X4) {};
\node[circle,draw](X-5) at (X5) {};
\node[circle,draw](X-6) at (X6) {};
\node[circle,draw,line width=1mm](X-7) at (X7) {};
\node[circle,draw](X-8) at (X8) {};
\node[circle,draw,line width=1mm](X-9) at (X9) {};
\node[circle,draw](X-10) at (X10) {};

\draw[very thick] (X-9)--(X-7);
\draw[very thick] (X-9)--(X-2);
\draw[very thick] (X-7)--(X-3);
\draw[very thick,dashed] (X-1)--node [midway, above,left]{$1$}(X-3);

\draw (X-4) -- (X-9);
\draw (X-5) --  (X-9);
\draw (X-2) --  (X-6);
\draw (X-8) --  (X-3);
\draw (X-10) -- (X-7);

\end{tikzpicture}}
\end{subfigure}~\begin{subfigure}[b]{0.35\textwidth}
\fbox{\begin{tikzpicture}[scale=0.45]
\coordinate(X1) at (8.426952,0.357648);
\coordinate(X2) at (7.566704,9.700764);
\coordinate(X3) at (9.119774,3.694816);
\coordinate(X4) at (5.128467,9.970362);
\coordinate(X5) at (3.674127,5.450939);
\coordinate(X6) at (8.709690,8.260512);
\coordinate(X7) at (5.651874,2.469880);
\coordinate(X8) at (9.677091,5.847142);
\coordinate(X9) at (3.998395,8.242882);
\coordinate(X10) at (4.199842,2.748433);
%\draw (X1) circle (10pt);
\node[circle,draw,fill](X-1) at (X1) {};
\node[circle,draw,fill](X-2) at (X2) {};
\node[circle,draw,fill](X-3) at (X3) {};
\node[circle,draw](X-4) at (X4) {};
\node[circle,draw](X-5) at (X5) {};
\node[circle,draw](X-6) at (X6) {};
\node[circle,draw,line width=1mm](X-7) at (X7) {};
\node[circle,draw](X-8) at (X8) {};
\node[circle,draw,line width=1mm](X-9) at (X9) {};
\node[circle,draw](X-10) at (X10) {};

\draw[very thick] (X-9)--(X-7);
\draw[very thick,dashed] (X-9)--node [midway, above]{$1$}(X-2);
\draw[very thick,dashed] (X-7)--node [midway, above]{$1$}(X-3);
\draw[very thick] (X-1)--(X-3);

\draw (X-4) --  (X-9);
\draw (X-5) --   (X-9);
\draw (X-2) --   (X-6);
\draw (X-8) --   (X-3);
\draw (X-10) --  (X-7);

\end{tikzpicture}}
\end{subfigure}~\begin{subfigure}[b]{0.35\textwidth}
\fbox{\begin{tikzpicture}[scale=0.45]
\coordinate(X1) at (8.426952,0.357648);
\coordinate(X2) at (7.566704,9.700764);
\coordinate(X3) at (9.119774,3.694816);
\coordinate(X4) at (5.128467,9.970362);
\coordinate(X5) at (3.674127,5.450939);
\coordinate(X6) at (8.709690,8.260512);
\coordinate(X7) at (5.651874,2.469880);
\coordinate(X8) at (9.677091,5.847142);
\coordinate(X9) at (3.998395,8.242882);
\coordinate(X10) at (4.199842,2.748433);
%\draw (X1) circle (10pt);
\node[circle,draw,fill](X-1) at (X1) {};
\node[circle,draw,fill](X-2) at (X2) {};
\node[circle,draw,fill](X-3) at (X3) {};
\node[circle,draw](X-4) at (X4) {};
\node[circle,draw](X-5) at (X5) {};
\node[circle,draw](X-6) at (X6) {};
\node[circle,draw,line width=1mm](X-7) at (X7) {};
\node[circle,draw](X-8) at (X8) {};
\node[circle,draw,line width=1mm](X-9) at (X9) {};
\node[circle,draw](X-10) at (X10) {};

\draw[very thick,dashed] (X-9)--node [midway, above,left]{$1$} (X-7);
\draw[very thick] (X-9)--(X-2);
\draw[very thick] (X-7)--(X-3);
\draw[very thick] (X-1)--(X-3);

\draw (X-4) --  (X-9);
\draw (X-5) --   (X-9);
\draw (X-2) --   (X-6);
\draw (X-8) --   (X-3);
\draw (X-10) --  (X-7);

\end{tikzpicture}}
\end{subfigure}
\end{center}
\caption{Values for the $y$ (left), $y'$ (center) and $y''$ (right) variables when defining the small tree \label{y}}
\end{center}
\end{figure}

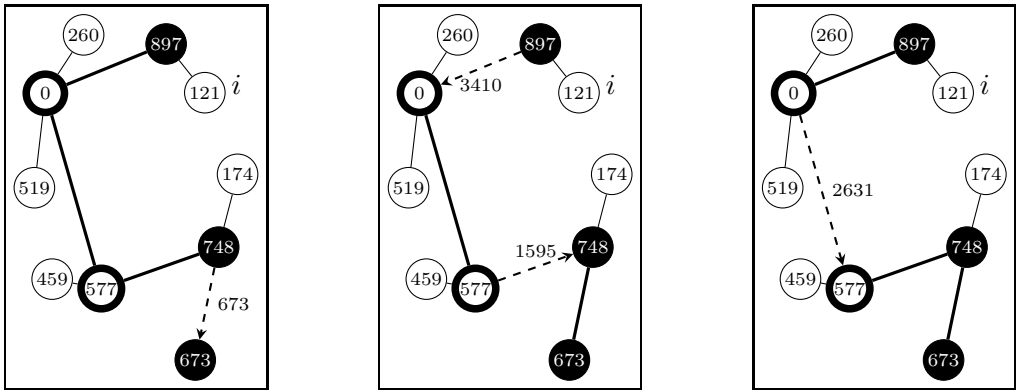
\begin{figure}[H]
\begin{center}
\begin{center}
\begin{subfigure}[b]{0.35\textwidth}
\fbox{\begin{tikzpicture}[scale=0.45]
\tikzstyle{arrow} = [thick,scale=3,->,>=stealth]
\coordinate(X1) at (8.426952,0.357648);
\coordinate(X2) at (7.566704,9.700764);
\coordinate(X3) at (9.119774,3.694816);
\coordinate(X4) at (5.128467,9.970362);
\coordinate(X5) at (3.674127,5.450939);
\coordinate(X6) at (8.709690,8.260512);
\coordinate(X7) at (5.651874,2.469880);
\coordinate(X8) at (9.677091,5.847142);
\coordinate(X9) at (3.998395,8.242882);
\coordinate(X10) at (4.199842,2.748433);
%\draw (X1) circle (10pt);
\node[circle,draw,fill,inner sep=1pt](X-1) at (X1) {\tiny \color{white}$673$};
\node[circle,draw,fill,inner sep=1pt](X-2) at (X2) {\tiny \color{white}$897$};
\node[circle,draw,fill,inner sep=1pt](X-3) at (X3) {\tiny \color{white}$748$};
\node[circle,draw,inner sep=1pt](X-4) at (X4) {\tiny $260$};
\node[circle,draw,inner sep=1pt](X-5) at (X5) {\tiny $519$};
\node[circle,draw,inner sep=1pt](X-6) at (X6) {\tiny $121$};
\node[circle,draw,line width=1mm,inner sep=1pt](X-7) at (X7) {\tiny $577$};
\node[circle,draw,inner sep=1pt](X-8) at (X8) {\tiny $174$};
\node[circle,draw,line width=1mm, inner sep=3pt](X-9) at (X9) {\tiny $0$};
\node[circle,draw,inner sep=1pt](X-10) at (X10) {\tiny $459$};

\node[above,right] at (9.2,8.5) {$i$};

\draw[very thick] (X-2) --  (X-9);
\draw[very thick,arrow,dashed] (X-3) -- node [right]{\tiny $673$} (X-1);
\draw[very thick] (X-7) -- (X-3);
\draw[very thick] (X-9) -- (X-7);

\draw (X-4)--(X-9);
\draw (X-5)--(X-9);
\draw (X-2)--(X-6);
\draw (X-8)--(X-3);
\draw (X-10)--(X-7);
\end{tikzpicture}}\end{subfigure}~\begin{subfigure}[b]{0.35\textwidth}
\fbox{\begin{tikzpicture}[scale=0.45]
\tikzstyle{arrow} = [thick,scale=3,->,>=stealth]
\coordinate(X1) at (8.426952,0.357648);
\coordinate(X2) at (7.566704,9.700764);
\coordinate(X3) at (9.119774,3.694816);
\coordinate(X4) at (5.128467,9.970362);
\coordinate(X5) at (3.674127,5.450939);
\coordinate(X6) at (8.709690,8.260512);
\coordinate(X7) at (5.651874,2.469880);
\coordinate(X8) at (9.677091,5.847142);
\coordinate(X9) at (3.998395,8.242882);
\coordinate(X10) at (4.199842,2.748433);
%\draw (X1) circle (10pt);
\node[circle,draw,fill,inner sep=1pt](X-1) at (X1) {\tiny \color{white}$673$};
\node[circle,draw,fill,inner sep=1pt](X-2) at (X2) {\tiny \color{white}$897$};
\node[circle,draw,fill,inner sep=1pt](X-3) at (X3) {\tiny \color{white}$748$};
\node[circle,draw,inner sep=1pt](X-4) at (X4) {\tiny $260$};
\node[circle,draw,inner sep=1pt](X-5) at (X5) {\tiny $519$};
\node[circle,draw,inner sep=1pt](X-6) at (X6) {\tiny $121$};
\node[circle,draw,line width=1mm,inner sep=1pt](X-7) at (X7) {\tiny $577$};
\node[circle,draw,inner sep=1pt](X-8) at (X8) {\tiny $174$};
\node[circle,draw,line width=1mm, inner sep=3pt](X-9) at (X9) {\tiny $0$};
\node[circle,draw,inner sep=1pt](X-10) at (X10) {\tiny $459$};

\node[above,right] at (9.2,8.5) {$i$};

\draw[very thick,arrow,dashed] (X-2) -- node [below]{\tiny $3410$} (X-9);
\draw[very thick] (X-3) --  (X-1);
\draw[very thick,arrow,dashed] (X-7) -- node [above]{\tiny $1595$} (X-3);
\draw[very thick] (X-9) --  (X-7);

\draw (X-4)--(X-9);
\draw (X-5)--(X-9);
\draw (X-2)--(X-6);
\draw (X-8)--(X-3);
\draw (X-10)--(X-7);
\end{tikzpicture}}\end{subfigure}~\begin{subfigure}[b]{0.35\textwidth}
\fbox{\begin{tikzpicture}[scale=0.45]
\tikzstyle{arrow} = [thick,scale=3,->,>=stealth]
\coordinate(X1) at (8.426952,0.357648);
\coordinate(X2) at (7.566704,9.700764);
\coordinate(X3) at (9.119774,3.694816);
\coordinate(X4) at (5.128467,9.970362);
\coordinate(X5) at (3.674127,5.450939);
\coordinate(X6) at (8.709690,8.260512);
\coordinate(X7) at (5.651874,2.469880);
\coordinate(X8) at (9.677091,5.847142);
\coordinate(X9) at (3.998395,8.242882);
\coordinate(X10) at (4.199842,2.748433);
%\draw (X1) circle (10pt);
\node[circle,draw,fill,inner sep=1pt](X-1) at (X1) {\tiny \color{white}$673$};
\node[circle,draw,fill,inner sep=1pt](X-2) at (X2) {\tiny \color{white}$897$};
\node[circle,draw,fill,inner sep=1pt](X-3) at (X3) {\tiny \color{white}$748$};
\node[circle,draw,inner sep=1pt](X-4) at (X4) {\tiny $260$};
\node[circle,draw,inner sep=1pt](X-5) at (X5) {\tiny $519$};
\node[circle,draw,inner sep=1pt](X-6) at (X6) {\tiny $121$};
\node[circle,draw,line width=1mm,inner sep=1pt](X-7) at (X7) {\tiny $577$};
\node[circle,draw,inner sep=1pt](X-8) at (X8) {\tiny $174$};
\node[circle,draw,line width=1mm, inner sep=3pt](X-9) at (X9) {\tiny $0$};
\node[circle,draw,inner sep=1pt](X-10) at (X10) {\tiny $459$};

\node[above,right] at (9.2,8.5) {$i$};

\draw[very thick] (X-2) --  (X-9);
\draw[very thick] (X-3) --  (X-1);
\draw[very thick] (X-7) -- (X-3);
\draw[very thick,arrow,dashed] (X-9) -- node [above,right]{\tiny $2631$} (X-7);

\draw (X-4)--(X-9);
\draw (X-5)--(X-9);
\draw (X-2)--(X-6);
\draw (X-8)--(X-3);
\draw (X-10)--(X-7);
\end{tikzpicture}}\end{subfigure}
\end{center}
\caption{Values for the $x$ (left), $x'$ (center) and $x''$ (right) variables associated with a given origin $i$.
The amount inside any node $j$ is the value of $w_{ij}$\label{x}}
\end{center}
\end{figure}

We also consider, instead of the flow $r$-variables in \eqref{thlpu1}, three new sets of variables which are constructed by splitting the flow traversing hub edges, by differentiating again between the three types of connections between hub nodes. For $\{k,m\}\in E$ with $k\neq m$, we define:

\begin{itemize}
\item $x_{ikm}$: amount of flow with origin in node $i$ which traverses arc $(k,m)$
if $k$ and $m$ are both hubs, but neither $k$ nor $m$ have been upgraded,
\item $x'_{ikm}$: amount of flow with origin in node $i$ which traverses arc $(k,m)$
if $k$ and $m$ are both hubs and $k$ or $m$ (only one of them) has been upgraded, 
\item $x''_{ikm}$: amount of flow with origin in node $i$ which traverses arc $(k,m)$
if $k$ and $m$ are both hubs and both $k$ and $m$ have been upgraded. 
\end{itemize}
Figure \ref{x} shows the values of these three families of variables (if they are not 0) 
for a fixed value of $i$.

With the above notation, we present now the disaggregated formulation for THLPU:

\begin{align}
 \min &\;\; \sum_{i\in N} \sum_{k=1:\atop k\neq i}^n (O_i d_{ik} + D_i d_{ki}) z_{ik} \label{thlpu2}\tag{${\rm DTHLPU}$}  \\ 
     & \hspace*{2cm}+  \sum_{i\in N} \sum_{k\in N} \sum_{m\in N:\atop m\neq k} 
     (c_{km} x_{ikm} + c'_{km} x'_{ikm} + c''_{km} x''_{ikm})\nonumber\\
\mbox{s.t. } 
&  \sum_{k\in N} z_{kk} = p,   \label{a} \\
&  \sum_{k=1}^{n-1} \sum_{m=k+1:\atop \{k,m\}\in E}^n (y_{km} + y'_{km} + y''_{km}) = p-1, \label{st} \\
&  \sum_{k\in N} z_{ik} = 1, \hspace{0.3cm} \forall i\in N,  \label{b} \\
&  y_{km} + y'_{km} + y''_{km} + z_{mk} \le z_{kk}, \hspace{0.3cm} \forall k<m\in N:\ \{k,m\}\in E, \label{ca} \\
&  y_{km} + y'_{km} + y''_{km} + z_{km} \le z_{mm}, \hspace{0.3cm} \forall k<m\in N:\ \{k,m\}\in E,\label{cb} \\ 
&  \sum_{k\in N} t_k = q   \label{d}, \\
&  t_k \le z_{kk}, \hspace{0.3cm}  \forall k\in N, \label{e} \\
&  y''_{km} \le t_k, \hspace{0.3cm} \forall \{k,m\}\in E, k<m, \label{f2} \\
&  y''_{km} \le t_m, \hspace{0.3cm} \forall \{k,m\}\in E, k<m, \label{f3} \\
&	 y'_{km} + y''_{km} \le t_k + t_m, \hspace{0.3cm} \forall \{k,m\}\in E, k<m, \label{f4}\\
&  x_{ikm} + x_{imk} \le O_{ikm} y_{km}, \hspace{0.3cm} \forall i\in N,  \forall \{k,m\}\in E, i\neq k<m\in N,  \label{g1a} \\
&  x_{iim}  \le  O_{ikm} y_{\min \{i,m\},\max\{ i,m\}}, \hspace{0.3cm} \forall \{i,m\}\in E, \label{g1b} \\
&  x'_{ikm} + x'_{imk} \le O_{ikm} y'_{km}, \hspace{0.3cm} \forall i\in N,  \forall \{k,m\}\in E, i\neq k<m\in N,  \label{g2a}
\end{align}
\begin{align}
&  x'_{iim}  \le O_{ikm} y'_{\min \{i,m\},\max\{ i,m\}}, \hspace{0.3cm} \forall \{i,m\}\in E, \label{g2b}\\
&  x''_{ikm} + x''_{imk} \le O_{ikm} y''_{km}, \hspace{0.3cm} \forall i\neq k<m\in N:\ \{k,m\}\in E, \label{g3a} \\
&  x''_{iim}  \le O_{ikm} y''_{\min \{i,m\},\max\{ i,m\}} \hspace{0.3cm}, \forall \{i,m\}\in E, \label{g3b}\\
&  O_i z_{ik} + \sum_{m=1:\atop m\neq k}^n\!(x_{imk} + x'_{imk} + x''_{imk}) =&\nonumber\\ 
& \hspace*{3.5cm} \sum_{m=1:\atop m\neq k}^n \! (x_{ikm} + x'_{ikm} + x''_{ikm}) + \sum_{j\in N} \! w_{ij} z_{jk},  \forall i,k\in N, \label{h}\\
&  z_{ik}, t_k\in \{0,1\} ,\hspace{0.3cm} \forall i, k\in N, \nonumber \\
&  y_{km},y'_{km},y''_{km}\in \{0,1\}, \hspace{0.3cm} \forall k<m\in N:\ \{k,m\}\in E, \nonumber \\
&  x_{km},x'_{km},x''_{km}\ge 0, \hspace{0.3cm} \forall k\neq m\in N.\nonumber
\end{align}

In the objective function of \eqref{thlpu2} the discounts are applied to the flow between hubs  given by the $x$-, $x'$- and $x''$-variables, and the total cost of the flow with origin and destination  in each non-hub node $i$ is added up without any discount. 

Constraints \eqref{a} and \eqref{st} fix the number of hubs and edges in the small tree.  These two constraints, plus the connection, forced by the flows between nodes, ensure that  the resulting structure will be a tree.

Constraints \eqref{b}, \eqref{ca} and \eqref{cb} guarantee that each non-hub node 
is allocated to a hub. Constraints \eqref{ca} and \eqref{cb} also fix to zero 
$y$-, $y'$- and $y''$-variables when one or two of their extremes are not hubs. 

Regarding the $t$-variables, constraints \eqref{d} establish in $q$ the number of upgraded 
nodes, whereas \eqref{e} ensure that the upgraded nodes will be hubs. Once the upgrading is 
known, \eqref{f2}-\eqref{f4} fix to zero $y$-, $y'$- and $y''$-variables when the extremes of 
the edge have not the adequate upgrading. Similarly, constraints \eqref{g1a}-\eqref{g3b} fix
to zero the $x$-, $x'$- and $x''$-variables when appropriate, in this case using an 
upper bound $O_{ikm}$ since they are continuous variables.

Observe that given a feasible solution of \eqref{thlpu1} one can easily construct a feasible solution of \eqref{thlpu2}, and vice versa. In particular, given feasible values of \eqref{thlpu2}, for $y_{km}$, $y'_{km}$, $y''_{km}$, $x_{ikm}$, $x'_{ikm}$ and $x''_{ikm}$, for $i, k, m \in N$, one can define
\begin{align*}
\bar s_{\min\{k,m\} \max\{k,m\}} &= y_{\min\{k,m\} \max\{k,m\}}+y'_{\min\{k,m\} \max\{k,m\}}+y''_{\min\{k,m\} \max\{k,m\}},\\
\bar r_{ikm} &= x_{ikm}+x'_{ikm}+x''_{ikm},\\
\bar \theta_{i\min\{k,m\} \max\{k,m\}} &= c_{km} (x_{ikm}+x_{imk}) +c'_{km} (x'_{ikm}+x'_{imk})+c''_{km} (x''_{ikm}+x''_{imk}),
\end{align*}
such that $(\bar s, \bar r, \bar \theta)$, together with the $z$ and $t$ values, is a feasible solution to \eqref{thlpu1}. In particular, one can see that \eqref{h} is nothing but the same flow conservation of flow constraint \eqref{h1} of \eqref{thlpu1}.

Several families of valid inequalities can be added to formulation \eqref{thlpu2} in order to 
reduce the size of the polyhedron associated to the linear relaxation, so improving the 
lower bounds it produces and reducing the computational times. 

\begin{lemma}\label{lem2}
The following inequalities are valid for \eqref{thlpu2}:
\begin{enumerate}
\item[\rm 1)] $x_{ikm} + x_{imk} + x'_{ikm} + x'_{imk} + x''_{ikm} + x''_{imk} \le O_{ikm} z_{kk}$,  $\forall i\neq k<m\in N:\ \{k,m\}\in E$,
\item[\rm 2)] $x_{ikm} + x_{imk} + x'_{ikm} + x'_{imk} + x''_{ikm} + x''_{imk} \le  O_{ikm} z_{mm}$, $\forall i\neq k<m\in N:\ \{k,m\}\in E$,
\item[\rm 3)] $x_{iim} + x'_{iim} + x''_{iim} \le O_{ikm} z_{ii}$, $\forall m\neq i\in N$,
\item[\rm 4)] $x_{iim} + x'_{iim} + x''_{iim} \le O_{ikm} z_{mm}$, $\forall m\neq i\in N$,
\item[\rm 5)] $y_{km} + y'_{km} + y''_{km} + z_{km} + z_{mk} \le 1$, $\forall k<m\in N:\ \{k,m\}\in E$.
\end{enumerate}
\end{lemma}

\begin{proof}

Observe that the inequalities in 1) and 2) assert that when $z_{kk}=0$ or $z_{mm}=0$, no reduction can be applied to the flow 
traversing edge $\{k,m\}$ and then the corresponding $x$-, $x'$- and $x''$-variables will take 
value $0$. Otherwise, the total amount of flow with origin in $i\neq k$ traversing $\{k,m\}$ 
in any direction will be bounded above by $O_{ikm}$.

The particular case of 1) and 2) when $i=k$ results in 3) and 4) .

Finally, 5) follows from the same construction as \eqref{ca} and \eqref{cb}. Given two nodes $k$ and $m$ in $N$, if $\{k,m\}$ is part of the large tree, only one of the following  situations may occur: i) $\{k,m\}$ is also part of the small tree (in whose case, $y_{km} + y'_{km} + y''_{km}=1$); or ii) one of $k$ and $m$ is a non-hub and the other is a hub and they are adjacent (in whose case $z_{km} + z_{mk}=1$). The inequality comes for the case in which $\{k,m\}$ is not an edge in the large tree ($k$ and $m$ are not adjacent).
 
\end{proof}

\subsection{A family of valid inequalities for \eqref{thlpu2}}

By the equivalence between the $s$ variables in \eqref{thlpu1} and the $(y,y',y'')$-variables in \eqref{thlpu2}, and also between the $r$-variables and the $(x,x',x'')$-variables, the set of valid inequalities \eqref{separadas1} can be adapted to \eqref{thlpu2}. In particular, they read:

\begin{eqnarray}
\Big(\sum_{j\in J\cup\{m\}} w_{ij}\Big) 
\Big(
\sum_{k=1:\atop {\{k,m\}\in E}}^{m-1} \!\!\!\!\left(y_{km}+y'_{km}+y''_{km}\right) + 
\sum_{k=m+1:\atop {\{k,m\}\in E}}^n \!\!\!\!\left(y_{mk}+y'_{mk}+y''_{mk}\right) \Big)+ 
\nonumber \\
\sum_{k\notin F\atop \{k,m\}\in E} \!\!\!\! \left(x_{ikm}+x'_{ikm}+x''_{ikm}\right) \geq \sum_{j\in J\cup\{m\}} \!\!\!\!w_{ij}(z_{jm}-z_{im})  \label{separadas2}
\end{eqnarray}
for all $i,m\in N,\ F \subseteq N\setminus\{m\},\ J\subseteq N\setminus \{ i,m \}$.

Since the set of variables of \eqref{thlpu2} and the valid inequalities \eqref{separadas2} differ from those in \cite{contreras}, in what follows, we explicitly describe a separation procedure for \eqref{separadas2}. 

Let $i,m\in N$, $Q\ge 0$ and $(\bar{x},\bar{y},\bar{z},\bar{x}',\bar{y}',\bar{x}'',\bar{y}'')$ 
be a feasible fractional solution of \eqref{thlpu2}. Then, an optimal solution, $\overline{F}$, that solves the following optimization problem

\begin{eqnarray*}
\mathrm{L(Q)} = \min_{F \subseteq N\backslash\{m\}} 
 \sum_{k\notin F:\atop \{k,m\}\in E} (\bar{x}_{ikm}+\bar{x}'_{ikm}+\bar{x}''_{ikm}) + 
\\
Q \Big( \sum_{k=1:\atop {\{k,m\}\in E}}^{m-1} (\bar{y}_{km} +\bar{y}'_{km}+\bar{y}''_{km}) + 
\sum_{k=m+1:\atop {\{k,m\}\in E}}^n (\bar{y}_{mk} +\bar{y}'_{mk}+\bar{y}''_{mk}
\Big)
\end{eqnarray*}

is given by 
$$\overline{F}= 
\{k<m:\ \{k,m\}\in E,\ {\bar{x}_{ikm}+\bar{x}'_{ikm}+\bar{x}''_{ikm}\over
\bar{y}_{km} +\bar{y}'_{km}+\bar{y}''_{km}  }\ge Q\}
\cup $$
$$\{k>m:\ \{k,m\}\in E,\ {\bar{x}_{ikm}+\bar{x}'_{ikm}+\bar{x}''_{ikm}\over
\bar{y}_{mk}+\bar{y}'_{mk}+\bar{y}''_{mk}   }\ge Q\}.$$
Observe that the function $\mathrm{L}$ is a piecewise linear function on the values of $Q\geq 0$, hence $\mathrm{L}$ has the following shape:
$$
\mathrm{L}(Q) = \left\{\begin{array}{cl}
a_1+ Q b_1 & \mbox{if $Q \in [A_1, A_{2}]$,}\\
\vdots & \vdots\\
a_R+ Q b_R & \mbox{if $Q \in [A_R, A_{R+1}]$.}
\end{array}\right. \quad \forall Q \geq 0.
$$

Now, the possible values of interest of $Q$ for fixed values of $i$ and $m$ are given by the 
different choices of the set $J\subset N\setminus \{ i,m\}$. Let $e=(e_1,e_2,\ldots ,e_n)\in \{ 0,1\}^n$ 
be the incidence vector of set $J$, producing a value of $Q_e:=\sum_{j\in N} e_j w_{ij}$. 
The best option for $e$ can be obtained by solving the auxiliary problem 
\begin{eqnarray}
 \min & & L(Q_e) - \sum_{j\in N} w_{ij}(\bar{z}_{jm}-\bar{z}_{im})e_j \nonumber\\
 \hbox{s.t. } & &e_m=1   \nonumber \\
              & & e_j \in \{ 0,1 \}\ \ \forall j\in N. \nonumber
\end{eqnarray}
The maximally violated inequality will be generated when the optimum of this problem is negative. For the computational experiments, as for formulation \eqref{thlpu1}, a simpler subfamily of valid inequalities of \eqref{separadas2} was considered. The inequalities and the separation procedure read exactly as those for \eqref{sep1}, by using the identification between the $r$ and the $x, x'$, and $x'''$ variables and between $s$ and the $y$, $y'$, and $y''$ variables.

As detailed for the valid inequalities and the separation procedure for \eqref{thlpu1} (see Section \ref{sec:3}), in our computational experiments we search and incorporate only those in which the sets $J$ and $F$ are singletons, because its simplicity and its relative gain in terms of strength and consumed CPU time.

\section{Computational Experiments}
\label{sec:6}
In this section we report the results of a series of computational experiments performed for solving the THLPU when using the disagregated formulation \eqref{thlpu2} and adding the family of valid inequalities \eqref{separadas2} by using our separation strategy. We use the same datasets used in Subsection \ref{ss:exp0}, and the same notation for the results. In tables \ref{cab1} and \ref{ap1} we report the average results when using \eqref{thlpu2} to solve the THLP for the CAP and AP datasets, respectively. In figures \ref{fig25a}, \ref{fig25b} and \ref{fig25c} we show some solutions obtained during our experiments for one of the instances of CAB and AP datasets.

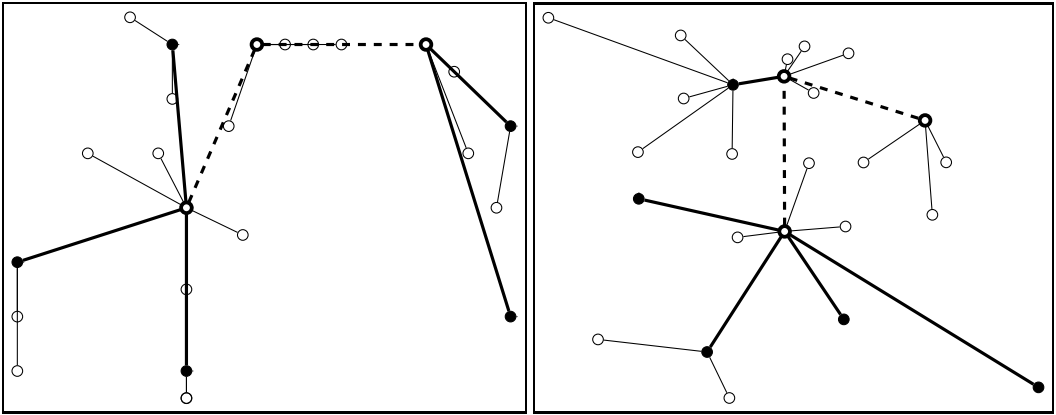
\begin{figure}
\begin{center}
\fbox{\begin{tikzpicture}[xscale=0.13, yscale=0.25]

\coordinate(X1) at (24.497000,0.000000);
\coordinate(X2) at (24.497000,0.010000);
\coordinate(X3) at (7.205000,1.448000);
\coordinate(X4) at (24.497000,1.448000);
\coordinate(X5) at (7.205000,4.344000);
\coordinate(X6) at (57.640000,4.344000);
\coordinate(X7) at (24.497000,5.792000);
\coordinate(X8) at (7.205000,7.240000);
\coordinate(X9) at (30.261000,8.688000);
\coordinate(X10) at (24.497000,10.136000);
\coordinate(X11) at (56.199000,10.136000);
\coordinate(X12) at (14.410000,13.032000);
\coordinate(X13) at (21.615000,13.032000);
\coordinate(X14) at (53.317000,13.032000);
\coordinate(X15) at (28.820000,14.480000);
\coordinate(X16) at (57.640000,14.480000);
\coordinate(X17) at (23.056000,15.928000);
\coordinate(X18) at (51.876000,17.376000);
\coordinate(X19) at (23.056000,18.824000);
\coordinate(X20) at (31.702000,18.824000);
\coordinate(X21) at (34.584000,18.824000);
\coordinate(X22) at (37.466000,18.824000);
\coordinate(X23) at (40.348000,18.824000);
\coordinate(X24) at (48.994000,18.824000);
\coordinate(X25) at (18.733000,20.272000);
\node[circle,draw,inner sep=0.5mm, inner sep=0.5mm](X-1) at (X1) {};
\node[circle,draw,inner sep=0.5mm, inner sep=0.5mm](X-2) at (X2) {};
\node[circle,draw,inner sep=0.5mm, inner sep=0.5mm](X-3) at (X3) {};
\node[circle,draw,inner sep=0.5mm,fill, inner sep=0.5mm](X-4) at (X4) {};
\node[circle,draw,inner sep=0.5mm, inner sep=0.5mm](X-5) at (X5) {};
\node[circle,draw,inner sep=0.5mm,fill, inner sep=0.5mm](X-6) at (X6) {};
\node[circle,draw,inner sep=0.5mm, inner sep=0.5mm](X-7) at (X7) {};
\node[circle,draw,inner sep=0.5mm,fill, inner sep=0.5mm](X-8) at (X8) {};
\node[circle,draw,inner sep=0.5mm, inner sep=0.5mm](X-9) at (X9) {};
\node[circle,draw,inner sep=0.5mm,line width=0.5mm, inner sep=0.5mm](X-10) at (X10) {};
\node[circle,draw,inner sep=0.5mm, inner sep=0.5mm](X-10) at (X10) {};
\node[circle,draw,inner sep=0.5mm, inner sep=0.5mm](X-11) at (X11) {};
\node[circle,draw,inner sep=0.5mm, inner sep=0.5mm](X-12) at (X12) {};
\node[circle,draw,inner sep=0.5mm, inner sep=0.5mm](X-13) at (X13) {};
\node[circle,draw,inner sep=0.5mm, inner sep=0.5mm](X-14) at (X14) {};
\node[circle,draw,inner sep=0.5mm, inner sep=0.5mm](X-15) at (X15) {};
\node[circle,draw,inner sep=0.5mm,fill, inner sep=0.5mm](X-16) at (X16) {};
\node[circle,draw,inner sep=0.5mm, inner sep=0.5mm](X-17) at (X17) {};
\node[circle,draw,inner sep=0.5mm, inner sep=0.5mm](X-18) at (X18) {};
\node[circle,draw,inner sep=0.5mm,fill, inner sep=0.5mm](X-19) at (X19) {};
\node[circle,draw,inner sep=0.5mm,line width=0.5mm,inner sep=0.5mm](X-20) at (X20) {};
\node[circle,draw,inner sep=0.5mm, inner sep=0.5mm](X-20) at (X20) {};
\node[circle,draw,inner sep=0.5mm, inner sep=0.5mm](X-21) at (X21) {};
\node[circle,draw,inner sep=0.5mm, inner sep=0.5mm](X-22) at (X22) {};
\node[circle,draw,inner sep=0.5mm, inner sep=0.5mm](X-23) at (X23) {};
\node[circle,draw,inner sep=0.5mm,line width=0.5mm, inner sep=0.5mm](X-24) at (X24) {};
\node[circle,draw,inner sep=0.5mm, inner sep=0.5mm](X-24) at (X24) {};
\node[circle,draw,inner sep=0.5mm, inner sep=0.5mm](X-25) at (X25) {};
\draw (X-1) --  (X-4);
\draw (X-2) --  (X-4);
\draw (X-3) --  (X-8);
\draw (X-4) --  (X-4);
\draw[very thick] (X-4) --  (X-10);
\draw (X-5) --  (X-8);
\draw (X-6) --  (X-6);
\draw[very thick] (X-6) --  (X-24);
\draw (X-7) --  (X-10);
\draw (X-8) --  (X-8);
\draw[very thick] (X-8) --  (X-10);
\draw (X-9) --  (X-10);
\draw (X-10) --  (X-10);
\draw[very thick] (X-10) --  (X-19);
\draw[very thick,dashed] (X-10) --  (X-20);
\draw (X-11) --  (X-16);
\draw (X-12) --  (X-10);
\draw (X-13) --  (X-10);
\draw (X-14) --  (X-24);
\draw (X-15) --  (X-20);
\draw (X-16) --  (X-16);
\draw[very thick] (X-16) --  (X-24);
\draw (X-17) --  (X-19);
\draw (X-18) --  (X-24);
\draw (X-19) --  (X-19);
\draw (X-20) --  (X-20);
\draw[very thick,dashed] (X-20) --  (X-24);
\draw (X-21) --  (X-20);
\draw (X-22) --  (X-20);
\draw (X-23) --  (X-20);
\draw (X-24) --  (X-24);
\draw (X-25) --  (X-19);
\end{tikzpicture}}~\fbox{\begin{tikzpicture}[xscale=0.14, yscale=0.126]

\coordinate(X1) at (12.636459,19.644937);
\coordinate(X2) at (22.994535,18.316494);
\coordinate(X3) at (25.109379,13.461034);
\coordinate(X4) at (35.977228,21.761232);
\coordinate(X5) at (54.465362,14.583895);
\coordinate(X6) at (16.506582,34.491875);
\coordinate(X7) at (25.889015,30.419914);
\coordinate(X8) at (30.357654,31.045623);
\coordinate(X9) at (36.139999,31.567076);
\coordinate(X10) at (44.386053,32.801278);
\coordinate(X11) at (16.426283,39.421336);
\coordinate(X12) at (25.369507,39.226369);
\coordinate(X13) at (32.669659,38.251497);
\coordinate(X14) at (37.849578,38.326664);
\coordinate(X15) at (45.698532,38.341098);
\coordinate(X16) at (20.768868,45.081384);
\coordinate(X17) at (25.455421,46.525809);
\coordinate(X18) at (30.305016,47.419112);
\coordinate(X19) at (33.111914,45.662965);
\coordinate(X20) at (43.694252,42.765665);
\coordinate(X21) at (7.921644,53.604716);
\coordinate(X22) at (20.487772,51.738539);
\coordinate(X23) at (30.630363,49.230520);
\coordinate(X24) at (32.246728,50.584340);
\coordinate(X25) at (36.425404,49.852413);
\node[circle,draw,inner sep=0.5mm](X-1) at (X1) {};
\node[circle,draw,inner sep=0.5mm,fill](X-2) at (X2) {};
\node[circle,draw,inner sep=0.5mm](X-3) at (X3) {};
\node[circle,draw,inner sep=0.5mm,fill](X-4) at (X4) {};
\node[circle,draw,inner sep=0.5mm,fill](X-5) at (X5) {};
\node[circle,draw,inner sep=0.5mm,fill](X-6) at (X6) {};
\node[circle,draw,inner sep=0.5mm](X-7) at (X7) {};
\node[circle,draw,inner sep=0.5mm,line width=0.5mm](X-8) at (X8) {};
\node[circle,draw,inner sep=0.5mm](X-8) at (X8) {};
\node[circle,draw,inner sep=0.5mm](X-9) at (X9) {};
\node[circle,draw,inner sep=0.5mm](X-10) at (X10) {};
\node[circle,draw,inner sep=0.5mm](X-11) at (X11) {};
\node[circle,draw,inner sep=0.5mm](X-12) at (X12) {};
\node[circle,draw,inner sep=0.5mm](X-13) at (X13) {};
\node[circle,draw,inner sep=0.5mm](X-14) at (X14) {};
\node[circle,draw,inner sep=0.5mm](X-15) at (X15) {};
\node[circle,draw,inner sep=0.5mm](X-16) at (X16) {};
\node[circle,draw,inner sep=0.5mm,fill](X-17) at (X17) {};
\node[circle,draw,inner sep=0.5mm,line width=0.5mm](X-18) at (X18) {};
\node[circle,draw,inner sep=0.5mm](X-18) at (X18) {};
\node[circle,draw,inner sep=0.5mm](X-19) at (X19) {};
\node[circle,draw,inner sep=0.5mm,line width=0.5mm](X-20) at (X20) {};
\node[circle,draw,inner sep=0.5mm](X-20) at (X20) {};
\node[circle,draw,inner sep=0.5mm](X-21) at (X21) {};
\node[circle,draw,inner sep=0.5mm](X-22) at (X22) {};
\node[circle,draw,inner sep=0.5mm](X-23) at (X23) {};
\node[circle,draw,inner sep=0.5mm](X-24) at (X24) {};
\node[circle,draw,inner sep=0.5mm](X-25) at (X25) {};
\draw (X-1) --  (X-2);
\draw (X-2) --  (X-2);
\draw[very thick] (X-2) --  (X-8);
\draw (X-3) --  (X-2);
\draw (X-4) --  (X-4);
\draw[very thick] (X-4) --  (X-8);
\draw (X-5) --  (X-5);
\draw[very thick] (X-5) --  (X-8);
\draw (X-6) --  (X-6);
\draw[very thick] (X-6) --  (X-8);
\draw (X-7) --  (X-8);
\draw (X-8) --  (X-8);
\draw[very thick,dashed] (X-8) --  (X-18);
\draw (X-9) --  (X-8);
\draw (X-10) --  (X-20);
\draw (X-11) --  (X-17);
\draw (X-12) --  (X-17);
\draw (X-13) --  (X-8);
\draw (X-14) --  (X-20);
\draw (X-15) --  (X-20);
\draw (X-16) --  (X-17);
\draw (X-17) --  (X-17);
\draw[very thick] (X-17) --  (X-18);
\draw (X-18) --  (X-18);
\draw[very thick,dashed] (X-18) --  (X-20);
\draw (X-19) --  (X-18);
\draw (X-20) --  (X-20);
\draw (X-21) --  (X-17);
\draw (X-22) --  (X-17);
\draw (X-23) --  (X-18);
\draw (X-24) --  (X-18);
\draw (X-25) --  (X-18);
\end{tikzpicture}}
\caption{Solutions of one of the instances of CAB (left) and AP (right) datasets with $(n,p,q)=(25,8,3)$ and $(\alpha, \rho, \gamma)=(0.8,0.5,0.2)$.\label{fig25a}}
\end{center}
\end{figure}

{\small
\begin{table}\label{t:gaps}
\centering\begin{tabular}{|c|c|c||cc|rr|rr|c|cc|}\hline
$n$ & $p$ & $q$ &{\sc GAP} & {\sc GAP}$_{VI}$ & {\sc Nodes} & {\sc Nodes}$_{VI}$ & {\sc Time} & {\sc Time}$_{VI}$ & {\sc Cuts} & {\sc UnS} & {\sc UnS}$_{VI}$\\\hline\hline
\multirow{6}{*}{10} & 3  &  1 & 1.97\% & 0.89\%  & 43  &  10   &    0.23  &  0.18  &  32.86   &   0\%   &   0\%\\\cline{2-12}
& \multirow{2}{*}{5}  &  1 & 7.57\% & 2.63\%  & 502  &  142   &    0.57  &  0.4  &  52.43   &   0\%   &   0\%\\
   &    &  3 & 18.08\% & 1.77\%  & 628  &  1103   &    0.48  &  0.6  &  65.14   &   0\%   &   0\%\\\cline{2-12}
&\multirow{3}{*}{8}  &  1 & 15.6\% & 7.96\%  & 13900  &  17452   &    6.25  &  5.66  &  76.71   &   0\%   &   0\%\\
   &    &  3 & 31.94\% & 4.72\%  & 12722  &  34866   &    7.29  &  7.24  &  78.57   &   0\%   &   0\%\\
   &    &  5 & 22.26\% & 7.91\%  & 113634  &  43225   &    56.76  &  8.84  &  91.14   &   0\%   &   0\%\\\hline
\multirow{6}{*}{15} & 3  &  1 & 0.66\% & 0.57\%  & 46  &  8   &    0.8  &  0.43  &  48.71   &   0\%   &   0\%\\\cline{2-12}
& \multirow{2}{*}{5}  &  1 & 6.11\% & 2.46\%  & 1342  &  678   &    6.34  &  1.57  &  59.43   &   0\%   &   0\%\\
   &    &  3 & 14.19\% & 1.81\%  & 762  &  4849   &    2.64  &  2.18  &  77.29   &   0\%   &   0\%\\\cline{2-12}
&\multirow{3}{*}{8}  &  1 & 14.37\% & 6.37\%  & 56108  &  39439   &    189.47  &  84.33  &  95.86   &   0\%   &   0\%\\
   &    &  3 & 29.22\% & 4.6\%  & 26091  &  6380   &    79.02  &  21.98  &  103   &   0\%   &   0\%\\
   &    &  5 & 21.75\% & 7.06\%  & 285057  &  20780   &    1193.56  &  59.01  &  103.14   &   14.29\%   &   0\%\\\hline
\multirow{6}{*}{20} & 3  &  1 & 2.54\% & 1.09\%  & 54  &  16   &    3.73  &  2.86  &  53.71   &   0\%   &   0\%\\\cline{2-12}
& \multirow{2}{*}{5}  &  1 & 10.56\% & 4.18\%  & 3024  &  1388   &    58  &  51.06  &  86   &   0\%   &   0\%\\
   &    &  3 & 16.6\% & 3.47\%  & 1990  &  3064   &    34.72  &  26.21  &  103   &   0\%   &   0\%\\\cline{2-12}
&\multirow{3}{*}{8}  &  1 & 15.75\% & 6.66\%  & 89633  &  18512   &    2409.97  &  2326.28  &  95.29   &   28.57\%   &   28.57\%\\
   &    &  3 & 27.96\% & 7.16\%  & 53712  &  15891   &    1479.62  &  917.33  &  100   &   14.29\%   &   0\%\\
   &    &  5 & 19.44\% & 8\%  & 159421  &  23355   &    2573.71  &  1422.24  &  100   &   28.57\%   &   14.29\%\\\hline
\multirow{6}{*}{25} & 3  &  1 & 1.67\% & 1.09\%  & 32  &  20   &    8.85  &  6.35  &  58.29   &   0\%   &   0\%\\\cline{2-12}
& \multirow{2}{*}{5}  &  1 & 10.08\% & 4.86\%  & 14558  &  4854   &    819.24  &  409.37  &  90.29   &   0\%   &   0\%\\
   &    &  3 & 14.38\% & 3.57\%  & 4538  &  7011   &    224.49  &  209.53  &  100   &   0\%   &   0\%\\\cline{2-12}
&\multirow{3}{*}{8}  &  1 & 17.2\% & 10.06\%  & 52995  &  27438   &    5149.66  &  3511.41  &  100   &   71.43\%   &   28.57\%\\
   &    &  3 & 26.08\% & 6.88\%  & 37933  &  22130   &    2558.66  &  2463.34  &  100   &   28.57\%   &   28.57\%\\
   &    &  5 & 18.98\% & 7.79\%  & 123705  &  40437   &    5035.68  &  3302.36  &  100   &   57.14\%   &   42.86\%\\\hline
\end{tabular}
\caption{Average Results for the CAB dataset using \eqref{thlpu2}.\label{cab1}}
\end{table}}

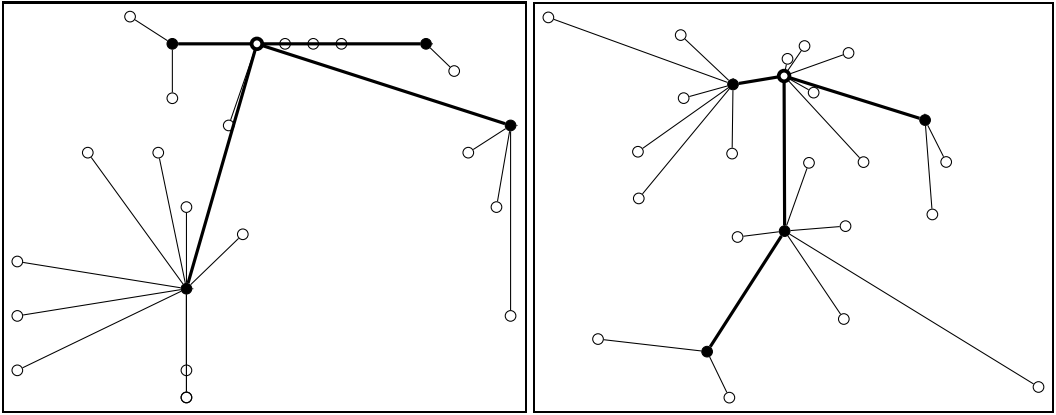
\begin{figure}
\begin{center}
\fbox{\begin{tikzpicture}[xscale=0.13, yscale=0.25]

\coordinate(X1) at (24.497000,0.000000);
\coordinate(X2) at (24.497000,0.010000);
\coordinate(X3) at (7.205000,1.448000);
\coordinate(X4) at (24.497000,1.448000);
\coordinate(X5) at (7.205000,4.344000);
\coordinate(X6) at (57.640000,4.344000);
\coordinate(X7) at (24.497000,5.792000);
\coordinate(X8) at (7.205000,7.240000);
\coordinate(X9) at (30.261000,8.688000);
\coordinate(X10) at (24.497000,10.136000);
\coordinate(X11) at (56.199000,10.136000);
\coordinate(X12) at (14.410000,13.032000);
\coordinate(X13) at (21.615000,13.032000);
\coordinate(X14) at (53.317000,13.032000);
\coordinate(X15) at (28.820000,14.480000);
\coordinate(X16) at (57.640000,14.480000);
\coordinate(X17) at (23.056000,15.928000);
\coordinate(X18) at (51.876000,17.376000);
\coordinate(X19) at (23.056000,18.824000);
\coordinate(X20) at (31.702000,18.824000);
\coordinate(X21) at (34.584000,18.824000);
\coordinate(X22) at (37.466000,18.824000);
\coordinate(X23) at (40.348000,18.824000);
\coordinate(X24) at (48.994000,18.824000);
\coordinate(X25) at (18.733000,20.272000);
\node[circle,draw,inner sep=0.5mm](X-1) at (X1) {};
\node[circle,draw,inner sep=0.5mm](X-2) at (X2) {};
\node[circle,draw,inner sep=0.5mm](X-3) at (X3) {};
\node[circle,draw,inner sep=0.5mm](X-4) at (X4) {};
\node[circle,draw,inner sep=0.5mm](X-5) at (X5) {};
\node[circle,draw,inner sep=0.5mm](X-6) at (X6) {};
\node[circle,draw,fill,inner sep=0.5mm](X-7) at (X7) {};
\node[circle,draw,inner sep=0.5mm](X-8) at (X8) {};
\node[circle,draw,inner sep=0.5mm](X-9) at (X9) {};
\node[circle,draw,inner sep=0.5mm](X-10) at (X10) {};
\node[circle,draw,inner sep=0.5mm](X-11) at (X11) {};
\node[circle,draw,inner sep=0.5mm](X-12) at (X12) {};
\node[circle,draw,inner sep=0.5mm](X-13) at (X13) {};
\node[circle,draw,inner sep=0.5mm](X-14) at (X14) {};
\node[circle,draw,inner sep=0.5mm](X-15) at (X15) {};
\node[circle,draw,fill,inner sep=0.5mm](X-16) at (X16) {};
\node[circle,draw,inner sep=0.5mm](X-17) at (X17) {};
\node[circle,draw,inner sep=0.5mm](X-18) at (X18) {};
\node[circle,draw,fill,inner sep=0.5mm](X-19) at (X19) {};
\node[circle,draw,line width=0.5mm,inner sep=0.5mm](X-20) at (X20) {};
\node[circle,draw,inner sep=0.5mm](X-20) at (X20) {};
\node[circle,draw,inner sep=0.5mm](X-21) at (X21) {};
\node[circle,draw,inner sep=0.5mm](X-22) at (X22) {};
\node[circle,draw,inner sep=0.5mm](X-23) at (X23) {};
\node[circle,draw,fill,inner sep=0.5mm](X-24) at (X24) {};
\node[circle,draw,inner sep=0.5mm](X-25) at (X25) {};
\draw (X-1) --  (X-7);
\draw (X-2) --  (X-7);
\draw (X-3) --  (X-7);
\draw (X-4) --  (X-7);
\draw (X-5) --  (X-7);
\draw (X-6) --  (X-16);
\draw (X-7) --  (X-7);
\draw[very thick] (X-7) --  (X-20);
\draw (X-8) --  (X-7);
\draw (X-9) --  (X-7);
\draw (X-10) --  (X-7);
\draw (X-11) --  (X-16);
\draw (X-12) --  (X-7);
\draw (X-13) --  (X-7);
\draw (X-14) --  (X-16);
\draw (X-15) --  (X-20);
\draw (X-16) --  (X-16);
\draw[very thick] (X-16) --  (X-20);
\draw (X-17) --  (X-19);
\draw (X-18) --  (X-24);
\draw (X-19) --  (X-19);
\draw[very thick] (X-19) --  (X-20);
\draw (X-20) --  (X-20);
\draw[very thick] (X-20) --  (X-24);
\draw (X-21) --  (X-20);
\draw (X-22) --  (X-20);
\draw (X-23) --  (X-20);
\draw (X-24) --  (X-24);
\draw (X-25) --  (X-19);
\end{tikzpicture}}~\fbox{\begin{tikzpicture}[xscale=0.14, yscale=0.126]

\coordinate(X1) at (12.636459,19.644937);
\coordinate(X2) at (22.994535,18.316494);
\coordinate(X3) at (25.109379,13.461034);
\coordinate(X4) at (35.977228,21.761232);
\coordinate(X5) at (54.465362,14.583895);
\coordinate(X6) at (16.506582,34.491875);
\coordinate(X7) at (25.889015,30.419914);
\coordinate(X8) at (30.357654,31.045623);
\coordinate(X9) at (36.139999,31.567076);
\coordinate(X10) at (44.386053,32.801278);
\coordinate(X11) at (16.426283,39.421336);
\coordinate(X12) at (25.369507,39.226369);
\coordinate(X13) at (32.669659,38.251497);
\coordinate(X14) at (37.849578,38.326664);
\coordinate(X15) at (45.698532,38.341098);
\coordinate(X16) at (20.768868,45.081384);
\coordinate(X17) at (25.455421,46.525809);
\coordinate(X18) at (30.305016,47.419112);
\coordinate(X19) at (33.111914,45.662965);
\coordinate(X20) at (43.694252,42.765665);
\coordinate(X21) at (7.921644,53.604716);
\coordinate(X22) at (20.487772,51.738539);
\coordinate(X23) at (30.630363,49.230520);
\coordinate(X24) at (32.246728,50.584340);
\coordinate(X25) at (36.425404,49.852413);
\node[circle,draw,inner sep=0.5mm](X-1) at (X1) {};
\node[circle,draw,fill,inner sep=0.5mm](X-2) at (X2) {};
\node[circle,draw,inner sep=0.5mm](X-3) at (X3) {};
\node[circle,draw,inner sep=0.5mm](X-4) at (X4) {};
\node[circle,draw,inner sep=0.5mm](X-5) at (X5) {};
\node[circle,draw,inner sep=0.5mm](X-6) at (X6) {};
\node[circle,draw,inner sep=0.5mm](X-7) at (X7) {};
\node[circle,draw,fill,inner sep=0.5mm](X-8) at (X8) {};
\node[circle,draw,inner sep=0.5mm](X-9) at (X9) {};
\node[circle,draw,inner sep=0.5mm](X-10) at (X10) {};
\node[circle,draw,inner sep=0.5mm](X-11) at (X11) {};
\node[circle,draw,inner sep=0.5mm](X-12) at (X12) {};
\node[circle,draw,inner sep=0.5mm](X-13) at (X13) {};
\node[circle,draw,inner sep=0.5mm](X-14) at (X14) {};
\node[circle,draw,inner sep=0.5mm](X-15) at (X15) {};
\node[circle,draw,inner sep=0.5mm](X-16) at (X16) {};
\node[circle,draw,fill,inner sep=0.5mm](X-17) at (X17) {};
\node[circle,draw,line width=0.5mm,inner sep=0.5mm](X-18) at (X18) {};
\node[circle,draw,inner sep=0.5mm](X-18) at (X18) {};
\node[circle,draw,inner sep=0.5mm](X-19) at (X19) {};
\node[circle,draw,fill,inner sep=0.5mm](X-20) at (X20) {};
\node[circle,draw,inner sep=0.5mm](X-21) at (X21) {};
\node[circle,draw,inner sep=0.5mm](X-22) at (X22) {};
\node[circle,draw,inner sep=0.5mm](X-23) at (X23) {};
\node[circle,draw,inner sep=0.5mm](X-24) at (X24) {};
\node[circle,draw,inner sep=0.5mm](X-25) at (X25) {};
\draw (X-1) --  (X-2);
\draw (X-2) --  (X-2);
\draw[very thick] (X-2) --  (X-8);
\draw (X-3) --  (X-2);
\draw (X-4) --  (X-8);
\draw (X-5) --  (X-8);
\draw (X-6) --  (X-17);
\draw (X-7) --  (X-8);
\draw (X-8) --  (X-8);
\draw[very thick] (X-8) --  (X-18);
\draw (X-9) --  (X-8);
\draw (X-10) --  (X-20);
\draw (X-11) --  (X-17);
\draw (X-12) --  (X-17);
\draw (X-13) --  (X-8);
\draw (X-14) --  (X-18);
\draw (X-15) --  (X-20);
\draw (X-16) --  (X-17);
\draw (X-17) --  (X-17);
\draw[very thick] (X-17) --  (X-18);
\draw (X-18) --  (X-18);
\draw[very thick] (X-18) --  (X-20);
\draw (X-19) --  (X-18);
\draw (X-20) --  (X-20);
\draw (X-21) --  (X-17);
\draw (X-22) --  (X-17);
\draw (X-23) --  (X-18);
\draw (X-24) --  (X-18);
\draw (X-25) --  (X-18);
\end{tikzpicture}}
\caption{2) Solutions of one of the instances of CAB (left) and AP (right) datasets with $(n,p,q)=(25,5,1)$ and $(\alpha, \rho, \gamma)=(0.8,0.5,0.2)$.\label{fig25b}}
\end{center}
\end{figure}

{\small
\begin{table}\label{t:gaps}
\centering\begin{tabular}{|c|c|c||cc|rr|rr|c|cc|}\hline
$n$ & $p$ & $q$ &{\sc GAP} & {\sc GAP}$_{VI}$ & {\sc Nodes} & {\sc Nodes}$_{VI}$ & {\sc Time} & {\sc Time}$_{VI}$ & {\sc Cuts} & {\sc UnS} & {\sc UnS}$_{VI}$\\\hline\hline
\multirow{6}{*}{10} & 3  &  1 & 5.97\% & 1.93\%  & 86  &  45   &    0.5  &  0.43  &  64.14   &   0\%   &   0\%\\\cline{2-12}
& \multirow{2}{*}{5}  &  1 & 12.56\% & 4.68\%  & 2143  &  299   &    3.6  &  1.27  &  84.29   &   0\%   &   0\%\\
   &    &  3 & 15.94\% & 3.75\%  & 1349  &  920   &    2.15  &  3.54  &  94.43   &   0\%   &   0\%\\\cline{2-12}
&\multirow{3}{*}{8}  &  1 & 21.43\% & 8.83\%  & 18709  &  4109   &    21.35  &  8.15  &  88.86   &   0\%   &   0\%\\
   &    &  3 & 31.72\% & 8.81\%  & 58509  &  10055   &    83.86  &  25.32  &  100   &   0\%   &   0\%\\
   &    &  5 & 24.73\% & 9.76\%  & 341521  &  11137   &    488.53  &  34.56  &  100   &   0\%   &   0\%\\\hline
\multirow{6}{*}{20} & 3  &  1 & 3.6\% & 1.99\%  & 310  &  164   &    13.45  &  10.71  &  75   &   0\%   &   0\%\\\cline{2-12}
& \multirow{2}{*}{5}  &  1 & 8.09\% & 4.55\%  & 16683  &  4869   &    692.13  &  307.85  &  96.71   &   0\%   &   0\%\\
   &    &  3 & 9.8\% & 2.62\%  & 4063  &  5606   &    188.01  &  514.45  &  100   &   0\%   &   0\%\\\cline{2-12}
&\multirow{3}{*}{8}  &  1 & 12.67\% & 7.15\%  & 52147  &  49098   &    3309.44  &  3430.02  &  100   &   42.86\%   &   28.57\%\\
   &    &  3 & 18.87\% & 5.47\%  & 51402  &  28278   &    2356.79  &  2522.03  &  100   &   28.57\%   &   28.57\%\\
   &    &  5 & 14.04\% & 5.75\%  & 165883  &  48332   &    5211.14  &  3554.61  &  100   &   57.14\%   &   42.86\%\\\hline
\multirow{6}{*}{25} & 3  &  1 & 2.91\% & 1.91\%  & 555  &  224   &    51.03  &  36.75  &  75.57   &   0\%   &   0\%\\\cline{2-12}
& \multirow{2}{*}{5}  &  1 & 7.83\% & 4.52\%  & 11320  &  4678   &    2220.5  &  902.69  &  92.57   &   28.57\%   &   0\%\\
   &    &  3 & 9.97\% & 2.79\%  & 3146  &  7428   &    325.84  &  1375.07  &  100   &   0\%   &   0\%\\\cline{2-12}
&\multirow{3}{*}{8}  &  1 & 12.48\% & 8.26\%  & 15721  &  13071   &    4076.31  &  5470.24  &  98.71   &   42.86\%   &   71.43\%\\
   &    &  3 & 17.48\% & 5.25\%  & 22312  &  18922   &    3746.89  &  6636.55  &  100   &   28.57\%   &   57.14\%\\
   &    &  5 & 12.47\% & 6.32\%  & 47079  &  13777   &    6422.61  &  4752.05  &  100   &   85.71\%   &   42.86\%\\\hline
\end{tabular}
\caption{Average Results for the AP dataset  using \eqref{thlpu2}.\label{ap1}}
\end{table}}

\begin{figure}
\begin{center}
\fbox{\begin{tikzpicture}[xscale=0.13, yscale=0.25]

\coordinate(X1) at (24.497000,0.000000);
\coordinate(X2) at (24.497000,0.010000);
\coordinate(X3) at (7.205000,1.448000);
\coordinate(X4) at (24.497000,1.448000);
\coordinate(X5) at (7.205000,4.344000);
\coordinate(X6) at (57.640000,4.344000);
\coordinate(X7) at (24.497000,5.792000);
\coordinate(X8) at (7.205000,7.240000);
\coordinate(X9) at (30.261000,8.688000);
\coordinate(X10) at (24.497000,10.136000);
\coordinate(X11) at (56.199000,10.136000);
\coordinate(X12) at (14.410000,13.032000);
\coordinate(X13) at (21.615000,13.032000);
\coordinate(X14) at (53.317000,13.032000);
\coordinate(X15) at (28.820000,14.480000);
\coordinate(X16) at (57.640000,14.480000);
\coordinate(X17) at (23.056000,15.928000);
\coordinate(X18) at (51.876000,17.376000);
\coordinate(X19) at (23.056000,18.824000);
\coordinate(X20) at (31.702000,18.824000);
\coordinate(X21) at (34.584000,18.824000);
\coordinate(X22) at (37.466000,18.824000);
\coordinate(X23) at (40.348000,18.824000);
\coordinate(X24) at (48.994000,18.824000);
\coordinate(X25) at (18.733000,20.272000);
\node[circle,draw,inner sep=0.5mm](X-1) at (X1) {};
\node[circle,draw,inner sep=0.5mm](X-2) at (X2) {};
\node[circle,draw,inner sep=0.5mm](X-3) at (X3) {};
\node[circle,draw,inner sep=0.5mm](X-4) at (X4) {};
\node[circle,draw,fill,inner sep=0.5mm](X-5) at (X5) {};
\node[circle,draw,inner sep=0.5mm](X-6) at (X6) {};
\node[circle,draw,line width=0.5mm,inner sep=0.5mm](X-7) at (X7) {};
\node[circle,draw,inner sep=0.5mm](X-7) at (X7) {};
\node[circle,draw,inner sep=0.5mm](X-8) at (X8) {};
\node[circle,draw,inner sep=0.5mm](X-9) at (X9) {};
\node[circle,draw,inner sep=0.5mm](X-10) at (X10) {};
\node[circle,draw,inner sep=0.5mm](X-11) at (X11) {};
\node[circle,draw,inner sep=0.5mm](X-12) at (X12) {};
\node[circle,draw,inner sep=0.5mm](X-13) at (X13) {};
\node[circle,draw,inner sep=0.5mm](X-14) at (X14) {};
\node[circle,draw,inner sep=0.5mm](X-15) at (X15) {};
\node[circle,draw,line width=0.5mm,inner sep=0.5mm](X-16) at (X16) {};
\node[circle,draw,inner sep=0.5mm](X-16) at (X16) {};
\node[circle,draw,inner sep=0.5mm](X-17) at (X17) {};
\node[circle,draw,inner sep=0.5mm](X-18) at (X18) {};
\node[circle,draw,fill,inner sep=0.5mm](X-19) at (X19) {};
\node[circle,draw,fill,inner sep=0.5mm](X-20) at (X20) {};
\node[circle,draw,inner sep=0.5mm](X-21) at (X21) {};
\node[circle,draw,inner sep=0.5mm](X-22) at (X22) {};
\node[circle,draw,inner sep=0.5mm](X-23) at (X23) {};
\node[circle,draw,inner sep=0.5mm](X-24) at (X24) {};
\node[circle,draw,inner sep=0.5mm](X-25) at (X25) {};
\draw (X-1) --  (X-7);
\draw (X-2) --  (X-7);
\draw (X-3) --  (X-5);
\draw (X-4) --  (X-7);
\draw (X-5) --  (X-5);
\draw[very thick] (X-5) --  (X-7);
\draw (X-6) --  (X-16);
\draw (X-7) --  (X-7);
\draw[very thick,dashed] (X-7) --  (X-16);
\draw[very thick] (X-7) --  (X-19);
\draw[very thick] (X-7) --  (X-20);
\draw (X-8) --  (X-5);
\draw (X-9) --  (X-7);
\draw (X-10) --  (X-7);
\draw (X-11) --  (X-16);
\draw (X-12) --  (X-19);
\draw (X-13) --  (X-19);
\draw (X-14) --  (X-16);
\draw (X-15) --  (X-20);
\draw (X-16) --  (X-16);
\draw (X-17) --  (X-19);
\draw (X-18) --  (X-16);
\draw (X-19) --  (X-19);
\draw (X-20) --  (X-20);
\draw (X-21) --  (X-20);
\draw (X-22) --  (X-20);
\draw (X-23) --  (X-20);
\draw (X-24) --  (X-16);
\draw (X-25) --  (X-19);
\end{tikzpicture}}~\fbox{\begin{tikzpicture}[xscale=0.14, yscale=0.126]

\coordinate(X1) at (12.636459,19.644937);
\coordinate(X2) at (22.994535,18.316494);
\coordinate(X3) at (25.109379,13.461034);
\coordinate(X4) at (35.977228,21.761232);
\coordinate(X5) at (54.465362,14.583895);
\coordinate(X6) at (16.506582,34.491875);
\coordinate(X7) at (25.889015,30.419914);
\coordinate(X8) at (30.357654,31.045623);
\coordinate(X9) at (36.139999,31.567076);
\coordinate(X10) at (44.386053,32.801278);
\coordinate(X11) at (16.426283,39.421336);
\coordinate(X12) at (25.369507,39.226369);
\coordinate(X13) at (32.669659,38.251497);
\coordinate(X14) at (37.849578,38.326664);
\coordinate(X15) at (45.698532,38.341098);
\coordinate(X16) at (20.768868,45.081384);
\coordinate(X17) at (25.455421,46.525809);
\coordinate(X18) at (30.305016,47.419112);
\coordinate(X19) at (33.111914,45.662965);
\coordinate(X20) at (43.694252,42.765665);
\coordinate(X21) at (7.921644,53.604716);
\coordinate(X22) at (20.487772,51.738539);
\coordinate(X23) at (30.630363,49.230520);
\coordinate(X24) at (32.246728,50.584340);
\coordinate(X25) at (36.425404,49.852413);
\node[circle,draw,inner sep=0.5mm](X-1) at (X1) {};
\node[circle,draw,fill,inner sep=0.5mm](X-2) at (X2) {};
\node[circle,draw,inner sep=0.5mm](X-3) at (X3) {};
\node[circle,draw,inner sep=0.5mm](X-4) at (X4) {};
\node[circle,draw,inner sep=0.5mm](X-5) at (X5) {};
\node[circle,draw,inner sep=0.5mm](X-6) at (X6) {};
\node[circle,draw,inner sep=0.5mm](X-7) at (X7) {};
\node[circle,draw,line width=0.5mm,inner sep=0.5mm](X-8) at (X8) {};
\node[circle,draw,inner sep=0.5mm](X-8) at (X8) {};
\node[circle,draw,inner sep=0.5mm](X-9) at (X9) {};
\node[circle,draw,inner sep=0.5mm](X-10) at (X10) {};
\node[circle,draw,inner sep=0.5mm](X-11) at (X11) {};
\node[circle,draw,inner sep=0.5mm](X-12) at (X12) {};
\node[circle,draw,inner sep=0.5mm](X-13) at (X13) {};
\node[circle,draw,inner sep=0.5mm](X-14) at (X14) {};
\node[circle,draw,inner sep=0.5mm](X-15) at (X15) {};
\node[circle,draw,inner sep=0.5mm](X-16) at (X16) {};
\node[circle,draw,fill,inner sep=0.5mm](X-17) at (X17) {};
\node[circle,draw,line width=0.5mm,inner sep=0.5mm](X-18) at (X18) {};
\node[circle,draw,inner sep=0.5mm](X-18) at (X18) {};
\node[circle,draw,inner sep=0.5mm](X-19) at (X19) {};
\node[circle,draw,fill,inner sep=0.5mm](X-20) at (X20) {};
\node[circle,draw,inner sep=0.5mm](X-21) at (X21) {};
\node[circle,draw,inner sep=0.5mm](X-22) at (X22) {};
\node[circle,draw,inner sep=0.5mm](X-23) at (X23) {};
\node[circle,draw,inner sep=0.5mm](X-24) at (X24) {};
\node[circle,draw,inner sep=0.5mm](X-25) at (X25) {};
\draw (X-1) --  (X-2);
\draw (X-2) --  (X-2);
\draw[very thick] (X-2) --  (X-8);
\draw (X-3) --  (X-2);
\draw (X-4) --  (X-8);
\draw (X-5) --  (X-8);
\draw (X-6) --  (X-8);
\draw (X-7) --  (X-8);
\draw (X-8) --  (X-8);
\draw[very thick,dashed] (X-8) --  (X-18);
\draw (X-9) --  (X-8);
\draw (X-10) --  (X-20);
\draw (X-11) --  (X-17);
\draw (X-12) --  (X-17);
\draw (X-13) --  (X-8);
\draw (X-14) --  (X-20);
\draw (X-15) --  (X-20);
\draw (X-16) --  (X-17);
\draw (X-17) --  (X-17);
\draw[very thick] (X-17) --  (X-18);
\draw (X-18) --  (X-18);
\draw[very thick] (X-18) --  (X-20);
\draw (X-19) --  (X-18);
\draw (X-20) --  (X-20);
\draw (X-21) --  (X-17);
\draw (X-22) --  (X-17);
\draw (X-23) --  (X-18);
\draw (X-24) --  (X-18);
\draw (X-25) --  (X-18);\end{tikzpicture}}
\caption{3) Solutions of one of the instances of CAB (left) and AP (right) datasets with $(n,p,q)=(25,5,2)$ and $(\alpha, \rho, \gamma)=(0.8,0.2,0.1)$.\label{fig25c}}
\end{center}
\end{figure}
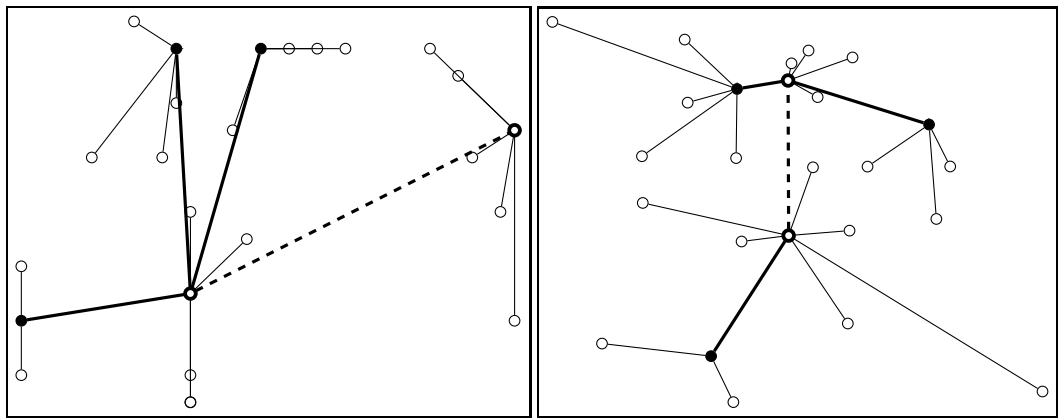

We observe that the disaggregated formulations has some strengths with respect to \eqref{thlpu1}. In this case, only 10\% (for \eqref{thlpu2}) and 5\% (for \eqref{thlpu2}+VI) of the CAB instances and 17\% (for \eqref{thlpu2}) and 15\% (for \eqref{thlpu2}+VI) of the AP instances, were not optimally solved within the time limit. The average GAP differences between adding or not valid inequalities to \eqref{thlpu2} were 11\% and  9\%, fo CAB and AP, respectively. $7$ out of the CAB instances and $4$ of the AP instances were optimally solved adding the valid inequalities but not with \eqref{thlpu1} and only $3$ of the AP instances were not solved adding the valid inequalities, but they did without them. Concerning the CPU times, in $85\%$ of the CAB instances and $60\%$ of the AP instances, the time for solving THLPU using \eqref{thlpu2}+VI was smaller than using \eqref{thlpu2} without valid inequalities.

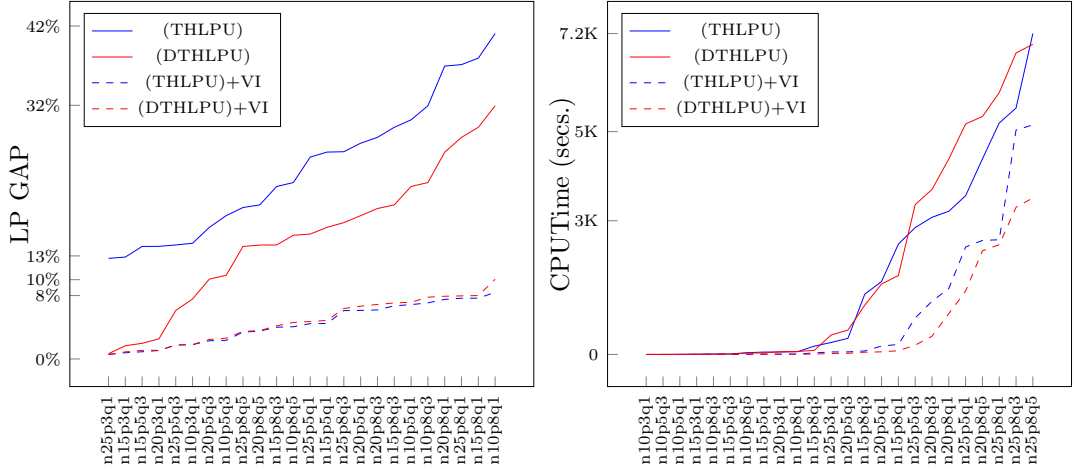
\begin{figure}
\begin{center}

\begin{tikzpicture}[xscale=0.9, yscale=0.9]
\begin{axis}[ylabel=LP GAP,
legend pos=north west,
xtick={0.1,0.2,0.3,0.4,0.5,0.6,0.7,0.8,0.9,1,1.1,1.2,1.3,1.4,1.5,1.6,1.7,1.8,1.9,2,2.1,2.2,2.3,2.4},
ytick={0,0.08,0.10, 0.13,0.32,0.42},
xtick pos=left,
ytick pos=left,
ylabel style={at={(-0.07,0.5)}},
yticklabels={$0\%$, $8\%$, $10\%$, $13\%$, $32\%$, $42\%$},
xticklabels={n25p3q1,n15p3q1,n15p5q3,n20p3q1,n25p5q3,n10p3q1,n20p5q3,n10p5q3,n25p8q5,n20p8q5,n15p8q3,n10p8q5,n25p5q1,n15p5q1,n25p8q3,n20p5q1,n20p8q3,n15p8q3,n10p5q1,n10p8q3,n20p8q1,n25p8q1,n15p8q1,n10p8q1},
xticklabel style={rotate=90, font=\tiny},
yticklabel style={font=\tiny}]

\addplot[no marks,blue] coordinates {(0.1,0.1268884) (0.2,0.1285529) (0.3,0.1419028) (0.4,0.142066) (0.5,0.1438113) (0.6,0.145944) (0.7,0.166023) (0.8,0.1808032) (0.9,0.1910593) (1,0.1943933) (1.1,0.2175429) (1.2,0.2225512) (1.3,0.2547006) (1.4,0.2609639) (1.5,0.2614362) (1.6,0.2721999) (1.7,0.2795704) (1.8,0.2922242) (1.9,0.3016656) (2,0.3193781) (2.1,0.3695029) (2.2,0.3713371) (2.3,0.3795459) (2.4,0.4105672)  };

\addplot[no marks, red] coordinates {
(0.1,0.006582554) (0.2,0.01667387) (0.3,0.01971189) (0.4,0.02538897) (0.5,0.06113225) (0.6,0.07565109) (0.7,0.1007675) (0.8,0.1055611) (0.9,0.1419028) (1,0.1436783) (1.1,0.1438113) (1.2,0.1560263) (1.3,0.1575313) (1.4,0.166023) (1.5,0.171987) (1.6,0.1808032) (1.7,0.1897668) (1.8,0.1943933) (1.9,0.2175429) (2,0.2225512) (2.1,0.2608424) (2.2,0.2795704) (2.3,0.2922242) (2.4,0.3193781) };

\addplot[no marks,dashed, blue] coordinates { 
(0.1,0.005671141) (0.2,0.007852326) (0.3,0.009069813) (0.4,0.01068785) (0.5,0.01746155) (0.6,0.01804381) (0.7,0.02298688) (0.8,0.02346545) (0.9,0.03381638) (1,0.03512423) (1.1,0.04004482) (1.2,0.04061376) (1.3,0.04453752) (1.4,0.04468433) (1.5,0.06098719) (1.6,0.06122056) (1.7,0.06185457) (1.8,0.0670588) (1.9,0.06852901) (2,0.07072154) (2.1,0.07514486) (2.2,0.07662432) (2.3,0.07670415) (2.4,0.08400692) };

\addplot[no marks, dashed,red] coordinates {
(0.1,0.005651014) (0.2,0.008926262) (0.3,0.01085642) (0.4,0.01089616) (0.5,0.01766566) (0.6,0.01813014) (0.7,0.02460162) (0.8,0.02625246) (0.9,0.03471339) (1,0.03572321) (1.1,0.04184086) (1.2,0.04602778) (1.3,0.04721389) (1.4,0.0486291) (1.5,0.06365751) (1.6,0.0665975) (1.7,0.06880628) (1.8,0.07062727) (1.9,0.07159302) (2,0.07794243) (2.1,0.07911036) (2.2,0.07957894) (2.3,0.07999247) (2.4,0.1005763) };

\addlegendentry{\tiny\eqref{thlpu1}}
\addlegendentry{\tiny\eqref{thlpu2}}
\addlegendentry{\tiny\eqref{thlpu1}+VI}
\addlegendentry{\tiny\eqref{thlpu2}+VI}
\end{axis}
\end{tikzpicture}~\begin{tikzpicture}[xscale=0.9, yscale=0.9]
\begin{axis}[ylabel=${\rm CPU Time}$ (secs.),
legend pos=north west,
xtick={0.1,0.2,0.3,0.4,0.5,0.6,0.7,0.8,0.9,1,1.1,1.2,1.3,1.4,1.5,1.6,1.7,1.8,1.9,2,2.1,2.2,2.3,2.4},
ytick={0,3000,5000,7200},
xtick pos=left,
ytick pos=left,
ylabel style={at={(-0.05,0.5)}},
yticklabels={$0$,$3$K, $5$K,$7.2$K},
xticklabels={n10p3q1,n10p5q3,n10p5q1,n15p3q1,n10p8q3,n15p5q3,n10p8q5,n15p5q1,n20p3q1,n10p8q1,n15p8q3,n25p3q1,n20p5q3,n15p8q3,n20p5q1,n15p8q1,n25p5q3,n20p8q3,n20p8q1,n25p5q1,n20p8q5,n25p8q1,n25p8q3,n25p8q5,
},
xticklabel style={rotate=90, font=\tiny},
yticklabel style={font=\tiny}]

\addplot[no marks,blue] coordinates {(0.1,0.6743421) (0.2,0.911711) (0.3,2.662554) (0.4,7.22831) (0.5,10.26102) (0.6,11.70498) (0.7,39.88853) (0.8,54.3245) (0.9,55.34119) (1,62.19299) (1.1,186.9844) (1.2,269.2224) (1.3,361.9155) (1.4,1352.03) (1.5,1637.181) (1.6,2475.534) (1.7,2846.83) (1.8,3077.457) (1.9,3213.18) (2,3561.087) (2.1,4387.132) (2.2,5193.181) (2.3,5528.636) (2.4,7200.081) 
};

\addplot[no marks, red] coordinates {
(0.1,0.8636907) (0.2,2.176303) (0.3,3.467604) (0.4,5.848562) (0.5,8.384597) (0.6,10.24605) (0.7,31.28262) (0.8,34.97518) (0.9,62.72033) (1,66.73788) (1.1,93.33421) (1.2,438.2319) (1.3,546.0088) (1.4,1105.032) (1.5,1581.081) (1.6,1770.632) (1.7,3361.695) (1.8,3701.573) (1.9,4385.508) (2,5171.51) (2.1,5337.845) (2.2,5874.183) (2.3,6763.241) (2.4,6958.075) 
 };

\addplot[no marks,dashed, blue] coordinates { 
(0.1,0.2330826) (0.2,0.4833191) (0.3,0.5702929) (0.4,0.7992233) (0.5,2.637205) (0.6,3.728841) (0.7,6.254884) (0.8,6.344827) (0.9,7.290982) (1,8.848145) (1.1,34.71767) (1.2,56.75748) (1.3,57.99853) (1.4,79.02219) (1.5,189.4694) (1.6,224.4866) (1.7,819.2417) (1.8,1193.561) (1.9,1479.625) (2,2409.974) (2.1,2558.66) (2.2,2573.711) (2.3,5035.683) (2.4,5149.662) 
};

\addplot[no marks, dashed,red] coordinates {(0.1,0.1808333) (0.2,0.4023146) (0.3,0.4323146) (0.4,0.6022357) (0.5,1.570271) (0.6,2.179924) (0.7,2.856963) (0.8,5.656551) (0.9,6.345223) (1,7.23874) (1.1,8.841873) (1.2,21.97999) (1.3,26.21288) (1.4,51.06265) (1.5,59.00779) (1.6,84.32545) (1.7,209.5312) (1.8,409.3741) (1.9,917.3269) (2,1422.237) (2.1,2326.278) (2.2,2463.344) (2.3,3302.363) (2.4,3511.408)  };

\addlegendentry{\tiny\eqref{thlpu1}}
\addlegendentry{\tiny\eqref{thlpu2}}
\addlegendentry{\tiny\eqref{thlpu1}+VI}
\addlegendentry{\tiny\eqref{thlpu2}+VI}
\end{axis}
\end{tikzpicture}
\caption{LP Gaps and CPU Times for the CAB dataset.\label{CABlines1}}
\end{center}
\end{figure}

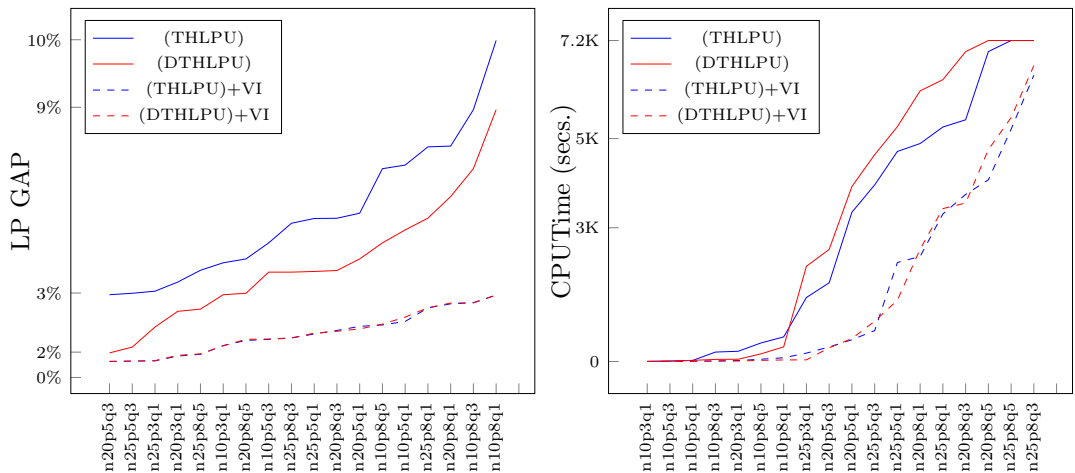
\begin{figure}
\begin{center}

\begin{tikzpicture}[xscale=0.9, yscale=0.9]
\begin{axis}[ylabel=LP GAP,
legend pos=north west,
xtick={0.1,0.2,0.3,0.4,0.5,0.6,0.7,0.8,0.9,1,1.1,1.2,1.3,1.4,1.5,1.6,1.7,1.8,1.9,2,2.1,2.2,2.3,2.4},
ytick={0,0.03,0.10, 0.32,0.4},
xtick pos=left,
ytick pos=left,
ylabel style={at={(-0.07,0.5)}},
yticklabels={$0\%$, $2\%$, $3\%$, $9\%$, $10\%$,  $32\%$, $40\%$},
xticklabels={n20p5q3,n25p5q3,n25p3q1,n20p3q1,n25p8q5,n10p3q1,n20p8q5,n10p5q3,n25p8q3,n25p5q1,n20p8q3,n20p5q1,n10p8q5,n10p5q1,n25p8q1,n20p8q1,n10p8q3,n10p8q1
},
xticklabel style={rotate=90, font=\tiny},
yticklabel style={font=\tiny}]

\addplot[no marks,blue] coordinates {(0.1,0.09803462) (0.2,0.09972343) (0.3,0.1022024) (0.4,0.1129897) (0.5,0.1269658) (0.6,0.1358303) (0.7,0.1404005) (0.8,0.1594156) (0.9,0.1826401) (1,0.1883123) (1.1,0.1885837) (1.2,0.1945561) (1.3,0.247298) (1.4,0.2515488) (1.5,0.2731571) (1.6,0.2740636) (1.7,0.3171821) (1.8,0.3990989) 
};

\addplot[no marks, red] coordinates {(0.1,0.0291037) (0.2,0.03596651) (0.3,0.05974103) (0.4,0.07832012) (0.5,0.08094651) (0.6,0.09803462) (0.7,0.09972343) (0.8,0.1247491) (0.9,0.1247581) (1,0.1255647) (1.1,0.1266797) (1.2,0.1404005) (1.3,0.1594156) (1.4,0.1747711) (1.5,0.1886972) (1.6,0.2142837) (1.7,0.247298) (1.8,0.3171821) 
 };

\addplot[no marks,dashed, blue] coordinates {(0.1,0.01894205) (0.2,0.01942651) (0.3,0.01969849) (0.4,0.02536903) (0.5,0.0274445) (0.6,0.0377647) (0.7,0.04391997) (0.8,0.04514521) (0.9,0.04678452) (1,0.05145225) (1.1,0.05568498) (1.2,0.06051096) (1.3,0.06220988) (1.4,0.0663581) (1.5,0.08243305) (1.6,0.08719868) (1.7,0.08863519) (1.8,0.0975779)  };

\addplot[no marks, dashed,red] coordinates {
(0.1,0.01910488) (0.2,0.01931742) (0.3,0.0199019) (0.4,0.02621238) (0.5,0.02791331) (0.6,0.0375383) (0.7,0.04519996) (0.8,0.04552573) (0.9,0.04682384) (1,0.05250077) (1.1,0.05474312) (1.2,0.05748738) (1.3,0.06316101) (1.4,0.07149701) (1.5,0.08257793) (1.6,0.08813513) (1.7,0.08826735) (1.8,0.09762072) };

\addlegendentry{\tiny\eqref{thlpu1}}
\addlegendentry{\tiny\eqref{thlpu2}}
\addlegendentry{\tiny\eqref{thlpu1}+VI}
\addlegendentry{\tiny\eqref{thlpu2}+VI}
\end{axis}
\end{tikzpicture}~\begin{tikzpicture}[xscale=0.9, yscale=0.9]
\begin{axis}[ylabel=${\rm CPU Time}$ (secs.),
legend pos=north west,
xtick={0.1,0.2,0.3,0.4,0.5,0.6,0.7,0.8,0.9,1,1.1,1.2,1.3,1.4,1.5,1.6,1.7,1.8,1.9,2,2.1,2.2,2.3,2.4},
ytick={0,3000,5000,7200},
xtick pos=left,
ytick pos=left,
ylabel style={at={(-0.05,0.5)}},
yticklabels={$0$,$3$K, $5$K,$7.2$K},
xticklabels={ n10p3q1,n10p5q3,n10p5q1,n10p8q3,n20p3q1,n10p8q5,n10p8q1,n25p3q1,n20p5q3,n20p5q1,n25p5q3,n25p5q1,n20p8q1,n25p8q1,n20p8q3,n20p8q5,n25p8q5,n25p8q3
 },
xticklabel style={rotate=90, font=\tiny},
yticklabel style={font=\tiny}]

\addplot[no marks,blue] coordinates {
(0.1,2.467477) (0.2,6.689445) (0.3,23.49947) (0.4,211.6551) (0.5,229.2259) (0.6,415.5731) (0.7,551.7424) (0.8,1435.729) (0.9,1765.749) (1,3347.988) (1.1,3960.224) (1.2,4711.684) (1.3,4889.963) (1.4,5259.531) (1.5,5422.102) (1.6,6952.097) (1.7,7200.041) (1.8,7200.079) 
};

\addplot[no marks, red] coordinates {
(0.1,2.34729) (0.2,13.38838) (0.3,25.67401) (0.4,46.94549) (0.5,50.84579) (0.6,170.9025) (0.7,328.59) (0.8,2136.171) (0.9,2513.639) (1,3924.913) (1.1,4638.895) (1.2,5267.489) (1.3,6070.698) (1.4,6321.671) (1.5,6950.072) (1.6,7200.038) (1.7,7200.039) (1.8,7200.045) 
};

\addplot[no marks,dashed, blue] coordinates {
(0.1,0.5018726) (0.2,2.148882) (0.3,3.602802) (0.4,13.45002) (0.5,21.3503) (0.6,51.03325) (0.7,83.86017) (0.8,188.0127) (0.9,325.8429) (1,488.5276) (1.1,692.1268) (1.2,2220.504) (1.3,2356.792) (1.4,3309.444) (1.5,3746.891) (1.6,4076.312) (1.7,5211.14) (1.8,6422.609) 
};

\addplot[no marks, dashed,red] coordinates {
(0.1,0.4300474) (0.2,1.265316) (0.3,3.535884) (0.4,8.145456) (0.5,10.71161) (0.6,25.31923) (0.7,34.56078) (0.8,36.74856) (0.9,307.854) (1,514.4482) (1.1,902.6933) (1.2,1375.068) (1.3,2522.025) (1.4,3430.02) (1.5,3554.611) (1.6,4752.047) (1.7,5470.236) (1.8,6636.551) };

\addlegendentry{\tiny\eqref{thlpu1}}
\addlegendentry{\tiny\eqref{thlpu2}}
\addlegendentry{\tiny\eqref{thlpu1}+VI}
\addlegendentry{\tiny\eqref{thlpu2}+VI}
\end{axis}
\end{tikzpicture}
\caption{LP Gaps and CPU Times for the AP dataset.\label{APines1}}
\end{center}
\end{figure}

One can observe that the THLPU is still very time consuming with the disaggregated formulation, but some improvements  were detected when comparing with \eqref{thlpu1}. First, the LP gaps obtained with \eqref{thlpu2} are an average of $0.1$ smaller than those obtained with \eqref{thlpu2} (see left picture in Figure \ref{CABlines1}). However, such a significative difference, does not always positively affects a decreasing on the CPU times needed to solve THLPU instances. The consuming CPU times for solving the problems using \eqref{thlpu1} and \eqref{thlpu2} are quite similar, in average. Nevertheless, when our family of valid inequalities are incorporated, the best CPU times obtained (when comparing the four approaches) are obtained with \eqref{thlpu2}+VI (see right picture in Figure  \ref{CABlines1}). Actually, using such a strengthening we were able to solve up to optimality  the greatest number of instances (all except $10$ for the CAB dataset), even being the average number of nodes explored in the branch-and-bound tree, greater, in average for \eqref{thlpu2}+VI than for \eqref{thlpu2} without incorporating the new family of inequalities. Note that when valid inequalities are considered in our model, the exploration of the search tree in the branch-and-bound procedure, may differ and although the LP relaxation of \eqref{thlpu2}+VI, the optimality has still to be checked, which may consume a huge amount of time.

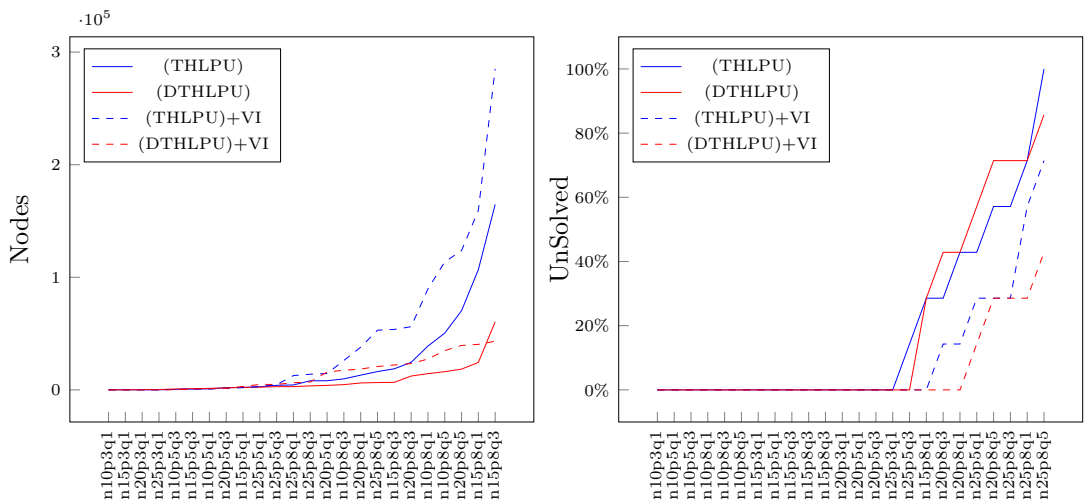
\begin{figure}
\begin{center}

\begin{tikzpicture}[xscale=0.9, yscale=0.9]
\begin{axis}[ylabel=${\rm Nodes}$,
legend pos=north west,
xtick={0.1,0.2,0.3,0.4,0.5,0.6,0.7,0.8,0.9,1,1.1,1.2,1.3,1.4,1.5,1.6,1.7,1.8,1.9,2,2.1,2.2,2.3,2.4},
xtick pos=left,
ytick pos=left,
ylabel style={at={(-0.07,0.5)}},
xticklabels={n10p3q1,n15p3q1,n20p3q1,n25p3q1,n10p5q3,n15p5q3,n10p5q1,n20p5q3,n15p5q1,n25p5q1,n25p5q3,n25p8q1,n25p8q3,n20p5q1,n10p8q3,n20p8q1,n25p8q5,n15p8q3,n20p8q3,n10p8q1,n10p8q5,n20p8q5,n15p8q1,n15p8q3},
xticklabel style={rotate=90, font=\tiny},
yticklabel style={font=\tiny}]

\addplot[no marks,blue] coordinates {(0.1,54.85714) (0.2,65.85714) (0.3,69) (0.4,77.57143) (0.5,433.8571) (0.6,729.7143) (0.7,901.4286) (0.8,2018.429) (0.9,2209.571) (1,2727.143) (1.1,4047.571) (1.2,4383) (1.3,8085.571) (1.4,8120.571) (1.5,9806.857) (1.6,13080) (1.7,16258.29) (1.8,18728.14) (1.9,24414.29) (2,38897.43) (2.1,50420.86) (2.2,70308.14) (2.3,106814.6) (2.4,164716.4)  };

\addplot[no marks, red] coordinates {
(0.1,41.71429) (0.2,97.85714) (0.3,219) (0.4,317.5714) (0.5,824.1429) (0.6,1081.857) (0.7,1324.429) (0.8,1764.571) (0.9,2184.286) (1,2456.286) (1.1,2852.286) (1.2,2914.143) (1.3,3559.143) (1.4,3978.143) (1.5,4679.286) (1.6,6165.429) (1.7,6517.143) (1.8,6713.714) (1.9,12297.14) (2,14418.29) (2.1,16187.43) (2.2,18452.43) (2.3,24390.43) (2.4,60480.29) };

\addplot[no marks,dashed, blue] coordinates {
(0.1,32) (0.2,43.28571) (0.3,45.71429) (0.4,54) (0.5,502) (0.6,627.5714) (0.7,761.7143) (0.8,1342) (0.9,1990.286) (1,3023.571) (1.1,4538.286) (1.2,12722) (1.3,13900.43) (1.4,14557.86) (1.5,26091.29) (1.6,37933) (1.7,52994.71) (1.8,53711.57) (1.9,56107.86) (2,89632.57) (2.1,113634.3) (2.2,123705) (2.3,159421.4) (2.4,285057) 
};

\addplot[no marks, dashed,red] coordinates {(0.1,8) (0.2,10.28571) (0.3,15.85714) (0.4,19.57143) (0.5,141.7143) (0.6,678.2857) (0.7,1102.571) (0.8,1388.143) (0.9,3064.286) (1,4849.429) (1.1,4854) (1.2,6379.857) (1.3,7010.571) (1.4,15891) (1.5,17452.43) (1.6,18511.57) (1.7,20780) (1.8,22129.86) (1.9,23355.29) (2,27437.71) (2.1,34865.86) (2.2,39439) (2.3,40437.29) (2.4,43225) };

\addlegendentry{\tiny\eqref{thlpu1}}
\addlegendentry{\tiny\eqref{thlpu2}}
\addlegendentry{\tiny\eqref{thlpu1}+VI}
\addlegendentry{\tiny\eqref{thlpu2}+VI}
\end{axis}
\end{tikzpicture}~\begin{tikzpicture}[xscale=0.9, yscale=0.9]
\begin{axis}[ylabel=${\rm UnSolved}$,
legend pos=north west,
xtick={0.1,0.2,0.3,0.4,0.5,0.6,0.7,0.8,0.9,1,1.1,1.2,1.3,1.4,1.5,1.6,1.7,1.8,1.9,2,2.1,2.2,2.3,2.4},
ytick={0,0.2,0.4,0.6,0.8,1},
xtick pos=left,
ytick pos=left,
ylabel style={at={(-0.09,0.5)}},
yticklabels={$0\%$, $20\%$, $40\%$, $60\%$, $80\%$, $100\%$},
xticklabels={n10p3q1,n10p5q1,n10p5q3,n10p8q1,n10p8q3,n10p8q5,n15p3q1,n15p5q1,n15p5q3,n15p8q3,n15p8q3,n20p3q1,n20p5q1,n20p5q3,n25p3q1,n25p5q3,n15p8q1,n20p8q3,n20p8q1,n25p5q1,n20p8q5,n25p8q3,n25p8q1,n25p8q5
 },
xticklabel style={rotate=90, font=\tiny},
yticklabel style={font=\tiny}]

\addplot[no marks,blue] coordinates {(0.1,0) (0.2,0) (0.3,0) (0.4,0) (0.5,0) (0.6,0) (0.7,0) (0.8,0) (0.9,0) (1,0) (1.1,0) (1.2,0) (1.3,0) (1.4,0) (1.5,0) (1.6,0.1428571) (1.7,0.2857143) (1.8,0.2857143) (1.9,0.4285714) (2,0.4285714) (2.1,0.5714286) (2.2,0.5714286) (2.3,0.7142857) (2.4,1) 
};

\addplot[no marks, red] coordinates {(0.1,0) (0.2,0) (0.3,0) (0.4,0) (0.5,0) (0.6,0) (0.7,0) (0.8,0) (0.9,0) (1,0) (1.1,0) (1.2,0) (1.3,0) (1.4,0) (1.5,0) (1.6,0) (1.7,0.2857143) (1.8,0.4285714) (1.9,0.4285714) (2,0.5714286) (2.1,0.7142857) (2.2,0.7142857) (2.3,0.7142857) (2.4,0.8571429) 
};

\addplot[no marks,dashed, blue] coordinates {(0.1,0) (0.2,0) (0.3,0) (0.4,0) (0.5,0) (0.6,0) (0.7,0) (0.8,0) (0.9,0) (1,0) (1.1,0) (1.2,0) (1.3,0) (1.4,0) (1.5,0) (1.6,0) (1.7,0) (1.8,0.1428571) (1.9,0.1428571) (2,0.2857143) (2.1,0.2857143) (2.2,0.2857143) (2.3,0.5714286) (2.4,0.7142857) 
};

\addplot[no marks, dashed,red] coordinates {(0.1,0) (0.2,0) (0.3,0) (0.4,0) (0.5,0) (0.6,0) (0.7,0) (0.8,0) (0.9,0) (1,0) (1.1,0) (1.2,0) (1.3,0) (1.4,0) (1.5,0) (1.6,0) (1.7,0) (1.8,0) (1.9,0) (2,0.1428571) (2.1,0.2857143) (2.2,0.2857143) (2.3,0.2857143) (2.4,0.4285714)  };

\addlegendentry{\tiny\eqref{thlpu1}}
\addlegendentry{\tiny\eqref{thlpu2}}
\addlegendentry{\tiny\eqref{thlpu1}+VI}
\addlegendentry{\tiny\eqref{thlpu2}+VI}
\end{axis}
\end{tikzpicture}
\caption{Nodes and UnSolved Instances for the CAB dataset.\label{CABlines2}}
\end{center}
\end{figure}

\begin{figure}
\begin{center}

\begin{tikzpicture}[xscale=0.9, yscale=0.9]
\begin{axis}[ylabel=${\rm Nodes}$,
legend pos=north west,
xtick={0.1,0.2,0.3,0.4,0.5,0.6,0.7,0.8,0.9,1,1.1,1.2,1.3,1.4,1.5,1.6,1.7,1.8,1.9,2,2.1,2.2,2.3,2.4},
xtick pos=left,
ytick pos=left,
ylabel style={at={(-0.07,0.5)}},
xticklabels={n10p3q1,n20p3q1,n25p3q1,n25p8q1,n25p5q1,n10p5q3,n25p5q3,n10p5q1,n25p8q3,n20p5q1,n20p5q3,n25p8q5,n20p8q1,n20p8q3,n20p8q5,n10p8q3,n10p8q1,n10p8q5},
xticklabel style={rotate=90, font=\tiny},
yticklabel style={font=\tiny}]

\addplot[no marks,blue] coordinates {(0.1,106.5714) (0.2,246.5714) (0.3,354.5714) (0.4,886.1429) (0.5,1018.429) (0.6,1037.571) (0.7,2267) (0.8,2558.143) (0.9,3073.857) (1,3437.143) (1.1,3438.143) (1.2,4107.714) (1.3,6950.286) (1.4,15215.29) (1.5,26173.71) (1.6,43448.43) (1.7,94849.14) (1.8,121931.6)  };

\addplot[no marks, red] coordinates {(0.1,105.1429) (0.2,434.4286) (0.3,737.5714) (0.4,1122) (0.5,1359) (0.6,1543.286) (0.7,1614) (0.8,1970.571) (0.9,2434.429) (1,2450.286) (1.1,3060.857) (1.2,5598.714) (1.3,6378) (1.4,6552.714) (1.5,6952.286) (1.6,8450.714) (1.7,9798.714) (1.8,22573.14) };

\addplot[no marks,dashed, blue] coordinates {
(0.1,86.14286) (0.2,309.5714) (0.3,554.7143) (0.4,1348.714) (0.5,2143) (0.6,3146) (0.7,4062.857) (0.8,11320.14) (0.9,15721.29) (1,16682.57) (1.1,18709.14) (1.2,22311.86) (1.3,47079.29) (1.4,51401.86) (1.5,52147) (1.6,58509.14) (1.7,165883) (1.8,341520.7) 
};

\addplot[no marks, dashed,red] coordinates {(0.1,45) (0.2,164.4286) (0.3,224.2857) (0.4,298.8571) (0.5,919.8571) (0.6,4108.571) (0.7,4677.857) (0.8,4868.857) (0.9,5606) (1,7427.571) (1.1,10054.57) (1.2,11137.43) (1.3,13070.86) (1.4,13777) (1.5,18921.57) (1.6,28278.29) (1.7,48331.86) (1.8,49097.57)};

\addlegendentry{\tiny\eqref{thlpu1}}
\addlegendentry{\tiny\eqref{thlpu2}}
\addlegendentry{\tiny\eqref{thlpu1}+VI}
\addlegendentry{\tiny\eqref{thlpu2}+VI}
\end{axis}
\end{tikzpicture}~\begin{tikzpicture}[xscale=0.9, yscale=0.9]
\begin{axis}[ylabel=${\rm UnSolved}$,
legend pos=north west,
xtick={0.1,0.2,0.3,0.4,0.5,0.6,0.7,0.8,0.9,1,1.1,1.2,1.3,1.4,1.5,1.6,1.7,1.8,1.9,2,2.1,2.2,2.3,2.4},
ytick={0,0.2,0.4,0.6,0.8,1},
xtick pos=left,
ytick pos=left,
ylabel style={at={(-0.09,0.5)}},
yticklabels={$0\%$, $20\%$, $40\%$, $60\%$, $80\%$, $100\%$},
xticklabels={n10p3q1,n10p5q1,n10p5q3,n10p8q1,n10p8q3,n10p8q5,n20p3q1,n20p5q3,n25p3q1,n25p5q3,n20p5q1,n20p8q3,n20p8q1,n25p5q1,n25p8q1,n20p8q5,n25p8q3,n25p8q5},
xticklabel style={rotate=90, font=\tiny},
yticklabel style={font=\tiny}]

\addplot[no marks,blue] coordinates {(0.1,0) (0.2,0) (0.3,0) (0.4,0) (0.5,0) (0.6,0) (0.7,0) (0.8,0) (0.9,0) (1,0.1428571) (1.1,0.4285714) (1.2,0.4285714) (1.3,0.5714286) (1.4,0.5714286) (1.5,0.7142857) (1.6,0.8571429) (1.7,1) (1.8,1) 
};

\addplot[no marks, red] coordinates {(0.1,0) (0.2,0) (0.3,0) (0.4,0) (0.5,0) (0.6,0) (0.7,0) (0.8,0) (0.9,0.1428571) (1,0.2857143) (1.1,0.4285714) (1.2,0.5714286) (1.3,0.7142857) (1.4,0.7142857) (1.5,0.8571429) (1.6,1) (1.7,1) (1.8,1)  };

\addplot[no marks,dashed, blue] coordinates {
(0.1,0) (0.2,0) (0.3,0) (0.4,0) (0.5,0) (0.6,0) (0.7,0) (0.8,0) (0.9,0) (1,0) (1.1,0) (1.2,0.2857143) (1.3,0.2857143) (1.4,0.2857143) (1.5,0.4285714) (1.6,0.4285714) (1.7,0.4285714) (1.8,0.5714286)};

\addplot[no marks, dashed,red] coordinates { (0.1,0) (0.2,0) (0.3,0) (0.4,0) (0.5,0) (0.6,0) (0.7,0) (0.8,0) (0.9,0) (1,0) (1.1,0) (1.2,0) (1.3,0.2857143) (1.4,0.2857143) (1.5,0.5714286) (1.6,0.5714286) (1.7,0.7142857) (1.8,0.7142857)};

\addlegendentry{\tiny\eqref{thlpu1}}
\addlegendentry{\tiny\eqref{thlpu2}}
\addlegendentry{\tiny\eqref{thlpu1}+VI}
\addlegendentry{\tiny\eqref{thlpu2}+VI}
\end{axis}
\end{tikzpicture}
\caption{Nodes and UnSolved Instances for the AP dataset.\label{APlines2}}
\end{center}
\end{figure}
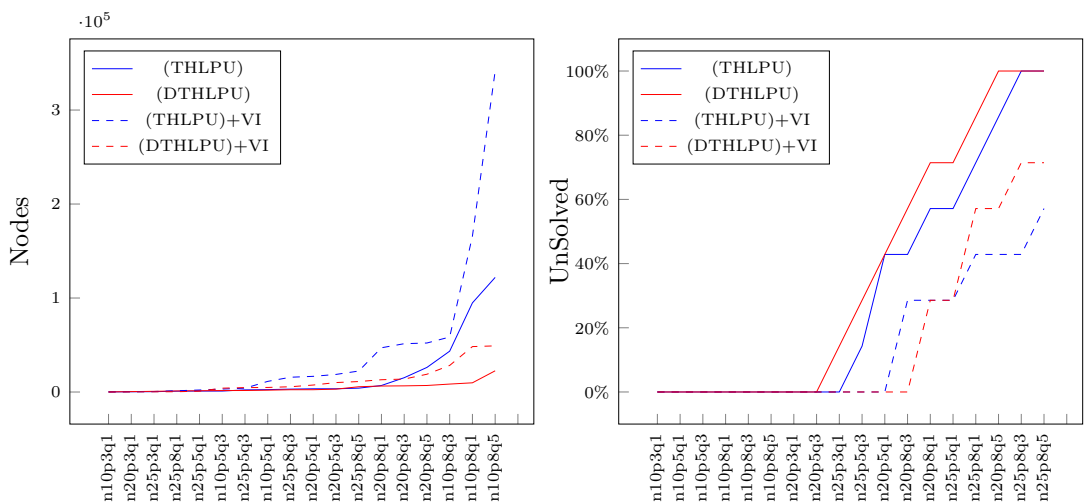

Concerning the number of valid inequalities added with the separation strategy, one can see in Figure \ref{cuts} that less number of cuts are, in general needed, with the disaggregated  formulation. Note that, apart from the upper bound for the number of cuts added in the procedure, we only stop adding cuts when the gap between two consecutive LP solutions is less than $1\%$.

\begin{figure}
\begin{center}

\begin{tikzpicture}[xscale=0.9, yscale=0.9]
\begin{axis}[ylabel=${\rm \#Cuts}$ CAB,
legend pos=south east,
xtick={0.1,0.2,0.3,0.4,0.5,0.6,0.7,0.8,0.9,1,1.1,1.2,1.3,1.4,1.5,1.6,1.7,1.8,1.9,2,2.1,2.2,2.3,2.4},
xtick pos=left,
ytick pos=left,
ylabel style={at={(-0.04,0.5)}},
xticklabels={n10p3q1,n20p3q1,n25p3q1,n10p5q1,n15p3q1,n15p5q1,n10p5q3,n15p5q3,n10p8q1,n10p8q3,n20p5q1,n25p5q1,n10p8q5,n15p8q1,n15p8q3,n15p8q5,n20p5q3,n20p8q1,n20p8q3,n20p8q5,n25p5q3,n25p8q1,n25p8q3,n25p8q5
},
xticklabel style={rotate=90, font=\tiny},
yticklabel style={font=\tiny}]

\addplot[no marks,dashed, blue] coordinates {
(0.1,44.14286) (0.2,45) (0.3,45.14286) (0.4,63.28571) (0.5,65.57143) (0.6,77.85714) (0.7,78.85714) (0.8,87.14286) (0.9,92.28571) (1,93.14286) (1.1,95.42857) (1.2,98.57143) (1.3,100) (1.4,100) (1.5,100) (1.6,100) (1.7,100) (1.8,100) (1.9,100) (2,100) (2.1,100) (2.2,100) (2.3,100) (2.4,100) 
};

\addplot[no marks, dashed,red] coordinates {
(0.1,32.85714) (0.2,48.71429) (0.3,52.42857) (0.4,53.71429) (0.5,58.28571) (0.6,59.42857) (0.7,65.14286) (0.8,76.71429) (0.9,77.28571) (1,78.57143) (1.1,86) (1.2,90.28571) (1.3,91.14286) (1.4,95.28571) (1.5,95.85714) (1.6,100) (1.7,100) (1.8,100) };

\addlegendentry{\tiny\eqref{thlpu1}+VI}
\addlegendentry{\tiny\eqref{thlpu2}+VI}
\end{axis}
\end{tikzpicture}~\begin{tikzpicture}[xscale=0.9, yscale=0.9]
\begin{axis}[ylabel=${\rm \#Cuts}$ AP,
legend pos=south east,
xtick={0.1,0.2,0.3,0.4,0.5,0.6,0.7,0.8,0.9,1,1.1,1.2,1.3,1.4,1.5,1.6,1.7,1.8},
xtick pos=left,
ytick pos=left,
ylabel style={at={(-0.04,0.5)}},
xticklabels={n10p3q1,n10p5q1,n20p3q1,n25p3q1,n10p8q1,n10p5q3,n25p5q1,n10p8q3,n10p8q5,n20p5q1,n20p5q3,n20p8q1,n20p8q3,n20p8q5,n25p5q3,n25p8q1,n25p8q3,n25p8q5
},
xticklabel style={rotate=90, font=\tiny},
yticklabel style={font=\tiny}]

\addplot[no marks,dashed, blue] coordinates {
(0.1,72.42857) (0.2,82.57143) (0.3,83.71429) (0.4,85) (0.5,90.71429) (0.6,94.42857) (0.7,98) (0.8,100) (0.9,100) (1,100) (1.1,100) (1.2,100) (1.3,100) (1.4,100) (1.5,100) (1.6,100) (1.7,100) (1.8,100) 
};

\addplot[no marks, dashed,red] coordinates {
(0.1,64.14286) (0.2,75) (0.3,75.57143) (0.4,84.28571) (0.5,88.85714) (0.6,92.57143) (0.7,94.42857) (0.8,96.71429) (0.9,98.71429) (1,100) (1.1,100) (1.2,100) (1.3,100) (1.4,100) (1.5,100) (1.6,100) (1.7,100) (1.8,100) };

\addlegendentry{\tiny\eqref{thlpu1}+VI}
\addlegendentry{\tiny\eqref{thlpu2}+VI}
\end{axis}
\end{tikzpicture}
\caption{Number of Cuts for the CAB (left) and the AP (right) datasets.\label{cuts}}
\end{center}
\end{figure}
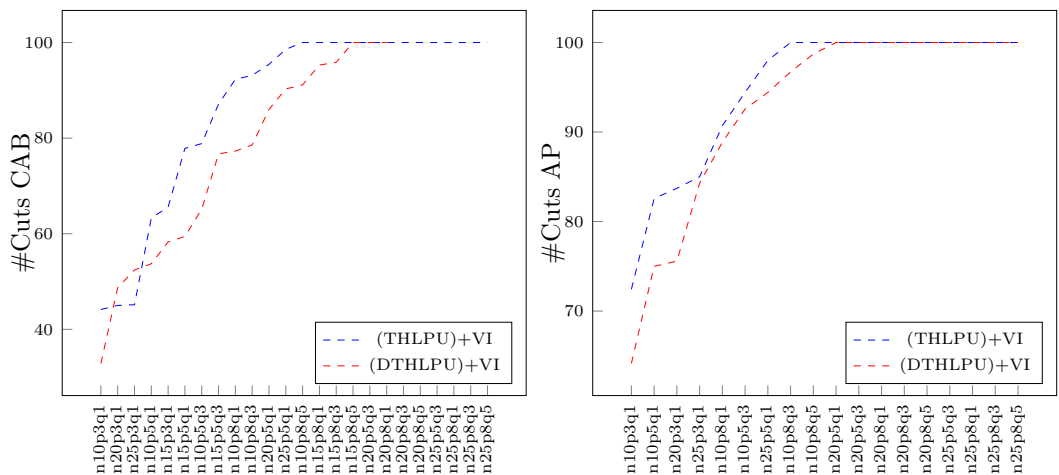

\section{Conclusions}\label{sec:7}

The Tree-of-Hubs Location Problem with Upgrading has not been, to the best of our knowledge,
previously studied. In this paper we introduce and analyze it, developing and tightening 
two Mixed Integer Programming formulations.
Several families of valid inequalities, some of them of reduced size and some others of exponential size, 
have been derived. The former can be fully incorporated to the formulation, and for the latter 
we devise separation procedures that allows one  to identify the most violated inequalities in the family.
We have computationally checked the improvements produced by the addition of this inequalities and present 
the first results for medium sized instances. 

Several different exact and heuristic procedures have been studied in the literature for the 
Tree-of-Hubs Location Problem, a simplified version of the problem we consider here. As a matter 
of future research, these approaches could also be tested on the THLPU. Also several extensions 
in different directions have been carried out that could be extended themselves by considering 
upgrading of nodes. In general, we consider single or multiple hub location problems with upgrading 
an interesting line of further study for authors in the hub location field. Also, different upgrading degrees may be consider for the nodes, each of them implying different reductions to the edge costs. The analysis of similar extended formulation for this problem would be the topic of future research.

\section*{Acknowledgments}
The first author was partially by the research projects MTM2016-74983-C2-1-R (MINECO, Spain) and PP2016-PIP06 (Universidad
de Granada) and the research group SEJ-534 (Junta de Andaluc\'ia). The second author  was partially 
supported by the research projects  MTM2015-65915-R (MINECO, Spain),  19320/PI/14 ({\sl Fundaci\'on S\'eneca}) and the project ``Cost-sensitive classification. A mathematical 
optimization approach'' ({\sl Fundaci\'on BBVA}).

%\section*{References}

\end{document}